\documentclass[reqno,12pt]{amsart}
\usepackage{amsmath}
\usepackage{amssymb}
\usepackage{amsthm}
\usepackage{graphicx}
\usepackage{color}
\usepackage{cleveref}
\usepackage{multirow}
\usepackage{bm}

\newtheorem{theorem}{Theorem}
\newtheorem{prop}{Proposition}
\newtheorem{corr}{Corollary}

\newtheorem{define}{Definition}

\theoremstyle{remark}
\newtheorem{remark}{Remark}

\newcommand{\C}{\mathbb{C}}

\DeclareMathOperator{\ran}{Range}

\newcommand{\w}{\omega}

\newcommand{\cn}{\textrm{cn}}

\author[K.~P.~Leisman]{Leisman, Katelyn Plaisier\textsuperscript{1}}\thanks{\textsuperscript{1}University of Illinois at Urbana-Champaign, \texttt{kleisman@illinois.edu}}
\author[J.~C.~Bronski]{Bronski, Jared C.\textsuperscript{2}}\thanks{\textsuperscript{2}University of Illinois at Urbana-Champaign, \texttt{bronski@illinois.edu}}
\author[M.~A.~Johnson]{Johnson, Mathew A.\textsuperscript{3}}\thanks{\textsuperscript{3}University of Kansas, \texttt{matjohn@ku.edu}}
\author[R.~Marangell]{Marangell, Robert\textsuperscript{4}}\thanks{\textsuperscript{4}University of Sydney,  \texttt{robert.marangell@sydney.edu.au}}

\title[Periodic NLS Stability ]{Stability of Traveling wave solutions of
  Nonlinear Dispersive equations of NLS type}
\begin{document}

\maketitle

\begin{abstract}
In this paper we present a rigorous modulational stability theory for periodic traveling wave solutions to equations of nonlinear Schr\"odinger type. We first argue that, for Hamiltonian dispersive equations with a non-singular symplectic form and $d$ conserved quantities (in addition to the Hamiltonian), one expects that generically ${\mathcal L}$, the linearization around a periodic traveling wave, will have a $2d$ dimensional generalized kernel, with a particular Jordan structure: The kernel $\ker({\mathcal L})$ is expected to be $d$ dimensional, the first generalized kernel $\ker({\mathcal L}^2)/\ker({\mathcal L})$ is expected to be $d$ dimensional, and there are expected to be no higher generalized kernels. The breakup of this $2d$ dimensional kernel under perturbations arising from a change in boundary conditions dictates the modulational stability or instability of the underlying periodic traveling wave. 

This general picture is worked out in detail for the case of equations of nonlinear Schr\"odinger (NLS) type. We give explicit genericity conditions that guarantee that the Jordan form is the generic one: these take the form of non-vanishing determinants of certain matrices whose entries can be expressed in terms of a finite number of moments of the traveling wave solution. Assuming that these genericity conditions are met we give a normal form for the small eigenvalues that result from the break-up of the generalized kernel, in the form of the eigenvalues of a quadratic matrix pencil. We compare these results to direct numerical simulation in a number of cases of interest: the cubic and quintic NLS equations, for focusing and defocusing nonlinearities, subject to both longitudinal and transverse perturbations. The stability of traveling waves of the cubic NLS (both focusing and defocusing) subject to longitudinal perturbations has been previously studied using the integrability:  in this case our results agree with those in the literature. All of the remaining cases appear to be new.  

\end{abstract}

\section{Introduction}

In this paper we consider the modulational stability of periodic solutions to equations of nonlinear Schr\"odinger (NLS) type
\begin{equation}
i w_t = w_{xx} + \zeta f(|w|^2)w
\label{gNLS}
\end{equation}
subject to perturbations in $L_2({\mathbb R})$ as well as stability of the $x$-periodic solutions to the two dimensional NLS equation
\begin{equation}
  i w_t = w_{xx} \pm w_{zz} + \zeta f(|w|^2)w
\label{2dgNLS}
\end{equation}
to perturbations in the transverse $z$ direction for an arbitrary well-behaved nonlinearity
$f(|w|^2)$. 
If one considers perturbations which are co-periodic then the operator found by linearization about a periodic traveling wave has compact resolvent, and stability is governed by the results of Grillakis-Shatah-Strauss \cite{GSSI,GSSII}. 
In most applications, however, the domains of interest will be unbounded, or at least many times the size of the fundamental period, and one is interested in understanding the effects of arbitrary perturbations. 
This is closely connected with Whitham \cite{Whitham,WhithamBook} modulation theory, where one tries to assess the asymptotic behavior of  long-wavelength perturbations. 
The first consideration of the question of the stability of periodic solutions to the cubic NLS equation seems to have been by Rowlands \cite{Rowlands} in the cubic case, who presented a formal perturbation theory calculation for the stability of the trivial-phase solutions to the cubic NLS---the cnoidal and dnoidal solutions. 
The NLS with cubic nonlinearity is, of course, completely integrable and there is a long history of exploiting the integrable structure in order to compute the spectrum of the linearized operator. 
This was pioneered by Alfimov, Its and Kulagin \cite{AlfimovItsKulagin} who gave a complete nonlinear description of the modulational instability---they show how to construct a time dependent solution to the focusing nonlinear Schr\"odinger equation that is asymptotic to a spatially periodic solution as $|t| \rightarrow \infty$. 
In other words they construct the whole homoclinic orbit to an unstable periodic traveling wave. While this was worked out in detail only for the dnoidal wave the construction applies more generally to any spatially periodic traveling wave, and in principle even to temporally periodic breather type
solutions. 
In a similar spirit Bottman, Deconinck and Nivala, and Deconinck and Segal \cite{BottmanDeconinckNivala,DeconinckSegal} used the exact integrability to give a very explicit description of the spectrum of the linearization of the cubic NLS around a periodic traveling wave using the well-known connection between the squared eigenfunctions of the Lax pair and the eigenfunctions of the linearized operator. 
In this vein we also mention a rigorous analysis of the periodic linearized spectrum of a ``trivial phase'' solution to (\ref{gNLS}) in the focusing cubic case by Gustafson, Le Coz and Tsai \cite{GustafsonLeCozTsai}, and the work of Gardner \cite{Gardner} on stability in the long-period (soliton)  limit. We recommend the text of Kamchatnov \cite{Kamchatnov} for more information on the modulations of nonlinear periodic traveling waves. 

While somewhat further afield it is worth mentioning the large
literature on the integrable modulation theory. Lax and Levermore \cite{LL1,LL2,LL3}, and
independently and from a somewhat different point of view Flashka,
Forest and McLaughlin \cite{FFM} used the integrability of the Korteweg-de Vries
equation to extend the formal Whitham theory to all time---among
other advantages the integrability allows one to continue the solution
past the point of wave-breaking. This analysis has since been extended
to many other completely integrable equations \cite{Grava,VenakidesDeift,JinLevermoreMcLaughlin,EM}. While we are considering a non-integrable evolution equation the analysis here has in many ways been shaped by the integrable
theory.

Here we take a rigorous modulation theory approach. We explicitly compute the
generalized kernel of the linearization about an arbitrary traveling
wave solution to (\ref{gNLS}) under periodic boundary conditions.
We know from Floquet theory that the spectrum of a periodic 
operator on $L_2({\mathbb R})$  is given by a union of the spectra
of a one-parameter family of operators, the parameter being the
Floquet exponent. Since we are able to compute the
generalized kernel of the linearized operator explicitly when the Floquet exponent is zero we can study the breakup of this generalized
kernel for small but non-zero quasi-momentum. We find a normal form
for the spectral curves in a neighborhood of the origin in the
spectral plane. This calculation is very much in the spirit of the
previously mentioned 
work of Rowlands \cite{Rowlands} for the cubic NLS, although the current
calculation applies most directly to the non-trivial phase case, while
Rowlands' calculation was done for the trivial phase solutions. Similarly in the case of transverse perturbations one can compute how the generalized kernel breaks up as $k$, the transverse wavenumber, varies. 
  
While the present calculation is similar in spirit to earlier calculations on equations
of Korteweg-de Vries (KdV) and nonlinear Klein-Gordon type \cite{BMR,BenzoniNobleRodriguesII, BenzoniNobleRodriguesI,BridgesFan,BronskiHur, BronskiJohnson,BronskiJohnsonKapitula,DH,GallayHaragus,HSS,Haragus, Johnson,JNPZ,Johnson2010,JohnsonZumbrunBronski,JonesMarangellMIllerPlazaI,JonesMarangellMIllerPlazaII,MarangellMiller,Serre} we feel that the present calculation makes clear a number of issues that were
not clear in previous calculations regarding the relationship between
the Hamiltonian structure of the problem and the nature of the
spectrum in a neighborhood of the origin in the spectral plane. While
the KdV equation has a Hamiltonian structure the fact that the
symplectic form, $\partial_x$, has a kernel gives the KdV stability
problem some features that are not typical of Hamiltonian stability
problems where the symplectic form is non-singular.

In this paper we will often consider the Jordan structure of the
kernel of an operator ${\mathcal L}$ of Hamiltonian form. In the usual
way the kernel will be defined to be 
$\ker({\mathcal L})=\ker_0({\mathcal L})= \{u~|~{\mathcal L} u = 0 \}.$
Analogously the $k^{th}$ generalized kernel will be defined to be
$\ker_k({\mathcal L})=\ker({\mathcal L}^{k+1}) / \ker({\mathcal L}^{k})$, in other words,
the vectors at position $k$ in a Jordan chain. Typically, as we will
see, the linearization has only a kernel and a $1^{st}$ generalized
kernel, and no higher generalized kernels, and these two are dual
in the Hamiltonian sense---the kernel is associated with the angle
variables, while the first generalized kernel is associated with the
action variables. Throughout the paper when we refer to the kernel or the kernel proper we are referring to $\ker_0({\mathcal L}).$ When we refer to the generalized kernel with no further designation we mean $\cup_k \ker_k({\mathcal L})$, the set of vectors for which ${\mathcal L}^k {\bf u} = 0$ for some value of $k$.   

We briefly remark on some notational conventions to be followed in the remainder of the paper. Operators will normally be represented by upper-case calligraphic Latin letters, such as $\mathcal{L}$ and matrices by upper-case bold letters such as $\bf M$, but we will occasionally deviate from this in the interest of clarity.  Lower-case bold letters will always denote vectors. Since the operators in question will typically be non-self-adjoint it is convenient to use both left and right (generalized) eigenvectors, which will be denoted by ${\bf v}_k$ and ${\bf u}_l$ respectively, with the inner product of the two represented as ${\bf v}_k {\bf u}_l.$

\subsection{Hamiltonian flows with symmetry, and the Jordan structure of the linearized operator.}

We consider the stability of quasi-periodic traveling
wave solutions to equations of NLS type. 
The calculation presented here is rather involved, so we begin by
giving an overview of the motivating ideas.  The linearized operator will be denoted by ${\mathcal L}$,
or by ${\mathcal J} \partial^2 {\tilde H}$ if we wish to emphasize the
Hamiltonian nature. As the linearized operator $\mathcal{L} $ is
non-self-adjoint it will be convenient to consider both the left and
the right eigenvalues. The right eigenvectors will be denoted ${\bf u}$
and the left eigenvectors will be denoted ${\bf v}$:
\begin{align*}
  &{\mathcal L} {\bf u} = \lambda {\bf u} & \\
  & {\bf v} {\mathcal L} = \lambda {\bf v}. & 
\end{align*}
Note that because the linearized operator 
${\mathcal L}={\mathcal J} \delta^2 {\tilde H}$ takes the Hamiltonian form the
symplectic form ${\mathcal J}$ maps between the left and the right
eigenspaces. If ${\bf u}$ is a right eigenvector of 
 ${\mathcal  L}={\mathcal J} \delta^2 {\tilde H}$, with eigenvalue $\lambda$
 then ${\bf v} = {\bf u}^t {\mathcal J}$ is a left eigenvector, with eigenvalue
 $-\lambda$. (Here we are assuming that $\mathcal{J}$ is in the standard form
 ${\mathcal J}^t = -{\mathcal J}, {\mathcal J}^t {\mathcal J} = {\mathcal I}$.) 
 
We consider an abstract Hamiltonian system
\begin{equation}
  \frac{\partial w}{\partial t} = {\mathcal J} \delta H(w)
  \label{eqn:Ham}
\end{equation}
where $\delta H$ represents the first variation of the Hamiltonian
functional (in other words $\delta$ is the Euler-Lagrange operator)
and the skew-symmetric operator ${\mathcal J}$ represents the
symplectic structure. We assume that ${\mathcal J}$ is
invertible, as is the case for equations of nonlinear Schr\"odinger
(NLS) type.

We further assume that there are $d$ additional
functionals (the conserved quantities, or momenta) $\{P_j(w)\}_{j=0}^{d-1}$  
that 
pairwise Poisson commute - in other words that
\begin{align}
  &\langle\delta P_j, {\mathcal J} \delta H \rangle = 0,&  \label{eqn:Com1} \\
  &\langle\delta P_j, {\mathcal J} \delta P_k \rangle = 0,\qquad j \neq k.& \label{eqn:Com2}  
\end{align}
From here on for brevity we have suppressed the dependence on $w$.  Note that the conserved quantities $P_j$ generate $d$  commuting
flows, each with its own phase $s_j$: in other words the evolution equations 
\begin{equation}
 \frac{\partial w}{\partial s_j} = {\mathcal J} \delta P_j 
\label{eqn:ComFlows}
\end{equation}
 all commute, so given a function  $w$ one can generate
 $w(s_0,s_1,\ldots s_{d-1})$  depending on $d$ parameters $\{s_j\}_{j=0}^{d-1}$ representing a rotation through angle $s_j$ in the $j^{th}$ phase. For the generalized NLS, for example, there are two
 additional conserved quantities (beyond the Hamiltonian), the mass and
 the momentum
 \begin{align*}
   & P_0 = \int |w|^2 ~dx,& \\
  &  P_1 = \frac{i}{2} \int (w_x^*w - w_x w^*)~dx.& 
 \end{align*}
 The flows associated with these conserved quantities are given by
 \begin{align*}
   & w_{s_0} = i w =\mathcal J \delta P_0,& \\
   & w_{s_1} = w_x =\mathcal J \delta P_1,& 
 \end{align*}
 and these generated the two symmetries of the generalized NLS,
 translation invariance and the $U(1)$ symmetry
 \begin{equation}
   w(x) \mapsto w(x + s_1) e^{i s_0}.
 \end{equation}
 
The traveling wave solutions to Equations (\ref{eqn:Ham} --
\ref{eqn:Com2}) are given by critical points of the reduced Hamiltonian
\begin{equation}
  {\mathcal J} \delta\left(H  - \sum_{j=0}^{d-1} c_j P_j\right)= 0.
\label{eqn:ConCrit}
\end{equation}
From an energetic point of view the traveling waves represent
constrained critical points---the traveling waves are critical points
of the Hamiltonian $H$  subject to the constraints that conserved
quantities $P_j$ are fixed, with the constants $c_j$  as Lagrange multipliers which enforce the momentum
constraints. A slightly more physical interpretation of  the $c_j$ is as angular frequencies that dictate the speed in the $j^{th}$ phase: if one takes a solution to Equation  (\ref{eqn:ConCrit}) and allows it to flow under the
original Hamiltonian flow
\begin{align*}
  w_t &= {\mathcal J} \delta H &\\
      &= {\mathcal J} \sum c_j \delta P_j, & 
\end{align*}
a solution to (\ref{eqn:ConCrit}) rotates with
angular frequency  $c_j$ in the $j^{th}$ phase. Note that since the vector fields Poisson commute the flow generated by the sum is simply the composition of the individual flows. 

The stability of traveling wave solutions requires an  
understanding of the linearized operator, which takes the
Hamiltonian form: the product of a skew-adjoint operator (the
symplectic form ${\mathcal J}$) with a Hermitian operator (the second variation
of the reduced Hamiltonian) 
\begin{equation}
 {\mathcal L} =  {\mathcal J} \delta^2\left(H - \sum c_j P_j\right)  = {\mathcal J} \delta^2 \tilde H.
  \label{eqn:SecVar}
\end{equation}
The main observation is that the structure of the generalized kernel
of the second variation (\ref{eqn:SecVar}) reflects the Hamiltonian
structure of the problem. In particular one expects (in a generic
situation) that the generalized kernel of the operator ${\mathcal J}\delta^2 \tilde H$ has the following structure: there are $d$ elements of the kernel 
$\{{\bf u}_{2j}\}_{j=0}^{d-1}\in \ker_0\left({\mathcal J}  \delta^2 \tilde H\right)$, 
and $d$ elements of the first generalized kernel 
$\{{\bf u}_{2j+1}\}_{j=0}^{d-1}\in\ker_1\left({\mathcal J}  \delta^2 \tilde H\right)$, satisfying
\begin{align*}
  & {\mathcal J}\delta^2 \tilde H {\bf u}_{2j} = 0 &\\
  & {\mathcal J}\delta^2 \tilde H {\bf u}_{2j+1} = {\bf u}_{2j}.&
\end{align*}
In other words the Jordan structure of the kernel is (generically)
expected to be a direct sum of $d$ copies of $2\times 2$ Jordan blocks. In the next
section we will explicitly give the genericity conditions for equations of NLS
type. The basis vectors for the kernel $\{{\bf u}_{2j}\}_{j=0}^{d-1}$ and first
generalized kernel $\{{\bf u}_{2j+1}\}_{j=0}^{d-1}$ have natural
interpretations: the null-vectors are associated with the positions or
phases $s_j$ and the generalized null-vectors 
are
associated with the dual variables, the
corresponding wavespeeds $c_j$.

To see this recall that (for fixed values of
the momenta $\{P_j\}_{j=0}^{d-1}$) any solution $w$ to the traveling wave
equation (\ref{eqn:ConCrit}) immediately generalizes to a $d$ parameter
solution $w(s_0,s_1,\ldots,s_{d-1})$ by the flows in (\ref{eqn:ComFlows}).  
Given such a $d$  parameter family of
solutions to (\ref{eqn:ConCrit}) (parameterized by $\{s_j\}_{j=0}^{d-1}$) we can
generate elements in the kernel of (\ref{eqn:SecVar}) by
differentiating (\ref{eqn:ConCrit}) with respect to $s_k$
\begin{equation}
  {\mathcal J} \delta^2\left(H -\sum_{j=0}^{d-1} c_j P_j \right) w_{s_k}= 0.
  \label{eqn:Kernel}
\end{equation}
Therefore the  derivative of a solution $w$ with respect to the
$k^{th}$ phase  $s_k$   lies in the (right) kernel of
the linearized operator: 
${\bf u}_{2j} = \frac{\partial w}{\partial s_j}$.  
From the form of the linearized operator ${\mathcal J} \delta^2 \tilde H$, we see that the left null-vectors are
given by ${\bf v}_{2j+1}=  {\bf u}_{2j}^t {\mathcal J}$. Since the solutions satisfy
${\bf u}_{2j} = w_{s_j} = {\mathcal J} \delta P_j$ we have that the the elements of the left kernel
are given by ${\bf v}_{2j+1} = \delta P_j^t$   

We are assuming here that the gradients
$\delta P_j,$ evaluated on the traveling wave, are linearly
independent so that ${\bf v}_{2j+1}$ and thus (through the invertibility
of ${\mathcal J}$) ${\bf u}_{2j}$ are as well. Generically one expects this, however it will not always
be true. For integrable Hamiltonian flows, for instance, there are an
infinite number of conserved quantities but when one evaluates the
gradients on the traveling wave solutions only a finite number of
these will be linearly independent. 

There is a dual (in the Hamiltonian sense) set of vectors that lie in
the first generalized kernel of (\ref{eqn:SecVar}). These are found by
differentiating (\ref{eqn:ConCrit}) with respect to the Lagrange
multipliers $c_j$
\begin{equation}
  {\mathcal J} \delta^2\left(H - \sum_{j=0}^{d-1} c_j P_j\right) w_{c_k}
    = {\mathcal J} \delta P_k = \frac{\partial w}{\partial s_k} = {\bf u}_{2k}. 
  \label{eqn:GenNullSpace}
\end{equation}
Thus the elements of the first generalized kernel can be written as ${\bf u}_{2j+1}=w_{c_j}$, and those of the left first generalized kernel as ${\bf v}_{2j}=-{\bf u}_{2j+1}^t\mathcal J$ (this sign convention is slightly more convenient for later calculations). 

It is worth noting that this picture assumes that the symplectic form ${\mathcal J}$ is non-singular. If ${\mathcal J}$ has a kernel then one can have a Casimir---a conserved quantity whose functional gradient lies in the kernel of the symplectic form ${\mathcal J}$. In this case differentiation with respect to the corresponding Lagrange multiplier generates an element of the kernel proper. This occurs for equations of KdV and Boussinesq type, where the symplectic form is a derivative and there is a conserved mass Casimir whose functional gradient is constant. 

A couple of things to note here. We have the relations
\begin{align*}
  & {\mathcal L} {\bf u}_{2j} = 0 \qquad {\mathcal L} {\bf u}_{2j+1} = {\bf u}_{2j} & \\
   & {\bf v}_{2j+1} {\mathcal L} = 0 \qquad {\bf v}_{2j} {\mathcal L}  = {\bf v}_{2j+1} &
\end{align*}
and thus 
\begin{align*}
  & {\bf v}_{2j+1} {\bf u}_{2k} = {\bf v}_{2j} {\mathcal L} {\bf u}_{2k} = 0& \\
  & {\bf v}_{2j+1} {\bf u}_{2k+1} = {\bf v}_{2j} {\mathcal L} {\bf u}_{2k+1} = {\bf v}_{2j} {\bf u}_{2k}.& 
\end{align*}

Later in the paper, we will need to compute inner products between elements of the left and right kernel, which are given by
\begin{equation}
  {\bf v}_{2j} {\bf u}_{2k} = \delta P_j^t w_{c_k} = \frac{\partial P_j}{\partial c_k}.
  \label{eqn:Maxwell1}
\end{equation}
Thus the Gram matrix made up of these inner products can be expressed in terms of derivatives of
the conserved quantities with respect to the Lagrange multipliers. 
We also have a formula for the values of the conserved quantities in
terms of the reduced Hamiltonian. Considering the reduced Hamiltonian
$\tilde H = H - \sum c_j P_j$ as a function of the Lagrange
multipliers $c_j$ on the traveling wave solution we have 
\begin{equation}
  \frac{\partial \tilde H}{\partial c_j } = (\delta \tilde H)\cdot
  \frac{\partial w}{\partial c_j} - P_j = -P_j
  \label{eqn:Maxwell2}
\end{equation}
where the second equality follows from the fact that the traveling wave
is a critical point of $\tilde H$, and thus $\delta \tilde H=0$. 
In the context of thermodynamics the relations
(\ref{eqn:Maxwell1}) and (\ref{eqn:Maxwell2}) are generally
known as the Maxwell relations. Given this we have the identity 
\[
\frac{\partial P_k}{\partial c_j} 
= -\frac{\partial^2 \tilde H}{\partial c_j\partial c_k } 
= \frac{\partial P_j}{\partial c_k} 
\]
and so the Gram matrix of the generalized null-space is completely expressible in terms of the
second variation of the reduced Hamiltonian $\tilde H$ with respect to the wavespeeds $c_j.$

The Jordan structure of the linearized operator follows from the
Hamiltonian structure of the problem. While no assumption is being made on the
Hamiltonian $H$ (in particular we don't assume that it is integrable)
the traveling waves define an invariant $2d$  dimensional manifold
sitting inside the phase space, and the flow on this submanifold
is integrable.  An integrable Hamiltonian system can
always be written in terms of action and  angle variables
$\{\Gamma_j,\Theta_j\}_{j=0}^{d-1}$ such that the Hamilton's equations become  
\begin{eqnarray*}
  & \tilde H = \tilde H(\Gamma_0,\Gamma_1,\ldots, \Gamma_{d-1}), & \\
  & \frac{d\Theta_j}{dt} = \frac{\partial H}{\partial \Gamma_j}, & \\
  &  \frac{d \Gamma_j}{dt} = -\frac{\partial H}{\partial \Theta_j}=0. &
\end{eqnarray*}
In such a case the skew-symmetric form acting on the
second variation of the Hamiltonian,  ${\mathcal J} \delta^2 \tilde H$ will 
contain a $2d \times 2d$ block of the form
\[
  \left(\begin{array}{cc} 0 & \frac{\partial^2 \tilde H}{\partial \Gamma^2} \\ 0 & 0 \end{array}\right)
\]
which is clearly a Jordan block.  
Therefore we expect that in general when we compute the quantity
${\mathcal J} \delta^2 \tilde H ={\mathcal L}$ about a point that lies on the invariant $2d$ dimensional
submanifold we will have a generalized kernel of dimension $2d$, with
$\dim(\ker(\mathcal{L}))=d$ and $\dim(\ker({\mathcal L}^2)/\ker({\mathcal L}))=d$. 
There are assumptions here that must
be checked in practice. We are assuming that $\delta P_j$ are
linearly independent (or equivalently that $\partial^2_{\Gamma_j \Gamma_k} \tilde
H$ has full rank) and that there is nothing in the generalized
kernel other than $\{w_{s_j}\}, \{w_{c_j}\}$, both of which can
fail, but this is the generic picture that one should expect.

This suggests the following strategy for understanding the spectrum of
the linearized operator about a traveling wave, at least in a
neighborhood of the origin in the spectral plane. Given a reasonably
complete description of the generalized kernel of the linearized
operator one can compute how this kernel breaks up under
long wavelength perturbations. There are at least two distinct situations in which one
might want to do this, namely assessing the impact of longitudinal and
transverse perturbations on stability.

In the longitudinal case one would like to assess the impact of
subharmonic perturbations, those with a period which is a (large)
multiple of the period of the underlying traveling wave. In general if ${\mathcal L}$ is a linear periodic differential operator then the eigenfunctions $u(x;\mu)$ admit a   
Floquet decomposition $u(x;\mu) = e^{i \mu x} \hat u(x;\mu)$, where $\hat u(x;\mu)$ has a
period equal to that of the underlying traveling wave, and $\mu$
represents the quasi-momentum. This leads to a one-parameter
family of eigenvalue problems, parameterized by $\mu$, with $\mu=0$
representing the periodic eigenvalue problem. We have argued that the linearized periodic operator will generally have a $2d$ dimensional generalized kernel. One is then led to
the question of understanding how this generalized
kernel breaks up into $2d$ small eigenvalues for non-zero Floquet multiplier $\mu$. 

In the transverse case one might be interested in assessing the impact
of long-wavelength perturbations in a transverse direction. For
instance one might be interested in understanding the stability
of one-dimensional traveling waves to long wavelength perturbations in a transverse
direction. Here we have a similar situation, with the small parameter $k$ representing the wave-number of the transverse perturbation. The perturbations are somewhat simpler than those induced by the Floquet multiplier, but they are of a similar form.  

In the next section we present a detailed calculation for the periodic
traveling wave solutions to a generalized nonlinear Schr\"odinger
(NLS) equation. We note a slight change in notation between this
section and the last: in the present section we think of the period
and the Floquet multiplier of the underlying solution as being
fixed. In the next section it will be more convenient to consider the
period, etc., as being functions of underlying constants in the
problem. There are a couple of reasons for doing so. Firstly such
parameters arise naturally in the process of reducing the traveling
wave to quadrature, and the traveling wave solution satisfies a number
of identities with respect to these parameters (the Picard-Fuchs
relations). Further the genericity conditions on the Jordan form of
the kernel of the linearized operator can be expressed most naturally
in terms of the variation of the period and the Floquet multiplier in
terms of these parameters.

Finally we use the following notation: throughout this calculation we
will have numerous occasions to consider the linear independence of
various gradients of the period $T$, the Floquet exponent $\eta$, and
the conserved quantities. We will use the notation $\{A,B\}_{a,b}$
to denote the quantity 
$\{A,B\}_{a,b} = \frac{\partial A}{\partial a}\frac{\partial B}{\partial b}
- \frac{\partial A}{\partial b}\frac{\partial B}{\partial a}$, and similarly
\[
\def\arraystretch{1.6}
  \{A,B,C\}_{a,b,c} = \left|\begin{array}{ccc}
    \frac{\partial A}{\partial a} & \frac{\partial A}{\partial b} & \frac{\partial A}{\partial c} \\ 
    \frac{\partial B}{\partial a} & \frac{\partial B}{\partial b} & \frac{\partial B}{\partial c} \\
    \frac{\partial C}{\partial a} & \frac{\partial C}{\partial b} & \frac{\partial C}{\partial c} \\ 
  \end{array}\right|.
\]

\subsection{NLS Traveling Waves.}

In this section we consider the problem of applying the abstract considerations 
of the previous section to the nonlinear Schr\"odinger (NLS) equation 
\[
  i w_t = w_{xx} + \zeta f(|w|^2) w. 
\]

This is an infinite dimensional Hamiltonian flow, and can
be written in the form
\[
  w_t = {\mathcal J} \delta H
\]
where $H$ is the Hamiltonian functional with density $h$
\[
 H =  \int w_{x}^2 - \zeta F(|w|^2) dx = \int h(w,w^*,w_x,w_x^*) dx, 
\]
$\delta H$ is the variational derivative of the functional $H$:  
$\delta H=\frac{\partial h}{\partial w^*} - \frac{\partial }{\partial x} \frac{\partial h }{\partial w_x^*},$ $F$ is the antiderivative of the nonlinearity $F'=f,$ and ${\mathcal J}=i$. As discussed in the previous section for a general nonlinearity the NLS equation has two
conserved quantities, the mass and the momentum, given by 
\begin{align*} 
&M = \int |w|^2 dx, & \\
&P = \frac{i}{2} \int (w_x^* w - w_x w^*) dx. & 
 \end{align*}  

If one looks for a traveling wave solution in the form
\[w(x,t) = e^{i\w t}\phi(x+ct),\]
where (writing $y=x+ct$)
\begin{equation} \phi(y) = \exp(i (\theta_0+S(y+y_0)+cy/2))A(y+y_0),
\label{eqn:twphi}
\end{equation}
%
then the functions $A(y+y_0),\ S(y+y_0)$ satisfy
\begin{align}
  & 2 A_y  S_y + A  S_{yy} = 0, & \label{eqn:TravWave1}\\
  & S_y = \frac{\kappa}{A^2}, & \label{eqn:TravWave2}\\
  & A_{yy}= -(\omega+c^2/4) A + A  S_y^2 - \zeta f(A^2) A, &\label{eqn:TravWave3}\\
  & A_y^2 = 2 E - (\omega+c^2/4) A^2 - \frac{\kappa^2}{A^2} - \zeta F(A^2),& \label{eqn:TravWave4}
\end{align}
where Equation (\ref{eqn:TravWave4}) follows from (\ref{eqn:TravWave3}) by integration. 
There are seven different parameters $(\omega,E,\kappa,\zeta,c,y_0,\theta_0)$ defining a general periodic traveling wave, so it is worth saying something about the interpretations of them. First, we note that we can rearrange this solution as
\[w(x,t) = e^{i(\w +c^2/4)t}e^{i\theta_0}e^{iS(x+ct+y_0)}A(x+ct+y_0)e^{ic(x+ct/2)/2}.
\]
By Galilean invariance, we can reduce this to 
\[w(x,t) = e^{i(\w+c^2/4)t}e^{i\theta_0}e^{iS(x+y_0)}A(x+y_0),
\]
and by translation and $U(1)$ invariance, we can further reduce to
\[w(x,t) = e^{i(\w+c^2/4)t}e^{iS(x)}A(x).
\]
The parameters $y_0,\ \theta_0$ are the phases or angles associated with the translation invariance and $U(1)$ invariance, which are generated by
the conserved quantities $P$ and $M$ respectively. Derivatives
with respect to these variables will generate elements of the kernel of the linear
operator. We will show shortly that generically these derivatives span the kernel. 
Additionally, the parameters $c$ and $ \omega$ are the Lagrange multipliers 
associated with the constant momentum and mass
constraints. Derivatives with respect to these parameters will generically 
generate the $1^{st}$ generalized kernel of the linearized
operator, with explicit genericity conditions to be described later. 
Because of the translation invariance and $U(1)$ invariance, we can eventually omit the dependence on $\theta_0$ and $y_0$: after initial differentiation of $\phi(y)$ with respect to them, we will set $\theta_0=
y_0=0$. 
Because of the Galilean invariance, after differentiating $\phi(y)$ with respect to $c$, we can simplify all calculations by setting $c=0$. 
From these symmetries, we can also say that $A$ and $S$ depend only on $y+y_0,\ \omega,\ E,\ \kappa,\ \zeta$; they have no direct dependence on $\theta_0$ or $c$, and dependence on $y_0$ is equivalent to dependence on $y$.
The parameters $E,\kappa$ are associated with the boundary
conditions: generally the solutions to equations (\ref{eqn:TravWave1})
 ---(\ref{eqn:TravWave4}) satisfy quasi-periodic boundary conditions
  of the form 
\begin{align*}
  & w(T) = w(0) e^{i \eta}&  \\
  & w_x(T) = w_x(0) e^{i \eta} &   
\end{align*} 
where the period $T$ and the quasi-momentum $\eta$ depend on
$\omega,E,\kappa,\zeta$, and $\eta$ also depends on $c$. The discussion of the Hamiltonian
structure in the preceding section assumes a fixed function space, so a fixed
period $T$ and quasi-momentum $\eta$. In principle the
conditions $T,\eta$ fixed defines $E,\kappa$ implicitly in terms of
$c,\omega$. In this section we will treat $E,\ \kappa,\ \w$ as independent, as the
genericity conditions are most naturally expressed in terms of partial
derivatives of the conserved quantities and the period and
quasi-momentum in terms of the various parameters. The price that we
pay for this is that the elements of the first generalized kernel are
given by somewhat tedious directional derivatives in
$E,\kappa,\omega$.   Finally $\zeta$ is a somewhat artificial
parameter that we introduce in order to compute the preimage of
certain elements of the range of the linearized operator. We remark
that in the case of a power-law type nonlinearity, $f(A^2) = \pm
A^{2m}$, one can do these calculations using the inifinitesimal
generator of the scaling group, but for a more complicated
nonlinearity this additional parameter is more convenient.


We define the domain $\Omega$ to be the open set of parameter values for
which equation (\ref{eqn:TravWave4}) has a non-degenerate traveling
wave.
\begin{define}
  We define the parameter domain $\Omega$ to be the open set of parameter
  values $(E,\kappa, \omega,\zeta)$ such that
  \begin{itemize}
    \item The quantity $\kappa > 0$.
    \item The function 
      $R(A)=2 E - \omega  A^2 - \frac{\kappa^2}{A^2} - \zeta F(A^2)$ 
      is real and positive on some interval $(a_-,a_+)$ on the positive real axis, 
      and has simple roots at $a_\pm$.
  \end{itemize}
\end{define}

\begin{remark}
The set $\Omega$ represents the set of parameter values where the
generalized NLS has non-trivial phase solutions. The boundary sets
$\kappa=0$ and the set of parameter values for which $R(A)$ has real
roots of higher multiplicity (the zero set of the discriminant)
represent various degenerations including
the trivial phase periodic solutions, solitary wave solutions,
constant amplitude (Stokes wave) type solutions, etc. In order to be
able to differentiate with respect to parameters it is easiest to
assume that we are on the open set. In principle the stability of the
limiting solutions could be determined by taking an appropriate limit,
but in practice it is probably easier to repeat the calculation for
the subclass of solutions. Also note that when $\kappa>0$ there is a
neighborhood of the origin that is in the classically forbidden region
(in other words $R(A)$ is negative) and thus in this situation it
suffices to look at $A$ strictly positive. It is only in the
$\kappa=0$ case that one can have solutions (for instance the cnoidal
solutions) that pass through zero. This makes $\kappa=0$ a somewhat
singular perturbation of $\kappa\neq 0.$ 
\end{remark}

It is useful to introduce the function
\begin{equation}
{K} = \int_0^T (A_y)^2dy = \oint_\gamma  \sqrt{2 E - \omega A^2 - \frac{\kappa^2}{A^2} -
  \zeta F(A^2)} dA\label{eqn:bigKint}
\end{equation}
where $\gamma$ is any simple closed contour in the complex plane
encircling an interval $(a_-,a_+)$ on the real axis on which the function 
$R(A) =  2 E - \omega A^2 - \frac{\kappa^2}{A^2} - \zeta F(A^2)$ 
is positive with simple zeroes at $a_-$ and $a_+.$ If
this is the case then there is a periodic solution to 
$ A_y^2 = 2 E - \omega A^2 - \frac{\kappa^2}{A^2} - \zeta F(A^2)$ 
with period $T$  with a minimum value $A=a_-$ and a
maximum value $A=a_+$. We will always choose the translate of the traveling wave profile so that $A(0)=a_-$ and $A(T/2)=a_+$. Further the spatial period $T$, the
mass $M$ and the phase $S(T)-S(0)$ satisfy
\begin{align*}
  & \frac{\partial { K}}{\partial E} = T & \\
  &\frac{\partial { K}}{\partial \omega} = -\frac{1}{2} M  & \\
  & \frac{\partial {K}}{\partial \kappa} = S(0)-S(T) & 
\end{align*}
The function $w(y,t)$ is itself not periodic but rather
quasi-periodic. It satisfies the boundary conditions 
\begin{align*}
  & w(T) = \exp(i (S(T)-S(0))) w(0) & \\
  & w_y(T) = \exp(i (S(T)-S(0))) w_y(0) & 
\end{align*}
where the accumulated phase $\eta=S(T)-S(0)$ can be interpreted as the
quasi-momentum of the traveling wave.

\subsection{Linearization and the Jordan structure of the kernel}

We next consider the linearized equations of motion by considering a
perturbation of the form
\[
w(y,t) = e^{i\w t}(\phi(y)+\epsilon e^{i[S(y+y_0)+cy/2+\theta_0]} W(y,t))
\]
where  $\phi(y)$ is as in Equation \eqref{eqn:twphi}, and $W(y,t)$ can be written in the form $R(y,t) + i I(y,t)$. Note
that by ``factoring off'' the phase 
$\exp(i(S(y+y_0)+cy/2+\theta_0))$ 
from 
the perturbation we have effectively
removed the quasi-momentum. 
In prior sections we expressed the
Hamiltonian structure in terms of a single complex field, with
symplectic form ${\mathcal J} = i$. In the literature on the stability of
traveling wave solutions to NLS it is usual to work instead with
the real and imaginary parts of the field. We will follow that
convention here. In this case the symplectic form is given by
$\left(\begin{array}{cc} 0 & -1 \\ 1 & 0\end{array}\right)$ which, 
with a slight abuse of notation, we will continue to denote ${\mathcal J}.$ 

In terms of the real and imaginary parts $R,I$ 
the linearized equations of motion can be written as
\begin{align}
\nonumber  {\mathcal J} \delta^2 \tilde H\left(\begin{array}{c}R \\ I \end{array}\right) = \left(\begin{array}{c}R \\ I \end{array}\right)_t 
    & = \left(\begin{array}{cc}{\mathcal K} & -{\mathcal L}_- \\ {\mathcal L}_+ & {\mathcal K} \end{array}\right)
    \left(\begin{array}{c}R \\ I \end{array}\right) &\\ 
  & = \left(\begin{array}{cc} 0 & -1 \\ 1 & 0 \end{array}\right)
    \left(\begin{array}{cc}L_+ & {\mathcal K} \\ {\mathcal K}^T   & L_- \end{array}\right)
    \left(\begin{array}{c}R \\ I \end{array}\right) & \label{eqn:LinOp1}
\end{align}
where the operator 
\begin{equation}
  {\mathcal K} = S_{yy} + 2S_y \partial_y
  \label{eqn:LinOp2}
\end{equation}
is skew-symmetric (when considered as acting on a space of $T$-periodic functions), while the operators 
\begin{align}
  & {\mathcal L}_+ = -\w - \partial_{yy} +
     (S_y)^2 - \zeta f(A^2) - 2 \zeta f'(A^2) A^2 & \\
  & {\mathcal L}_- = -\w -\partial_{yy} + 
     (S_y)^2 - \zeta f(A^2) &  \label{eqn:LinOp3}
\end{align}
are self-adjoint. This implies that the linearized operator can be written in the
Hamiltonian form as the product of a skew-symmetric and a
symmetric operator, and in particular as  in Equation \eqref{eqn:LinOp1}.  Note that in the trivial phase case $\kappa=0$, as in
the solitary wave case, the linearized operator
${\mathcal L} = {\mathcal J}\delta^2 \tilde H$ has only the
off-diagonal terms, ($\delta^2 \tilde H$ is diagonal) but
in the non-trivial phase case both operators are full.   

The basic idea of the next calculation is the following: we have constructed a traveling wave solution that depends on a number of
parameters, $\kappa,E,\omega,c,\zeta,\theta_0,y_0$. If the traveling wave equation does not depend explicitly on the parameter, then differentiation with respect to that parameter 
will give an element of the {\em formal} kernel of the
differential operator (e.g., ${\mathcal J} \delta^2 \tilde H {\bf u}_{E } = 0$; this is the case for $y_0,\ \theta_0, E,\ \kappa$). Otherwise, differentiation with respect to the parameter gives an element of the range of the differential operator (e.g., ${\mathcal J} \delta^2 \tilde H {\bf u}_{\omega} =- {\mathcal J}\delta\frac{\partial \tilde H}{\partial \omega}=\mathcal J \delta M={\bf u}_{\theta_0}$; this is the case for $c$, $\w$, $\zeta$).  
We note that the above description is somewhat impressionistic, due to the form of our perturbation. 
Typically functions generated by differentiating with respect to parameters will not belong to the proper function space, as the period and quasi-momentum will be functions of the parameters, but appropriate linear combinations {\em will} lie in the proper function space. We are computing elements of the tangent space to the manifold of solutions of fixed period and quasi-momentum.  
 
The results of this calculation can be summarized in the
following theorem.

\begin{theorem}\label{thm:NullSpace}
Suppose
\begin{itemize}
  \item The parameter values $(E,\kappa, \omega, \zeta ) $  lie in the
    open set $\Omega.$
  \item The nonlinearity $F(x)$ is linearly independent from $1,x,\frac{1}{x}$
  \item The amplitude $A(y)$ is non-constant. 
\end{itemize}

Define the linearized operator ${\mathcal J} \delta^2 \tilde H$ as in
(\ref{eqn:LinOp1}--\ref{eqn:LinOp3}), subject to $T$ periodic boundary
conditions, and the quantities
\begin{align*}
  & T(\kappa,E,\omega,\zeta) = \oint \frac{dA}{\sqrt{2 E - \w A^2 - \frac{\kappa^2}{A^2} - \zeta F(A^2)}}, & \\
  & \eta(\kappa,E,\omega,\zeta) =  \oint
    \frac{\kappa dA}{A^2\sqrt{2 E - \w A^2 - \frac{\kappa^2}{A^2} - \zeta F(A^2)}}, & \\
  & M(\kappa,E,\omega,\zeta) = \oint \frac{A^2 dA}{\sqrt{2 E - \w A^2 - \frac{\kappa^2}{A^2} - \zeta F(A^2)}}, & \\
  & K(\kappa,E,\omega,\zeta) = \oint \sqrt{2 E - \w A^2 - \frac{\kappa^2}{A^2} - \zeta F(A^2)} dA, &
\end{align*}
with the integral taken around an appropriate contour in the complex plane.  
Then the generalized kernel of ${\mathcal L}={\mathcal J} \delta^2 \tilde H$ generically takes
the form of the direct sum of two Jordan chains of length two: there exist two null
vectors ${\bf u}_{0,2}$ such that ${\mathcal J} \delta^2 \tilde H {\bf u}_{0,2} =
0$, and two  generalized null-vectors ${\bf u}_{1,3}$ such that  
${\mathcal J} \delta^2 \tilde H {\bf u}_{1,3} = {\bf u}_{0,2} $, with ${\bf u}_j$ given as follows:
\begin{align*}
  &{\bf u}_0 = \{T,\eta\}_{E,\kappa} \left(\begin{array}{c} 0 \\  A(y) \end{array}\right) &\\
  &{\bf u}_1 = \{T,\eta\}_{E,\kappa} \left(\begin{array}{c} A_\omega(y) \\ A(y) S_\omega(y)\end{array}\right) 
    + \{T,\eta\}_{\kappa,\omega} \left(\begin{array}{c} A_E(y) \\ 
        A(y) S_E(y)\end{array}\right) \\
  &\hskip 170pt + \{T,\eta\}_{\omega,E} \left(\begin{array}{c} A_\kappa(y) \\ A(y) S_\kappa(y)\end{array}\right) &\\
 &{\bf u}_2 = \{T,\eta\}_{E,\kappa} \left(\begin{array}{c} A_y(y) \\ 
       A(y)S_y(y)  \end{array}\right) &\\
 &{\bf u}_3 = \{T,\eta\}_{E,\kappa} \left(\begin{array}{c} 0 \\ 
        yA(y)/2\end{array}\right)  
  +\frac{TT_\kappa}{2} \left(\begin{array}{c} A_E(y) \\ A(y) S_E(y)\end{array}\right) \\
  & \hskip 170pt -  \frac{TT_E}{2} \left(\begin{array}{c} A_\kappa(y) \\
       A(y) S_\kappa(y)\end{array}\right).& 
\end{align*}
The left eigenvectors, which can be chosen as 
\begin{align*}
  &{\bf v}_0 = \bf{u}_1^t {\mathcal J} & \\
  &{\bf v}_1 = -{\bf u}_0^t {\mathcal J} & \\
  &{\bf v}_2 = {\bf u}_3^t {\mathcal J} & \\
  &{\bf v}_3 = -{\bf u}_2^t {\mathcal J} & \\
\end{align*}
satisfy
\begin{align*}
  & {\bf v}_{1,3} {\mathcal L} = 0 & \\
  & {\bf v}_{0,2} {\mathcal L} = {\bf v}_{1,3}. &
\end{align*}
The genericity conditions can be expressed 
as follows:
\begin{itemize}
  \item The kernel $ \{~ {\bf u}~ |~~ {\mathcal J} \partial^2 \tilde H {\bf u} = 0\}$
    is {\em exactly} two dimensional unless the quantity
    $\sigma = \{T,\eta\}_{E,\kappa}=\frac{\partial T}{\partial E}
    \frac{\partial \eta}{\partial \kappa} -\frac{\partial T}{\partial \kappa}
    \frac{\partial \eta}{\partial E} =0$ in which case the kernel
    is higher dimensional. More specifically 
    \begin{align}
 \nonumber     \dim(\ker({\mathcal L})) &= 2 +
      \dim\left(\ker\left(\begin{array}{cc}T_E & T_\kappa \\ 
          \eta_E & \eta_\kappa\end{array}\right)\right)& \\
          &= 2 + \dim\left(\ker\left(\begin{array}{cc}K_{EE} & K_{\kappa E} \\ 
          K_{\kappa E} & K_{\kappa\kappa}\end{array}\right)\right).&
    \end{align}
  \item Assuming that $\sigma  = \{T,\eta\}_{E,\kappa}\neq 0$ there are no
    Jordan chains of length greater than two as long as the determinant
    \begin{equation}
      \left|\begin{array}{cccc}
    K_{\kappa\kappa} & K_{\kappa E}& K_{\kappa\omega}  & T \\  K_{\kappa E}& K_{EE} & K_{E\omega} & 0 \\
    K_{\kappa\omega} & K_{E\omega}& K_{\omega\omega}  & 0 \\
    T  &  0 & 0   & -M 
  \end{array}\right| 
    \end{equation} 
    does not vanish. More specifically we have that
    \begin{multline}
      \dim(\ker({\mathcal J} \partial^2 \tilde H)^3 )-  \dim(\ker({\mathcal J} \partial^2 \tilde H)^2 )\\ = \dim\left(\ker \left( \begin{array}{cccc}
    K_{\kappa\kappa} & K_{\kappa E}& K_{\kappa\omega}  & T \\  K_{\kappa E}& K_{EE} & K_{E\omega} & 0 \\
    K_{\kappa\omega} & K_{E\omega}& K_{\omega\omega}  & 0 \\
    T  &  0 & 0   & -M 
  \end{array}\right)\right).
    \end{multline}
\end{itemize}
\end{theorem}

\begin{proof}
From standard ODE arguments we have that the traveling wave solutions
are (for parameter values in $\Omega$) $C^1$ functions of the
parameters. By differentiating the equation
\[
  A_y^2 = R(A;E,\kappa,\omega,\zeta),
\]
we have that $A_y,A_E,A_\kappa$ solve linear inhomogeneous
equations with linearly independent right-hand sides
\begin{align}
  &  2 A_y A_{yy} = R_y(A) A_y & \\
  &  2 A_y A_{Ey} = R_y(A) A_E + 2 & \\
  &  2 A_y A_{\kappa y} = R_y(A) A_\kappa - 2\frac{\kappa}{A^2} & 
\end{align}
The functions $A_y,A_E,A_\kappa$ must be linearly independent since if
there existed a linear combination of $A_y,A_E,A_\kappa$ that vanished
then the corresponding linear combination of the above ordinary
differential equations would vanish. Since the righthand sides are $0,2$
and $-2\frac{\kappa}{A^2}$ and $1$ and $\frac{1}{A^2}$ are assumed
linearly independent this would imply that $A_y$ is identically zero
and thus $A(y)$ constant, which we have assumed to be not true.  
Differentiating the traveling wave equations with respect
to the parameters $y_0,\theta_0,E,\kappa$ shows that four solutions to the
ordinary differential equation ${\mathcal L}{\bf u} = 0 $ are given by  
\begin{subequations}
\begin{align}
   & {\mathcal L} \left( \begin{array}{c} 0 \\ A(y) \end{array} \right) =0 & \label{eqn:ODESOL1} \\
   & {\mathcal L} \left( \begin{array}{c} A_y(y) \\ A(y) S_y(y)   \end{array} \right)  = 0& \\
   & {\mathcal L} \left( \begin{array}{c} A_E(y) \\ A(y) S_E(y)  \end{array} \right) =0& \label{eqn:ODESOL3}  \\
   & {\mathcal L} \left( \begin{array}{c} A_\kappa(y) \\ A(y) S_\kappa(y) \end{array} \right)=0. & \label{eqn:ODESOL4}
\end{align}\label{eqn:ODESOLs}
\end{subequations}
Since there is no linear combination of
\[
    \left( \begin{array}{c} 0 \\ A(y) \end{array} \right),~~
    \left( \begin{array}{c} A_y(y) \\ A(y) S_y(y)  \end{array} \right), ~~
    \left( \begin{array}{c} A_\kappa(y) \\ A(y)S_\kappa(y) \end{array} \right),~~
    \left( \begin{array}{c} A_E(y) \\ A(y)S_E(y) \end{array} \right)
\]
for which even the first component vanishes then the four solutions are linearly independent.  
Note that these are four linearly independent solutions to the {\em ordinary
differential equation} ${\mathcal L}{\bf u}=0$. This generally {\em does not} give four
linearly independent solutions to the operator equation ${\mathcal L}{\bf u}
= 0$, as the operator equation must be considered on some function space and these functions may not
satisfy appropriate boundary conditions. The first two quantities, 
$\left( \begin{array}{c} 0 \\ A(y) \end{array} \right), \left( \begin{array}{c} A_y(y) \\ A(y) S_y(y)  \end{array} \right)$ are periodic, but the other two are typically not periodic since the period $T$ and the net phase $S(T) - S(0) = \eta $ are functions of $E,\ \kappa,\ \omega,\ \zeta$. Additionally differentiating with respect
to the quantities $\omega,\ c$ gives the relations
\begin{subequations}
\begin{align}
    &{\mathcal L} \left( \begin{array}{c} A_\omega(y) \\ A(y) S_\omega(y)  \end{array}\right) = \left( \begin{array}{c} 0 \\ A(y) \end{array} \right) & \\
    &{\mathcal L} \left( \begin{array}{c} 0 \\ yA(y)/2  \end{array}\right) = \left( \begin{array}{c} A_y(y) \\ A(y) S_y(y)  \end{array} \right) 
\end{align}\label{eqn:ODESOLSgeneral}
\end{subequations}
These are again not in the generalized kernel of ${\mathcal L}$ since
they are not periodic. However, one can find linear combinations of
either of the above vectors with 
$\left( \begin{array}{c} A_E(y) \\ A(y)S_E(y)  \end{array} \right)$
and $\left( \begin{array}{c} A_\kappa(y) \\ A(y)S_\kappa(y)  \end{array}\right)$ 
which are periodic. The details are given in Appendix \ref{ap:eigenfunctions}. 

Regarding the genericity conditions, we again note that equations
(\ref{eqn:ODESOL1}--\ref{eqn:ODESOL4}) give four linearly
independent solutions to the ODE ${\mathcal L} {\bf u} = 0$, the first two
being $T$ periodic. The traveling waves are chosen so
that $A'(0)=0$ for all parameter values. The change in
the second two solutions and their derivatives over one period
is given by 
\begin{align}
    & \left.\left( \begin{array}{c} A_E \\ A(y) S_E  \end{array} \right)\right|_0^T  
        =\left(\begin{array}{c}0 \\ A(0) \left(\eta_E  - \frac{\kappa}{A(0)^2}T_E\right) \end{array}\right) & \label{eqn:Jump1}\\
    & \left.\left( \begin{array}{c} A_E \\ A(y) S_E  \end{array} \right)_y\right|_0^T
        =  \left(\begin{array}{c} -T_E A''(0) \\ 0\end{array}\right) &\label{eqn:Jump2}\\
    & \left.\left( \begin{array}{c} A_\kappa \\ A(y) S_\kappa  \end{array} \right)\right|_0^T
        =\left(\begin{array}{c}0 \\ A(0) \left(\eta_\kappa - \frac{\kappa}{A(0)^2}T_\kappa\right) \end{array}\right) & \label{eqn:Jump3}\\
    & \left.\left( \begin{array}{c} A_\kappa \\ A(y) S_\kappa  \end{array} \right)_y\right|_0^T
        = \left(\begin{array}{c} -T_\kappa A''(0) \\ 0\end{array}\right) &\label{eqn:Jump4}
\end{align}
Note that $A(0)\neq 0$ and $A''(0)\neq 0$ so that the only way one can
construct a non-zero periodic linear combination of the solutions
(\ref{eqn:ODESOL3}) and (\ref{eqn:ODESOL4}) is if the determinant 
\[
  \sigma = \left| \begin{array}{cc} T_E & T_\kappa\\ \eta_E & \eta_{\kappa}\end{array}\right| = 0,
\]
and further the number of periodic solutions that can be
constructed is equal to the rank of the matrix $\left(\begin{array}{cc} T_E & T_\kappa \\ \eta_E & \eta_\kappa \end{array}\right)$. The same argument
applies to the left kernel of ${\mathcal L}$, with the same
genericity condition.

Assuming that the genericity condition $T_E \eta_\kappa - T_\kappa
\eta_E\neq 0$ is satisfied and the left and
right kernels of ${\mathcal L}$ are two dimensional we next show that
there is typically nothing in the next generalized kernel. In order
to have a element of $\ker({\mathcal L}^3) / \ker({\mathcal L}^2)$ we
must have a linear combination of ${\bf u}_1$ and ${\bf u}_3$ that lies in the
range of ${\mathcal L}$. By the Fredholm alternative this is
equivalent to finding a linear combination of ${\bf u}_1$ and ${\bf u}_3$ that is
orthogonal to ${\bf v}_1$ and ${\bf v}_3$. This, in turn, is equivalent to the
gram determinant  
\[
    \left|\begin{array}{cc}{\bf v}_1 {\bf u}_1 & {\bf v}_3 {\bf u}_1\\ {\bf v}_1 {\bf u}_3 & {\bf v}_3 {\bf u}_3 \end{array}\right|
    = \frac{\sigma^2}{4} \left|\begin{array}{cc} 
    \{\eta,T,M\}_{\kappa,E,\omega} & \frac{T}{2}\{T,M\}_{\kappa,E} \\  
    \frac{T}{2}\{T,M\}_{\kappa,E} & -\frac{T^2 T_E}{2} + \frac{M}{2}\{\eta,T\}_{\kappa,E} \end{array}\right|
\] 
being zero, where $\sigma =\{\eta,T\}_{\kappa,E}$.  The elements of the above gram matrix are computed in Appendix \ref{ap:matrixelements}. To relate this to the determinant of the four by four matrix we remind the reader of the identities $T=K_E;M = -2 K_\omega; \eta = - K_{\kappa}$. We next use the Dodgson-Jacobi-Desnanot condensation identity \cite{Dodgson}, which was proved by Dodgson prior to his well-known early work on
mirror symmetry \cite{Mirror}. The condensation identity says that if
$D$ is the determinant of a matrix, and $D_i^j$ is the minor determinant with the
$i^{th}$ row and $j^{th}$ column removed then   
\[
  D D_{ij}^{ij} = D_i^i D_j^j - D_i^j D_j^i. 
\]
Applying this to the $4\times 4$ determinant 
\begin{equation}
      D=\left|\begin{array}{cccc}
    K_{\kappa\kappa} & K_{\kappa E}& K_{\kappa\omega}  & T \\  K_{\kappa E}& K_{EE} & K_{E\omega} & 0 \\
    K_{\kappa\omega} & K_{E\omega}& K_{\omega\omega}  & 0 \\
    T  &  0 & 0   & -M 
  \end{array}\right| 
    \end{equation} 
    with $i=3$ and $j=4$ we find that 
    \begin{align*}
        &D_4^4 = \{K_\kappa,K_E,K_\omega\}_{\kappa,E,\omega}= -\frac{1}{\sigma} {\bf v}_1 {\bf u}_1,& \\
        & D_3^4 = D_4^3 = T\{K_E,K_\omega\}_{\kappa,E} = \frac{2}{\sigma} {\bf v}_3{\bf u}_1 =  \frac{2}{\sigma} {\bf v}_1{\bf u}_3,& \\
        & D_3^3 = -M \{K_\kappa,K_E\}_{\kappa,E} - T^2 K_{EE}= -\frac{4}{\sigma}{\bf v}_3 {\bf u}_3,&  \\
    &\left|\begin{array}{cc}{\bf v}_1 {\bf u}_1 & {\bf v}_3 {\bf u}_1\\ {\bf v}_1 {\bf u}_3 & {\bf v}_3 {\bf u}_3 \end{array}\right|
        = \frac{\sigma^3}{4} D.&
\end{align*}
\end{proof}
\begin{remark}
As was discussed earlier we always move to the Galilean frame in which $c=0.$ This is not strictly necessary---one can carry out the entire calculation as a function of $c$. This is somewhat algebraically messy, as many of the quantities in question become functions of $c$. It is somewhat theoretically cleaner however, so we remark that if one maintains the full $c$ dependence the non-existence of a higher generalized kernel can be shown to be equivalent to the matrix $\frac{\partial^2 K}{\partial (\kappa,E,\omega,c)^2 }$ being non-singular. The quantity $D$ is, modulo some unimportant constants, equal to the determinant of $\frac{\partial^2 K}{\partial (\kappa,E,\omega,c)^2 }$.

    We would now like to introduce a few  more pieces of notation.
    By differentiating the traveling wave
    equation with respect to $\zeta$, we have the identity
    \[
        {\mathcal L} \left(\begin{array}{c} A_\zeta \\ A S_\zeta \end{array}\right) =  \left(\begin{array}{c} 0 \\ f(A^2)A \end{array}\right). 
    \]
    The functions $A_\zeta, S_\zeta$  are typically not periodic but
    they can be periodized in the same way as was done
    previously for the elements of the generalized
    kernel, by adding linear combinations of
    $A_\kappa,A_E$, etc. We define the following
    quantities  
    \begin{equation}\arraycolsep=18pt\def\arraystretch{1.6}
    \begin{array}{ccc}
     \gamma  = \{T,\eta\}_{\kappa,\omega} & \rho = \{T,\eta\}_{\omega,E} & \tau = TT_\kappa/2 \\
    \nu  = -TT_E/2
        & \xi = \{T,\eta\}_{\kappa,\zeta} 
        & \psi  = \{T,\eta\}_{\zeta,E}.   
    \end{array} \label{eqn:determinants}
    \end{equation}
    Given these definitions we have
    \begin{subequations}
   \begin{align}
        &{\bf u}_0=\sigma\left(\begin{array}{c}0\\ A\end{array}\right)\\
        &{\bf u}_1=\gamma\left(\begin{array}{c}A_E\\AS_E\end{array}\right)
            +\rho\left(\begin{array}{c}A_\kappa\\AS_\kappa\end{array}\right)
            +\sigma\left(\begin{array}{c}A_\w\\AS_\w\end{array}\right)\\
        &{\bf u}_2=\sigma\left(\begin{array}{c}A_y\\AS_y\end{array}\right)\\
        &{\bf u}_3=\tau\left(\begin{array}{c}A_E\\AS_E\end{array}\right)
            +\nu\left(\begin{array}{c}A_\kappa\\AS_\kappa\end{array}\right)
            +\sigma\left(\begin{array}{c}0\\yA/2\end{array}\right)\\
        &{\bf v}_0=\gamma\left(\begin{array}{c}AS_E\\-A_E\end{array}\right)^t
            +\rho\left(\begin{array}{c}AS_\kappa\\-A_\kappa\end{array}\right)^t
            +\sigma\left(\begin{array}{c}AS_\w\\-A_\w\end{array}\right)^t\\
        &{\bf v}_1=-\sigma\left(\begin{array}{c}A\\0\end{array}\right)^t\\
        &{\bf v}_2=\tau\left(\begin{array}{c}AS_E\\-A_E\end{array}\right)^t
            +\nu\left(\begin{array}{c}AS_\kappa\\-A_\kappa\end{array}\right)^t
            +\sigma\left(\begin{array}{c}yA/2\\0\end{array}\right)^t\\
        &{\bf v}_3=-\sigma\left(\begin{array}{c}AS_y\\-A_y\end{array}\right)^t 
    \end{align}
    \end{subequations}
    It is also convenient to construct two more functions that will be
    useful in the remainder of the paper. These will not belong to the kernel or the
    generalized kernel, but will instead represent a particular function
    in the range of the linearized operator, together with its pre-image. These are given by
    \begin{align*}
        &{\bf u}_4=\sigma\left(\begin{array}{c}0\\ f(A^2)A\end{array}\right)& \\
        &{\bf u}_5=\xi\left(\begin{array}{c}A_E\\AS_E\end{array}\right)
            +\psi\left(\begin{array}{c}A_\kappa \\AS_\kappa \end{array}\right)
            +\sigma\left(\begin{array}{c}A_\zeta\\AS_\zeta\end{array}\right)\\
    \end{align*}

    We see that ${\bf u}_4$ is periodic, since $A(y)$ is periodic, and
    ${\bf u}_5$ is periodic and satisfies ${\mathcal L}{\bf u}_5 = {\bf u}_4$.  
\end{remark}
  
At this point we have constructed the generalized kernel of the
linearization of the NLS equation around a periodic traveling wave
solution. We know that generically the kernel consists of a direct
sum of two Jordan chains of length two, one associated with each
commuting flow. The next step in the calculation is to understand
how this four dimensional generalized kernel breaks up under a
perturbation. This perturbation can represent either slow
modulations in a transverse spatial direction or perturbations in
the boundary conditions due to a change in the quasi-momentum or
Floquet exponent.

\section{Modulational Stability Results}

\subsection{Perturbation of a Jordan Subspace}
  
In the prior section we established the structure of the
kernel of the linearized operator. The purpose of this section 
is to derive the fundamental perturbation result. Generically under
a perturbation of size $\mu$ a Jordan block of size $d$ will have a representation
in terms of a Puiseaux series in powers of
$\mu^{\frac{1}{d}}$. The breakup of a Jordan block under
perturbation has been considered by many authors including
Lidskii \cite{Lidskii}, Baumg{\"a}rtel \cite{Baumgartel}, Moro, Burke
and Overton \cite{MBO} and (in the context of stability of
Hamiltonian systems) Maddocks and Overton \cite{Maddocks.Overton}.
In many of these works the perturbations are assumed to be in some
sense generic, and the resulting series are in the form of the
aforementioned Puiseaux series. In the problems of interest here the
perturbations are highly non-generic, and admit a standard power series
representation in powers of $\mu$. The first result covers the
situation of interest to our problem.  
  
\begin{prop} \label{prop:Breakup}
    Suppose that an operator ${\mathcal L}$  with compact resolvent
    has a $d$-dimensional
    kernel spanned by $\{{\bf u}_{2j}\}_{j=0}^{d-1}$ and a $d$-dimensional first
    generalized kernel spanned by $\{{\bf u}_{2j+1}\}$ satisfying
    ${\mathcal L} {\bf u}_{2j+1} = {\bf u}_{2j}$, and similarly a left basis
    satisfying ${\bf v}_{2j+1} {\mathcal L} = 0$, ${\bf v}_{2j} {\mathcal L} =
    {\bf v}_{2j+1}$.  Consider a perturbation of the form 
    ${\mathcal L} + \mu {\mathcal L}^{(1)} + \mu^2 {\mathcal  L}^{(2)}$, 
    where ${\mathcal L}^{(1)} (\lambda - {\mathcal L})^{-1}$, 
    ${\mathcal L}^{(1)} (\lambda - {\mathcal L})^{-1} {\mathcal L}^{(1)}$ 
    and ${\mathcal L}^{(2)}$ are bounded operators, and suppose that the first order perturbation
    satisfies the additional conditions that
    \begin{equation}
      {\bf v}_{2j+1} {\mathcal L}^{(1)} {\bf u}_{2k} = 0 \qquad \forall j,k=0,...,d-1.
      \label{eqn:Solve}
    \end{equation}
    Then, to leading order in $\mu$, the $2d$ dimensional generalized
    kernel breaks up into $2d$ eigenspaces. The eigenvalues are given
    to leading order by
    \[
        \lambda(\mu) = \lambda_1 \mu + o(\mu)
    \]
    where $\lambda_1$ is a root of the polynomial
    \begin{equation}
        \det(\lambda_1^2 {\bf M}^{(2)} + \lambda_1 {\bf M}^{(1)} + {\bf M}^{(0)}) = 0 
        \label{eqn:QuadPencil}
    \end{equation}
    and ${\bf M}^{(i)}$ are $d \times d$ matrices defined as follows:
    \begin{align*}
        &{\bf M}^{(2)}_{j,k} = {\bf v}_{2j+1} {\bf u}_{2k+1} & \\
        &{\bf M}^{(1)}_{j,k} = -{\bf v}_{2j+1} {\mathcal L}^{-1} {\mathcal L}^{(1)} {\bf u}_{2k}
            - {\bf v}_{2j+1} {\mathcal L}^{(1)}{\mathcal L}^{-1} {\bf u}_{2k} & \\
        & ~~~~=- {\bf v}_{2j} {\mathcal L}^{(1)} {\bf u}_{2k} -
            {\bf v}_{2j+1} {\mathcal L}^{(1)} {\bf u}_{2k+1} & \\
        & {\bf M}^{(0)}_{j,k} = {\bf v}_{2j+1} {\mathcal L}^{(1)} {\mathcal
            L}^{-1} {\mathcal L}^{(1)} {\bf u}_{2k}  -{\bf v}_{2j+1} {\mathcal L}^{(2)}  {\bf u}_{2k} 
    \end{align*}
    and ${\mathcal L}^{-1}$ is the Moore-Penrose pseudo-inverse.
\end{prop}

\begin{remark}
    A few remarks on the previous proposition. First note that
    condition (\ref{eqn:Solve})
    is a necessary and sufficient condition for the quantity 
    ${\bf v}_{2j+1} {\mathcal L}^{(1)} {\mathcal L}^{-1} {\mathcal L}^{(1)} {\bf u}_{2k}$ 
    to be well-defined. Also note that, compared with self-adjoint perturbation theory there is a
    certain amount of ``mixing'' of orders: certain terms that are
    formally second order in $\mu$ in the expansion of the operator
    contribute to first order in   $\mu$ in the eigenvalue. This is
    typical when one perturbs a Jordan block due to the fact that
    the range has a non-trivial intersection with the kernel. 
\end{remark}
\begin{proof}

We give an argument based on a Schur complement calculation and
the Weierstrauss preparation theorem.

Since the generalized kernel of the operator ${\mathcal L}$ has  Jordan chains of length two we can decompose our function space as $\ker_0({\mathcal L}) + \ker_1({\mathcal L}) + \ran({\mathcal L}^2)$. If ${\bf u}$ is an eigenfunction of the perturbed problem we will denote the projections of ${\bf u}$ onto $\ker_0({\mathcal L})$, $\ker_1({\mathcal L})$ and $ \ran({\mathcal L}^2)$ by ${\bf x}_0, {\bf x}_1$ and ${\bf x}_2$ respectively. We will express  the operators ${\mathcal L},{\mathcal L}^{(1)}$ and ${\mathcal L}^{(2)}$ in block form in this basis as follows:
\begin{align*}
& {\mathcal L} = \left(\begin{array}{ccc}{\bf0} & \boxed{{\bf A}_{01}} & {\bf0} \\ {\bf0} & {\bf0} & {\bf0} \\ {\bf0} & {\bf0} & \boxed{{\bf A}_{22}}\end{array}\right) ~~~  {\mathcal L}^{(1)} = \left(\begin{array}{ccc}\boxed{{\bf B}_{00}} & {\bf B}_{01} & {\bf B}_{02} \\ {\bf0} & \boxed{{\bf B}_{11}} & \boxed{{\bf B}_{12}} \\ \boxed{{\bf B}_{20}} & {\bf B}_{21} & {\bf B}_{22}\end{array}\right)& \\
& {\mathcal L}^{(2)}=\left(\begin{array}{ccc}{\bf C}_{00} & {\bf C}_{01} & {\bf C}_{02} \\ \boxed{{\bf C}_{10}} & {\bf C}_{11} & {\bf C}_{12} \\ {\bf C}_{20} & {\bf C}_{21} & {\bf C}_{22}\end{array}\right) ~~~ {\bf u} = \left(\begin{array}{c}
     {\bf x}_0  \\
     {\bf x}_1 \\
     {\bf x}_2
\end{array}\right)
\end{align*}
The Jordan structure of ${\mathcal L}$ implies that the indicated blocks are zero.  
The fact that the $(1,0)$ block of ${\mathcal L}^{(1)}$ is zero is forced by Equation (\ref{eqn:Solve}). The rest of the entries are arbitrary, but only the boxed ones contribute to the leading order perturbation theory, as will be seen.  We begin with the perturbed eigenvalue problem
\[
\left({\mathcal L} + \mu {\mathcal L}^{(1)} + \mu^2 {\mathcal L}^{(2)}\right) {\bf u} = \lambda {\bf u}.
\]
Given the definitions above this is equivalent to the system 
\begin{align*}
    & {\bf A}_{01}{\bf x}_1 + \mu ({\bf B}_{00} {\bf x}_0 + {\bf B}_{01} {\bf x}_1 + {\bf B}_{02} {\bf x}_2) +O(\mu^2)= \lambda {\bf x}_0&\\
    & \mu ({\bf B}_{11} {\bf x}_1 + {\bf B}_{12} {\bf x}_2) +\mu^2 ({\bf C}_{10}{{\bf x}_0}+{\bf C}_{11}{\bf x}_1+{\bf C}_{12}{\bf x}_2) = \lambda {\bf x}_1 &\\
    & {\bf A}_{22} {\bf x}_2 + \mu ({\bf B}_{20} {\bf x}_0+ \mu {\bf B}_{21} {\bf x}_1 + \mu {\mathcal B}_{22} {\bf x}_2)+O(\mu^2) = \lambda {\bf x}_2.
\end{align*}
A judicious application of elimination, or the Schur complement formula, gives
\begin{align*}
    & ({\bf A}_{01} + \mu {\bf B}_{01}+O(\mu^2)) {\bf x}_1 = (\lambda - \mu {\bf B}_{00}+O(\mu^2)) {\bf x}_{0}&\\
    & (\lambda - \mu {\bf B}_{11} +O(\mu^2)) {\bf x}_1 = \mu^2 ({\bf C}_{10} -{\bf B}_{12} {\bf R} {\bf B}_{20}) {\bf x}_0+O(\mu^3),&
\end{align*}
where ${\bf R}={\bf R}(\lambda,\mu) = ( {\bf A}_{22} - \lambda+\mu {\bf B}_{22})^{-1}$. Note that by assumption ${\bf A}_{22}^{-1} {\bf B}_{22}$ is bounded, and thus $({\bf A}_{22} - \lambda+\mu {\bf B}_{22})$ is invertible for sufficiently small $\mu,\lambda$. 
If one further eliminates ${\bf x}_1$, we obtain a simplified equation for ${\bf x}_0$: 
\begin{align}
    \nonumber   &\mu^2 ({\bf C}_{10} - {\bf B}_{12} {\bf R} {\bf B}_{20}) {\bf x}_0= (\lambda - \mu {\bf B}_{11} + O(\mu^2)) \times & \\
   & ({\bf A}_{01}+\mu {\bf B}_{01}+O(\mu^2))^{-1} (\lambda - \mu {\bf B}_{00}+O(\mu^2)) {\bf x}_{0}+O(\mu^3),&
   \label{eqn:Perturb}
\end{align}
giving an analytic matrix pencil \cite{Lidskii}.   
Note that boundedness of ${\bf B}_{01}$ implies invertibility of  $({\bf A}_{01} + \mu {\bf B}_{01})$ for $\mu$ sufficiently small. Writing (\ref{eqn:Perturb}) in the form ${\bf N}(\lambda,\mu) {\bf x}_0= 0$ we find the condition 
\begin{multline}
    \det({\bf N}(\lambda,\mu)) = \det\left( (\lambda - \mu {\bf B}_{11}){\bf A}_{01}^{-1}(\lambda
    - \mu {\bf B}_{00})\right.\\ - \left.\mu^2 ({\bf C}_{10} - {\bf B}_{12} {\bf A}_{22}^{-1}
    {\bf B}_{20})\right) (1+O(\lambda,\mu))=0.
\end{multline}
The Weierstrauss preparation theorem then implies that there is a
branch of solutions $\lambda(\mu) = \lambda_1 \mu + O(\mu^2)$ with
$\lambda_1$ a root of 
\[
    \det\left( (\lambda_1 - \mu {\bf B}_{11}){\bf A}_{01}^{-1}(\lambda_1
    -  \mu{\bf B}_{00}) -  \mu^2({\bf C}_{10} - {\bf B}_{12} {\bf A}_{22}^{-1}
    {\bf B}_{20})\right)=0.\]
It only remains to recognize that this is the same as formula
(\ref{eqn:QuadPencil}). To see this note that the Moore-Penrose
inverse is given by
\[
   {\mathcal L}^{-1} =  \left(\arraycolsep=2.0pt\def\arraystretch{1.4}\begin{array}{c|c|c} 
   0 & 0 & 0 \\ \hline 
   {\bf A}_{01}^{-1} & \hskip 4pt 0\hskip 4pt & 0 \\\hline 
   0 & 0 & {\bf A}_{22}^{-1} \end{array}\right). 
\]

From this we see that the $(1,0)$
block of ${\mathcal L}^{-1} {\mathcal L}^{(1)}$
is ${\bf A}_{01}^{-1}{\bf B}_{00}$, the $(1,0)$ block of  
${\mathcal  L}^{(1)} {\mathcal L}^{-1}$ is ${\bf B}_{11}{\bf A}_{01}^{-1}$, and
the $(1,0)$ block of ${\mathcal L}^{(1)} {\mathcal L}^{-1} {\mathcal L}^{(1)}$ is
${\bf B}_{11} {\bf A}_{01}^{-1}{\bf B}_{00} + {\bf B}_{12} {\bf A}_{22}^{-1} {\bf B}_{20}$. 
The fact that ${\bf A}_{01}$ is a Gram matrix rather than the identity reflects 
the fact that our basis is not orthonormal.

\end{proof}

\subsection{Evaluation of Stability to Longitudinal and Transverse Perturbations}\label{sec:matrixcoeff}

We are now ready to derive the equations governing the modulational stability of quasi-periodic traveling waves. We have previously shown that generically the operator $\mathcal{L}$ has a two dimensional kernel and a two dimensional first generalized kernel, and we have also shown how such a Jordan block breaks up under a certain restricted class of perturbations. The last step is to realize that both transverse and longitudinal long-wavelength perturbations can be treated as small perturbations to the operator ${\mathcal L}$,
and to compute the appropriate matrix elements of the perturbation. 

For the case of longitudinal perturbations we first note that the operator ${\mathcal L}$ is a differential operator on ${\mathbb R}$ with $T$-periodic coefficients. It follows from Floquet theory~\cite{Kuchment} that the spectrum of ${\mathcal L}$ on $L_2({\mathbb R})$ is given by a union over all quasi-momenta in the dual torus $\mu\in (-\frac{\pi}{T},\frac{\pi}{T}]$ of the spectrum of ${\mathcal L}$ subject to quasi-periodic boundary conditions ${\bf u}(T) = e^{i \mu T} {\bf u}(0); {\bf u}_y(T) = e^{i \mu T} {\bf u}_y(0)$. This $\mu$-dependent boundary condition can be translated to a $\mu$-dependent operator on $L_2({\mathbb R})$ by the change of variable $ {\bf u} = e^{i \mu y} \tilde {\bf u}$, in terms of which ${\mathcal L} {\bf u} = \lambda {\bf u}$ becomes ${\mathcal L}(\mu) \tilde {\bf u} = \lambda \tilde {\bf u}$, where ${\mathcal L}(\mu) = e^{-i \mu y} {\mathcal L} e^{i \mu y} = {\mathcal L} + \mu {\mathcal L}^{(1)} + \mu^2 {\mathcal L}^{(2)}$, with ${\mathcal L}^{(1)}$ and ${\mathcal L}^{(2)}$ given by  
\begin{align*}
        & {\mathcal L}^{(1)} = 2 i\left(\begin{array}{cc}  S_y & \partial_y \\ - \partial_y &  S_y  \end{array}\right)  & \\
        & {\mathcal L}^{(2)} = \left(\begin{array}{cc}  0 & -1 \\ 1 & 0 \end{array}\right).  & 
    \end{align*}
All that remains is to compute the various matrix elements in the block decomposition of the perturbation terms and apply Proposition \ref{prop:Breakup}. 

\begin{corr}[Normal Form for Longitudinal Perturbations]
    Suppose that ${\mathcal L}$ is the linearization of the nonlinear Schr\"odinger equation about a traveling wave solution, and that the genericity conditions on the kernel of ${\mathcal L}$ are met. 
    For small values of $\mu$, the Floquet
    multiplier, a normal form for the four bands of continuous spectrum emerging from the origin in the spectral plane is as follows:
    \begin{equation}
        \det\left(\lambda^2
        \left(\begin{array}{cc} a_2 & b_2 \\ b_2 & d_2 \end{array} \right) 
        + \lambda \mu \left(\begin{array}{cc} a_1 & b_1 \\ b_1 & d_1 \end{array}\right) 
        + \mu^2 \left(\begin{array}{cc} a_0 & b_0 \\ b_0 & d_0 \end{array}\right)\right)
        + O(5) = 0  \label{eqn:quartic}
    \end{equation}
    where $O(5)$ denotes terms $\lambda^j \mu^k$ with $j,k\geq 0$ and  $j+k\geq 5$.
    The matrix elements are given by (cf. equation \ref{eqn:QuadPencil})

    \begin{subequations}
    \begin{align}
        & a_2 ={\bf v}_1 {\bf u}_1 =-\frac{\sigma}{2}(\gamma M_E +\rho M_\kappa+ \sigma M_\omega)
        &\\
        \nonumber & b_2 = {\bf v}_1 {\bf u}_3 
        =-\frac{\sigma\rho T}{2}
        &\\
        & \ \ \ \ \ = {\bf v}_3{\bf u}_1 = -   \frac{\sigma}{2}(\tau M_E +\nu M_\kappa) 
        &\\
        & d_2 = {\bf v}_3 {\bf u}_3 
        =-\frac{\sigma}{2}\left(\nu T+\sigma M/2\right)\\
        & a_1 = -{\bf v}_0\mathcal L^{(1)}{\bf u}_0 - {\bf v}_1\mathcal L^{(1)}{\bf u}_1 =  i \sigma T(2\rho) & \\
        & b_1  = -{\bf v}_0\mathcal L^{(1)}{\bf u}_2 - {\bf v}_1\mathcal L^{(1)}{\bf u}_3 =  i \sigma T(\nu +\gamma)& \\
        & d_1 =   -{\bf v}_2\mathcal L^{(1)}{\bf u}_2 - {\bf v}_3\mathcal L^{(1)}{\bf u}_3 =  i \sigma T (2 \tau +\sigma\kappa) & \\
        & a_0 = {\bf v}_1\mathcal L^{(1)} \mathcal L^{-1} \mathcal L^{(1)}{\bf u}_0 - {\bf v}_1 \mathcal L^{(2)}{\bf u}_0 = 2 \sigma T \nu& \\
        & b_0 = {\bf v}_1\mathcal L^{(1)} \mathcal L^{-1} \mathcal L^{(1)}{\bf u}_2 - {\bf v}_1 \mathcal L^{(2)}{\bf u}_2  =  2 \sigma T \tau & \\
        & d_0 = {\bf v}_3\mathcal L^{(1)} \mathcal L^{-1} \mathcal L^{(1)}{\bf u}_2 - {\bf v}_3 \mathcal L^{(2)}{\bf u}_2  =  2 \sigma T ( \omega \gamma-\zeta \xi -E\sigma) &
    \end{align}\label{eqn:quarticmatrixelements}
    \end{subequations}
\end{corr}
\begin{proof}
Assuming that the genericity condition holds,  the normal form follows from Proposition (\ref{prop:Breakup}), in particular Equation (\ref{eqn:QuadPencil}). 
All of the elements of the matrices ${\bf M}^{(0)},{\bf M}^{(1)},{\bf M}^{(2)}$ can be computed explicitly.  
In Appendix \ref{ap:matrixelements} we compute a few of these matrix elements; the remainder follow similarly.  Note that all of the matrix elements can be expressed in terms of $\eta,T,M$ and their derivatives (see Equation \eqref{eqn:determinants}).

Note that the matrices ${\bf M}^{(0)}$ and ${\bf M}^{(2)}$ are real and symmetric while ${\bf M}^{(1)}$ is symmetric and purely imaginary. It is perhaps more convenient to work with $i\lambda$ instead: $i\lambda$ satisfies a real quartic equation, and has roots that are either real or are complex conjugate pairs. This corresponds to roots of $\lambda$ either being purely imaginary or symmetric about the imaginary axis $(\lambda,-\lambda^*)$. This represents {\em half} of the usual Hamiltonian quartet $\lambda,-\lambda,\lambda^*,-\lambda^*$. The other half of the Hamiltonian quartet comes from mapping $\mu \mapsto -\mu$, which is equivalent to taking the complex conjugate of the equation.   
\end{proof}
 
The case of transverse perturbations is similar though a bit more straightforward. Here we are considering the stability of periodic solutions of the one-dimensional NLS 
\[
  i w_t = w_{xx} + \zeta f(|w|^2) w
  \]
 as solutions of the two-dimensional NLS 
\[
  i w_t = w_{xx} \pm w_{zz} + \zeta f(|w|^2) w. 
  \]
  The $+$ case is often referred to as the elliptic 2D NLS, while the $-$ case is referred to as the hyperbolic 2D NLS.   
  The linearization of the $2D$ NLS around a one-dimensional solution with perturbation of the form
  \begin{equation}
      w(y,z,t) = e^{i\omega t}\left(\phi(y)+\epsilon e^{i[S(y+y_0)+cy/2+\theta_0]}W(y,z,t)\right)
  \end{equation}
  leads to a linear operator of the form ${\mathcal L} \mp {\mathcal J} \frac{\partial^2}{\partial z^2}$, where ${\mathcal L}$ is the linearized operator from the one-dimensional case.  Since ${\mathcal L}$ is independent of the transverse variable $z$ one can take the Fourier transform in the $z$, with dual variable $k$, to find the operator 
\[
{\mathcal L}(k) = {\mathcal L} \pm k^2 {\mathcal L}^{(2)}.
\]
Thus for transverse modulations there is no ${\mathcal L}^{(1)}$ term and the ${\mathcal L}^{(2)}$ takes a similar form to that for longitudinal modulations, with the small parameter being the transverse wave-number $k$ rather than the Floquet exponent $\mu$. 
\begin{corr} (Normal Form for Transverse Pertubations) 
  Suppose that ${\mathcal L}(k)$ is the Fourier transform in $z$ of the linearization of the 2D NLS equation about a one-dimensional periodic traveling wave solution that satisfies the genericty conditions. 
  
  For small $k$ a normal form for the four bands of continuous spectrum emerging from the origin in the spectral plane is  
  \begin{equation}
  \det\left(\lambda^2 \left(\begin{array}{cc} a_2 & b_2 \\ b_2 & d_2 \end{array}\right)  \pm k^2 \left(\begin{array}{cc} a_0^\prime & b_0^\prime \\ b_0^\prime & d_0^\prime  \end{array}\right)\right) + O(5) =0
  \label{eqn:transquartic}
  \end{equation}
 where  the coefficients $a_2,b_2,d_2$ are as in the previous corollary and $a_0^\prime,b_0^\prime,d_0^\prime$ are given by 
 
 \begin{align*}
     & a_0^\prime = - {\bf v}_1\mathcal L^{(2)}{\bf u}_0 = -  \sigma^2 M & \\
     & b_0^\prime = - {\bf v}_1\mathcal L^{(2)}{\bf u}_2 = -  \sigma^2 \kappa T & \\
     & d_0^\prime = - {\bf v}_3\mathcal L^{(2)}{\bf u}_2 = -  \sigma^2 (2 ET-\w M -\zeta U) & \\
 \end{align*}
 with
 \[U = \int_0^T F(A^2)dy.\]
 
 Note that (again for small wavenumber $k$) the eigenvalues of the linearized hyperbolic 2D NLS operator under transverse perturbation are $i$ times the eigenvalues of the linearized elliptic 2D NLS. 
 
\end{corr}
\begin{proof}
As in the previous result the proof amounts to computing the appropriate matrix elements. Note that the ${\bf M}^{(0)}$ terms are slightly different in the two cases, as the longitudinal case includes terms arising from ${\mathcal L}^{(2)}$ as well as terms arising from ${\mathcal L}^{(1)} {\mathcal L}^{-1} {\mathcal L}^{(1)},$ the latter of which do not arise in the transverse case.
\end{proof}

 As noted previously $i w_t = w_{xx} +  w_{zz} + \zeta f(|w|^2) w$ is called the elliptic NLS and $i w_t = w_{xx} - w_{zz} + \zeta f(|w|^2) w$ the hyperbolic NLS. In the numerics section we will refer to these as focusing or defocusing based on the nature of the 1D problem, $i w_t = w_{xx} + \zeta f(|w|^2) w$, though in the hyperbolic case the distinction between focusing and defocusing is not particularly meaningful.  

These results give a rigorous normal form for the eigenvalues of the linearized operator ${\mathcal L}$ in a neighborhood of the origin in the spectral plane as a function of the quasimomentum $\mu$ or the transverse wavenumber $k$ respectively. One could in principle combine these, and consider modulations in both the transverse and longitudinal directions, though we have not done so here. 
The eigenvalues are locally given by the roots of a homogeneous quartic polynomial in $\lambda$ and $\mu$ or $k$. The coefficients of this polynomial can be expressed in terms of the quantities $T,M,\eta,U $ and their derivatives with respect to the various parameters. The most interesting case physically is that of a potential $F(|A|^2)$ that is polynomial. In this situation it is a classical result that there are only a finite number of such differentials, and that the derivatives of such quantities can always be expressed as a linear combination of such quantities, a result known as the Picard--Fuchs relations~ \cite{Fuchs}.  In the next section we briefly outline the derivation of the Picard--Fuchs relations, as these are a natural way to compute the derivatives from both the analytical as well as the numerical point of view.

\subsection{Picard--Fuchs relations for hyperelliptic equations.}
The calculation in the previous section gives a rigorous normal form
for the bands of continuous spectrum in a neighborhood of the origin
in the spectral plane in terms of the derivatives with respect to
certain parameters of certain moments of the traveling wave
solution. In the case of a potential $F(w^2)$ that is polynomial in
the field $w$ these conserved
quantities can be expressed in terms of hyperelliptic integrals. Since there are only a finite number of such
differentials the derivatives of such integrals with respect to the
parameters may be expressed in terms of linear combinations of the
integrals themselves. These relations are known as the Picard-Fuchs
equations. 
It is preferable from both an analytic
point and a numerical point of view to use the Picard-Fuchs relations
to ultimately express all derivatives in terms of various moments of
the solutions, so we give a brief description of the necessary theory
here.

The classical Picard-Fuchs equation considers integrals of the form 
\[
J_k =\frac{1}{2}\oint_\gamma \frac{z^k dz}{\sqrt{\alpha_0 + \alpha _1
    z + \alpha_2 z^2 + \ldots \alpha_d z^d}}=\frac12\oint_\gamma
\frac{z^k dz}{\sqrt{P(z,{\bf \alpha})}} \qquad k\geq 0,
\]
where $P(z)$ is a real polynomial with degree $d$ with the property that
there is an interval $[a,b]$ on the real axis where $P> 0$ on
$(a,b)$ and $P$ has simple zeroes at $a$ and $b$, and $\gamma$ is a
fundamental cycle enclosing the real interval $[a,b]$. For the sake of
simplicity it is easiest to assume that {\em all of the roots} are simple, not
just $a,b$, but this can be relaxed somewhat. The following facts \cite{Fuchs,givental,moll} are well-known:    
\begin{enumerate}
    \item[Fact 1] There are at most $(d-1)$ such quantities: for $k>d-2$ every such quantity $J_k$ can be expressed as a linear combination of the $J_j$ with $j \leq (d-2)$, with coefficients rational in $\alpha_j$
    \item[Fact 2] The derivatives $\frac{\partial J_j}{\partial \alpha_k}$ can be expressed as a linear combination of the $J_j ~~~ j \in \{0\ldots d-2\}.$ 
    \item[Fact 3] The quantity $\frac{\partial J_j}{\partial \alpha_k}$ is a function of only $j+k$, and thus $\frac{\partial J_j}{\partial \alpha_k}=\frac{\partial J_k}{\partial \alpha_j}$.
\end{enumerate}
The second fact is what is usually known as the Picard Fuchs equations. To
see the first we note that, from integration by parts,  we have the identity
\[
\oint_\gamma \frac{P'(z,\alpha)}{\sqrt{P(z,\alpha)}} dz = 0,
\]
which gives 
\[
\sum_{j=1}^{d} j \alpha_{j} J_{j-1} = 0,
\]
which gives $J_{d-1}$ in terms of $J_0,J_1\ldots J_{d-2}$. More
generally we have that (again from integration by parts) 
\[
\oint_\gamma \frac{z^k P'(z,\alpha)}{\sqrt{P(z,\alpha)}} dz = -\oint_\gamma 2 k
z^{k-1} \sqrt{P(z,\alpha)} dz = -\oint_\gamma \frac{2 k z^{k-1} P(z,\alpha)}{\sqrt{P(z,\alpha)}} dz
\]
or, equivalently, that 
\[
\oint_\gamma \frac{Q(z)}{\sqrt{P(z,\alpha)}}dz =0
\]
where $Q(z) = z^k P'(z,\alpha) + 2 k z^{k-1} P(z,\alpha).$ This gives
a relation 
\[
\sum_{j=1}^d (j+2k) \alpha_j J_{j+k-1} =0
\] (for $k\ge0$) so we can solve for $J_{k+d-1}$ in terms of $J$ with smaller
arguments. We can apply this repeatedly (if tediously) to always
express any such integral in terms of $J_j$ with $j\leq d-2$.

To see the second fact, we first define $D_k$ as follows
\[
D_k = -\frac{1}{4} \oint \frac{z^k}{(P(z,\alpha))^{\frac32}} dz \qquad
k \in 0 \ldots d-2  
\]
and note that $\frac{\partial J_j}{\partial \alpha_k}= D_{j+k}$. Since all $J_k$ are expressible as
linear combinations of $\{J_j\}_{j=0}^{d-2}$, the largest subscript
that we need is $\frac{\partial J_{d-2}}{\partial \alpha_d}=D_{2d-2}.$
The goal is to derive $2d-1$ linear equations relating $\{D_j\}_{j=0}^{2d-1}$ with
$\{J_k\}_{k=0}^{d-2}$, and to solve these linear equations. 

These equations are of two types. The prototypical first type arises as follows. We begin with
\begin{multline}
\alpha_0 D_0 + \alpha_1 D_1 \ldots \alpha_d D_d 
= -\frac14\oint_\gamma \frac{\sum \alpha_i z^i}{(P(z,\alpha))^{\frac32} } dz\\ 
=-\frac14 \oint \frac{dz}{\sqrt{P(z,\alpha)}} 
= -\frac12 J_0,
\end{multline}
and similarly we have (for $0\leq k\leq d-2$)
\begin{multline}
\sum_{j=0}^d \alpha_j D_{k+j}  = -\frac14 \oint_\gamma \frac{\sum \alpha_j z^{j+k}}{(P(z,\alpha))^{\frac32} } dz \\
=-\frac14 \oint \frac{z^k dz}{\sqrt{P(z,\alpha)}} = -\frac12 J_k \qquad k\in 0\ldots d-2
\end{multline}
The prototype for the second class of equations is the identity
\[
-\frac{1}{4} \oint_\gamma \frac{z^k P'(z,\alpha)}{(P(z,\alpha))^{\frac32}}dz
=-\frac{1}{4} \oint_\gamma  \frac{2 k z^{k-1}}{\sqrt{P(z,a)}} dz 
\]
which follows from integration by parts. At the levels of the
integrals $D_j,J_k$ it can be expressed as 
\begin{equation}
\sum j \alpha_j D_{j+k}  = -\frac12 2 k J_{k-1} \qquad k \in 0\ldots d-1.
\end{equation}
There are $d-1$ equations of the first type and $d$ equations of the
second. This gives a set of $2d-1$ equations in $2d-1$ unknowns $D_0 \ldots D_{2d-2}$, which can be solved to give $\{D_j\}_{j=1}^{2d-1}$ in terms of $\{J_k\}_{k=1}^{2d-1}$. 
The $(2d-1)\times (2d-1)$ matrix mapping $\{D_j\}_{j=1}^{2d-1}$ to $\{J_k\}_{k=1}^{2d-1}$ is the 
Sylvester matrix of the polynomials $P(z,\alpha)$ and
$P'(z,\alpha)$. The Sylvester matrix is singular if and
only if the two polynomials have a common root, which in this case is equivalent to $P$ having a root of higher multiplicity. As we have assumed that the roots of $P(z,\alpha)$ are
simple the Sylvester matrix is invertible, and the derivatives $D_j$
can be expressed as linear combinations of the moments $J_k$. 

In order to have a periodic traveling wave solution it suffices to
have a real interval $z\in (a,b)$ on which $P>0$ with simple roots at
$a,b$.
In the case where $P$ has a double root at some $c\neq a,b $ there is a smaller degree linear system that one can solve. This occurs in the case of the quintic NLS, where for certain parameter values the elliptic function solutions reduce to rational trigonometric or hyperbolic trigonometric solutions. Modifications for this case are reasonably straightforward,
but we do not detail them here. 

We need a very slight generalization of the classical Picard-Fuchs
equations, as we will actually need to consider integrals of the form
$J_{-1}=\frac12 \oint_\gamma \frac{dz}{z\sqrt{P(z,\alpha)}}$, where
  $\gamma$ is a curve that does not enclose the origin. This
  necessitates including an extra derivative, $D_{-1} = -\frac14
  \oint_\gamma \frac{dz}{z(P(z,\alpha))^{\frac32}}.$   We need to
  include an additional equation of the first type, leading to a
  system of $2d$ equations in $2d$ unknowns of the form ${\bf S}{\bf d} = {\bf j}$ where 
\begin{align}
&{\bf S} = \left(\begin{array}{cccccccc}
\alpha_0 & \alpha_1 & \alpha_2 & \ldots & \alpha_d & 0 & \ldots & 0 \\
0& \alpha_0 & \alpha_1 & \alpha_2 & \ldots & \alpha_d & \ddots & \vdots  \\
\vdots & \ddots & \ddots & \ddots & \ddots & \ddots & \ddots & 0 \\
0& \ldots & 0 & \alpha_0 & \alpha_1 & \alpha_2 & \ldots & \alpha_d \\
\\
0 & \alpha_1 & 2\alpha_2 &  \ldots & d \alpha_d & 0 & \ldots & 0\\
0 & 0 & \alpha_1 & 2\alpha_2 & \ldots & d \alpha_d & \ddots & \vdots \\
\vdots & \ddots & \ddots & \ddots & \ddots & \ddots & \ddots & 0 \\
0& \ldots & 0 & 0 & \alpha_1 & 2\alpha_2 &  \ldots & d\alpha_d \\
\end{array}\right)\label{eqn:Syl} \\
&{\bf d}= \left(\begin{array}{c}
D_{-1} \\ D_0 \\ D_1 \\ D_2 \\ \vdots \\ D_d \\ D_{d+1} \\ \vdots \\ D_{2d-3} \\ D_{2d-2} 
\end{array}\right), \hskip 45pt  
{\bf j}=\left(\begin{array}{c}
J_{-1} \\ J_0 \\ J_1 \\ \vdots \\ J_{d-2} \\ 0 \\ 2 J_0 \\ 4 J_1 \\ \vdots \\ 2 (d-1) J_{d-2} 
\end{array}\right) 
\end{align}
The lower-right $2d-1\times 2d-1$ block of ${\bf S}$ is the usual Sylvester matrix, so (assuming $\alpha_0\neq 0$)  this determinant will also vanish if and only if $P(z,\alpha)$ has a root of higher
multiplicity.

In the case of the nonlinear Schr\"odinger equation after quadrature the traveling waves reduce to  
\[
A_y^2 = \frac{-\kappa^2}{A^2} + 2E  -\w A^2  -\zeta F(A^2). 
\]  

The period $T$ of such a solution is given by 
\[
T = \oint\frac{dA}{\sqrt{\frac{-\kappa^2}{A^2} + 2E -\w A^2 - \zeta F(A^2) }},
\]
the mass $M$ by  
\[
M = \oint\frac{A^2 dA}{\sqrt{\frac{-\kappa^2}{A^2} + 2E -\w A^2 - \zeta F(A^2) }},
\]
and the quasi-momentum is $\eta = \kappa \chi$ (for $c=0$),
where $\chi$
is given by an integral of the form
\[
\chi =\oint\frac{dA}{A^2\sqrt{\frac{-\kappa^2}{A^2} + 2E -\w A^2 - \zeta F(A^2) }}
\]
It is convenient to make the change of variables $z=A^2$ to give 
\begin{eqnarray*}
T &=& \frac12 \oint_\gamma \frac{dz}{\sqrt{-\kappa^2 + 2 E z -
   \w z^2 -\zeta  z F(z)}} =J_0\\
M &=& \frac12 \oint_\gamma \frac{z dz}{\sqrt{-\kappa^2 + 2 E z -
   \w z^2 - \zeta z F(z)}} =J_1\\
\chi &=&  \frac12 \oint_\gamma \frac{ dz}{z\sqrt{-\kappa^2 + 2 E z -
   \w z^2 - \zeta z F(z)}} =J_{-1},
\end{eqnarray*}
and so if $F(z)$ is polynomial (and can be assumed to have no constant or linear term) the quantities $T,M,\chi$ (and $U = \int_0^T F(A^2) dy$ in the case of transverse perturbations) are governed by the Picard--Fuchs relations under the identification  
$\alpha_0=-\kappa ^2,\ \alpha_1=2E,\ \alpha_2=-\w,$ and  $\alpha_j$ determined by the coefficients of $F(z)$ for $j>2$.  For the cubic NLS, for instance, we have $\alpha_3=-\zeta/2$ and $\alpha_j=0$ for $j\geq 4$, while for the quintic NLS we have $\alpha_3=0,\alpha_4=-\zeta/3$ and all higher  $\alpha$ zero. 

Assuming a polynomial potential  $F(z)$ the Picard--Fuchs relations hold, and all derivatives of moments of the traveling wave profile can be expressed in terms of moments $J_kj$ of the traveling
wave solution. This is perhaps most satisfactory in the cubic case, $F(z) =
z^2/2$, where derivatives of $\eta,T,M$ can be expressed as linear
combinations of $\eta,T,M$ only, but for any polynomial potential $F(z)$ the derivatives of $\eta,T,M$ can be expressed as linear combinations of a finite number of moments of the traveling wave profile.  Since all of the inner products, genericity conditions, etc are expressible in terms of $\eta,T,M$ and their derivatives this means that ultimately all of these quantities are expressible in terms of a finite collection of moments of the traveling wave profile.  This means that quantities arising in the computation of the normal form such as $\sigma = T_E \eta_\kappa - T_\kappa \eta_E $ can be expressed as a homogeneous quadratic polynomial in a finite number of moments of the traveling wave profile. Similarly all three by three determinants can be expressed as a cubic polynomial in the moments of the traveling wave profile, etc.  In addition to being useful analytically it is preferable from the point of view of numerics to express the quantities of interest directly in terms of moments of the solution, as it avoids errors introduced by numerical differentiation of the moment integrals. We will not give explicit formulae for all of the quantities involved in the stability calculation, as many of them are quite cumbersome, but it is worth mentioning the quantity $\sigma = T_E \eta_\kappa - T_\kappa \eta_E$, as the vanishing of $\sigma$ indicates that the null-space of ${\mathcal L}$ has dimension greater than two. In the cubic case one can use the Picard-Fuchs relations to compute $\sigma$ as 
\begin{align*}
\sigma = -\frac{\kappa^2 \zeta^2}{16 \det({\bf S})} \left( 3 \kappa^2 T^2- 4 E M T + \omega M^2 \right).    
\end{align*}
where ${\bf S}$ is the $2d \times 2d$ Sylvester matrix from Equation (\ref{eqn:Syl}) with $d=3$ and $(\alpha_0,\alpha_1,\alpha_2,\alpha_3) = (-\kappa^2,2E,-\omega,-\zeta/2)$.  
Similarly for the quintic case the quantity $\sigma$ can be expressed as 
\[
\sigma = -\frac{\kappa^2\zeta^2}{16 \det({\bf S}) }\left(4 \kappa^4 T^2 + (4 \kappa^2 \omega - 9 E^2) T J_2 + \omega^2 J_2^2\right)
\]
where $J_2$ is the second moment of the square amplitude, $J_2 = \int_0^T |A(y)|^4 dy = \frac{1}{2} \oint \frac{z^2 dz }{\sqrt{\kappa^2 + 2 E z  - \w z^2 -\zeta z^4/3}}$, and ${\bf S}$ is the $2d \times 2d$ Sylvester matrix from Equation (\ref{eqn:Syl}) with $d=4$ and $(\alpha_0,\alpha_1,\alpha_2,\alpha_3,\alpha_4) = (-\kappa^2,2E,-\omega,0,-\zeta/3)$. 

We also remark that for polynomial nonlinearities the quantity $U=\int_0^TF(A^2)dy$, which arises in the case of transverse perturbations, can be expressed in terms of lower moments. In the cubic case this gives the relation $3 \zeta U + \omega M - E T=0$ and in the quintic case the relation  $6 \zeta U + \omega M - E T = 0$.

\section{Numerical Results}

In this section we  present some numerical
simulations to support our results, and additionally find some trigonometric solutions to the quintic NLS that do not seem to have appeared in the literature.

There are two main types of numerical simulations in this section. The first computes the roots of the normal form polynomials given in Equations (\ref{eqn:quartic}) and (\ref{eqn:transquartic}), along with associated quantities such as the genericity conditions.  First the Picard--Fuchs relations are used to express all derivatives of moments in terms of the moments themselves. The moments can be expressed in the form 
\[
J_j = \int_{r_1}^{r_2} \frac{z^j dz}{\sqrt{\sum_{k=0}^d \alpha_k z^k}} \qquad j \in -1,0,\ldots d-2,
\]
where $r_1,r_2$ are two simple real positive roots of the denominator. Note that moments with $j \geq d-1$ can always be expressed in terms of lower moments.  These integrals have an integrable  singularity at the roots $z=r_{1,2}$, behaving like $(z-r_j)^{-\frac12}$, so it is more convenient and numerically more well-behaved to make the change of variables $ z = \frac{r_1+r_2}{2}+ \frac{r_1-r_2}{2} \sin \Phi$, which leads to a regular integral of the form 
\[
J_j = \int_{-\frac{\pi}{2}}^{\frac{\pi}{2}} \frac{(\frac{r_1+r_2}{2}+ \frac{r_1-r_2}{2} \sin \Phi)^j}{\sqrt{Q(\frac{r_1+r_2}{2}+ \frac{r_1-r_2}{2} \sin \Phi)}}d \Phi
\]
where $Q(z)$ is the polynomial $Q(z) = \sum_{k=0}^d \alpha_k z^k/(r_2-z)(z-r_1).$
 
As a check of the modulation theory predictions we also compute the full $L_2({\mathbb R})$ spectrum of ${\mathcal L}$ as follows. We use direct numerical integration of the traveling wave differential equations to find the amplitude $A(y)$ and the phase $S(y)$. We numerically expand $A(y), S_y(y)$ and $S_{yy}(y)$ in $T$-periodic Fourier series. Usually this is done using $N=20$ Fourier modes but for certain parameter values (typically those close to a trivial phase solution) we retain more Fourier modes.  We next compute the Fourier expansions for $A,S_y,S_{yy}$ and use this to find a finite dimensional Galerkin truncation of ${\mathcal L}(\mu).$ We then find the eigenvalues of the finite dimensional approximation to ${\mathcal L}(\mu)$ or ${\mathcal L}(k)$ for a collection of $\mu \in (-\frac{\pi}{T},\frac{\pi}{T}]$ or $k \in (0,\infty)$.  This method gives a numerical approximation to the spectrum for all values of $\mu$ or $k$, rather than just giving the asymptotics for small $\mu$ or $k$, but is much more computationally intensive. 

We present numerical experiments studying the linearized stability of periodic traveling wave solutions to the cubic and quintic NLS equations in  both the focusing and the defocusing cases. We consider the stability to both modulations in the longitudinal direction (in $x$, or $y=x+ct$) as well as stability of $y$ dependent solutions to modulations in a transverse $z$ variable:
\[
i w_t = w_{xx} \pm w_{zz} + \zeta f(|w|^2)w. 
\]
Note that in the case of transverse perturbations it is not necessary to consider the elliptic and the hyperbolic cases---respectively the $+$ and $-$ signs in the above equation---independently.  The modulation theory predicts that the eigenvalues bifurcating from the  origin are approximately linear in $z$-directional Fourier variable $k$ for small $k$. Since the stability problems for the elliptic and hyperbolic cases are formally related by the map $k \mapsto i k$ it follows that to leading order the eigenvalues of the linearization of the 2D hyperbolic NLS are $i$ times the eigenvalues of the 2D elliptic NLS. 

The case of the stability of the cubic NLS to longitudinal perturbations has previously been treated by Deconinck and collaborators using the exact integrability of cubic NLS. This provides a nice touchstone, and the results here are in agreement with the results of the integrable theory. All remaining cases: the stability of the periodic solutions to the cubic NLS as well as all results for the quintic NLS are to our knowledge new.    

Table \ref{tbl:summary} summarizes the results we find in this section, indicating the possible dimensions of the unstable manifold for perturbations in the longitudinal and transverse directions. To be clear what is meant here: the table gives the number of roots of the modulation equations with a non-zero real part.  If this is non-zero that solution is unstable but if it is zero one cannot rigorously conclude stability -- the spectral curve could be tangent to the imaginary axis without being purely imaginary,  or there could be a finite wavelength (non-modulational) instability. We also only indicate possibilities that occur on open sets in the parameter space. In the cubic focusing case there is a curve along which two of the roots have zero real part and a part of the spectrum is tangent to the imaginary axis. 

Note that in the focusing quintic case, as we will describe, there are two ways to parameterize the traveling wave, depending on how many roots of $A_y^2=0$ are real valued.  In general, the unstable manifold of transverse perturbations is usually two dimensional except in a few special cases. As expected, both defocusing cubic and defocusing quintic appear longitudinally stable\footnote{Of course this method can never establish stability, only instability, as there may be eigenvalues with non-zero real part located away from a neighborhood of the origin.}. Finally, both the focusing cubic and focusing quintic appear to always be modulationally unstable to longitudinal perturbations.
\begin{table}
\begin{tabular}{cc|c|c|c|}\cline{3-5}
& & & Transverse & Transverse \\
& & Longitudinal & Elliptic & Hyperbolic \\ \cline{1-5}
\multicolumn{1}{ |c }{\multirow{2}{*}{Cubic} } &
\multicolumn{1}{ |c| }{Focusing} & 4D & 2D & 2D     \\ \cline{2-5}
\multicolumn{1}{ |c  }{}                        &
\multicolumn{1}{ |c| }{Defocusing} & 0D & 2D & 2D    \\ \cline{1-5}
\multicolumn{1}{ |c  }{\multirow{3}{*}{Quintic} } &
\multicolumn{1}{ |c| }{Focusing (4)} & 2D, 4D & 0D, 2D & 2D, 4D \\ \cline{2-5}
\multicolumn{1}{ |c  }{}                        &
\multicolumn{1}{ |c| }{Focusing (2)} & 2D, 4D & 2D & 2D \\ \cline{2-5}
\multicolumn{1}{ |c  }{}                        &
\multicolumn{1}{ |c| }{Defocusing} & 0D & 2D & 2D \\ \cline{1-5}
\end{tabular}
\vspace{2pt}
\caption{This table indicates the possible dimensions of the unstable manifolds in all the cases that we discuss in this work. As noted in the text we are only considering possibilities that occur on open sets of parameters.  }\label{tbl:summary}
\vspace{-5pt}
\end{table}

We begin with the cubic case. To briefly summarize the results of the integrable theory: the periodic traveling waves of the defocusing cubic NLS are spectrally stable, while the periodic traveling waves  of the  focusing cubic NLS are unstable, and the linearized operator generically has four spectral curves emerging from the origin in the spectral plane.

\subsection{Focusing cubic case: stability to longitudinal perturbations. }
In the focusing cubic case the spectral curves have been computed by
Deconinck and Segal \cite{DeconinckSegal}. 
In the derivation, we have four parameters: $E,\ \kappa,\ \omega,\ \zeta$ (having already eliminated dependence on $\theta_0$, $y_0$, and $c$), but we can study all possible situations by looking at a restricted 2D parameter space.  First, one can scale $\zeta$ into the amplitude of the solution.  Since this method puts no restrictions on the specific traveling wave solution, without loss of generality we can choose one value for $\zeta$.  In addition, one of the three remaining parameters can be scaled into space-time.
We have elected to compute the spectrum
for the same parameter values as were depicted in that paper, to
facilitate comparison. (Note, however, that our eigenvalues are
exactly twice those in \cite{DeconinckSegal} due to the different scaling chosen
in that paper.) In these numerical experiments we first calculated the
spectrum of the linearized operator numerically using a spectral
method: the periodic functions $A(y),S_y(y)$ were expanded in Fourier
series, allowing the linearized operator ${\mathcal L}(\mu)$ to be
expressed as a sum of diagonal and Toeplitz type operators. We then
truncated to a finite number of Fourier modes and computed the
spectrum of the resulting matrix for a sequence of different values of
$\mu$. The traveling wave solutions in Deconinck-Segal are
parameterized by the elliptic modulus $k$ and the parameter $b$, so as
to be expressible in terms of standard elliptic functions, the
relation between $(k,b)$ and the parameters considered here being
\begin{align*}
  &\kappa = \sqrt{b(1-b)(b-k^2)} & \\ 
  &E = b(1+k^2)-\frac{3 b^2}{2} - \frac{k^2}{2} & \\
  & \omega = 1 + k^2 - 3 b & \\
  &\zeta = 2.&
\end{align*}

When parameterized in this way, the cubic $P_3(A^2)$ defined by 
\begin{equation}
A^2(A_y)^2=P_3(A^2)=-\kappa^2+2EA^2-\w A^4-\zeta A^6/2\label{eqn:cubicA2}
\end{equation} 
has roots at $A^2=b-1,\ b-k^2$, and $b$, as shown in Figure \ref{fig:cubicpoly}(a).
\begin{figure}
\includegraphics[width=0.32\textwidth]{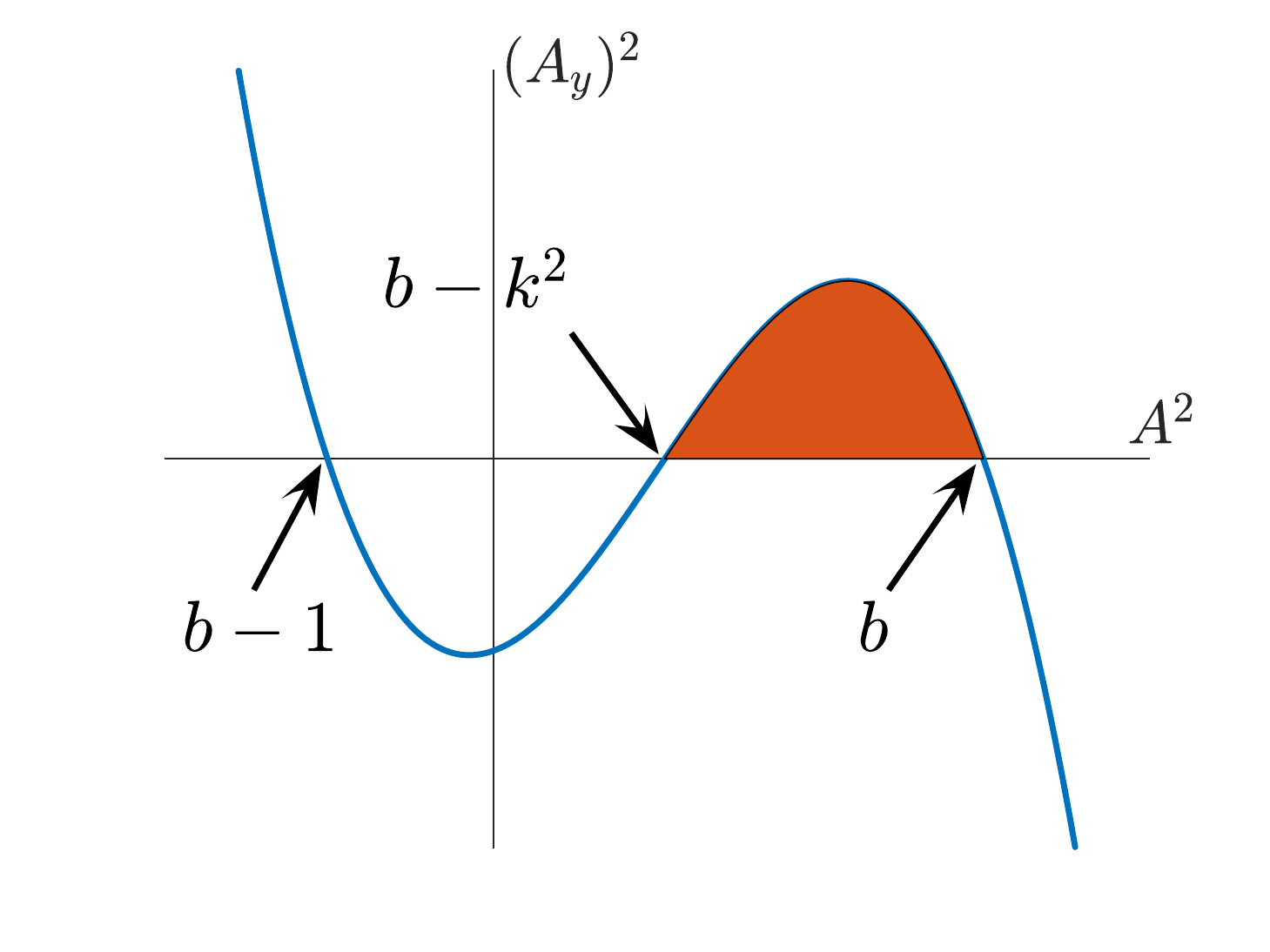}
\includegraphics[width=0.32\textwidth]{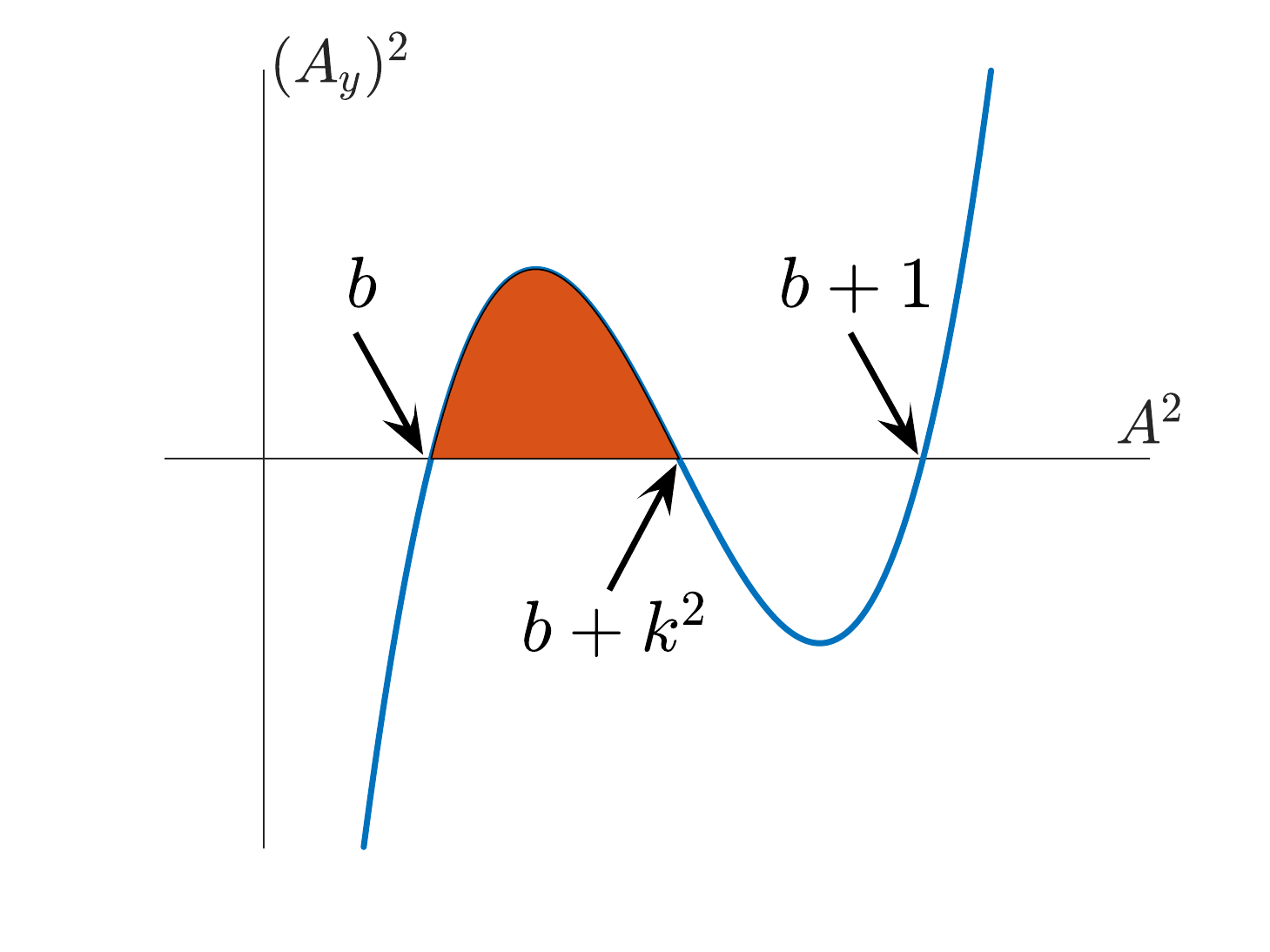}
\includegraphics[width=0.32\textwidth]{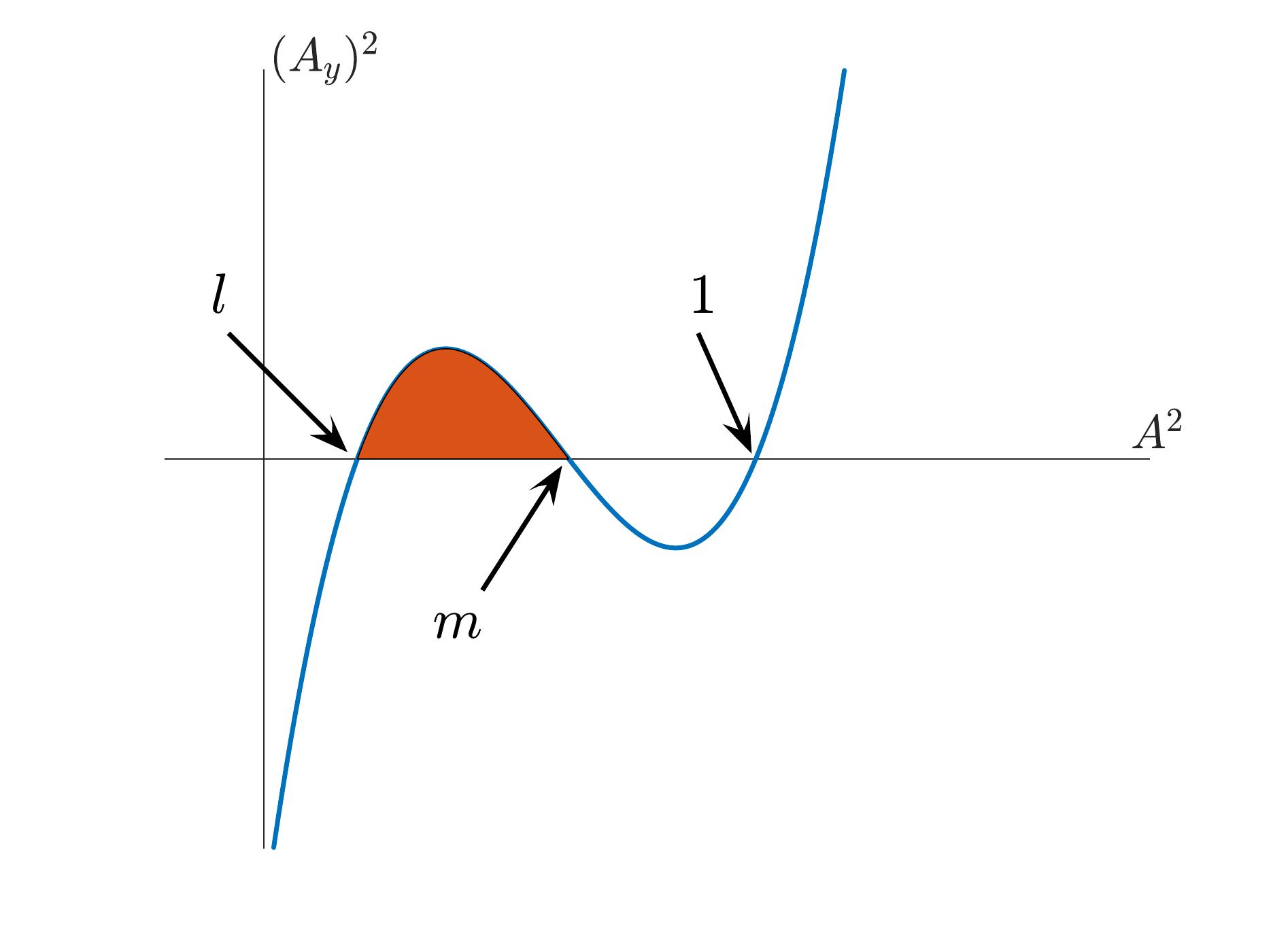}
\caption{We illustrate how we parameterize the cubic NLS by plotting the polynomial $A^2(A_y)^2=P_3(A^2)$ from Eq.~\eqref{eqn:cubicA2}.  On the left is the parameterization for the focusing case.  The middle and right are both for the defocusing case, middle with parameters as used in \cite{BottmanDeconinckNivala}, right with alternative parameters.}\label{fig:cubicpoly}
\end{figure}
It is worth noting here that, from the point of view adopted in this
paper, the trivial-phase solutions, {\em particularly} the $\cn$
type solutions, are a somewhat singular perturbation of the
non-trivial phase solutions. Note that for $\kappa
\neq 0$ the amplitude $A(y)$ is never zero, and the solutions are quasi-periodic with a non-trivial phase.  When $\kappa=0$, on the other hand, the phase is trivial and the amplitude may be zero. The way that these are reconciled is that when $\kappa$ is very small the phase is very flat except in a neighborhood of the region where the amplitude becomes small, where the phase jumps by $\pi$. The result is that solutions close to the trivial phase ones
require some care to evaluate numerically, as the phase has sharp transitions and may require many Fourier modes to adequately resolve.    

The first graph in Figure \ref{fig:DeconinckSegal} represents the same parameter
values as Figure 7a in Deconinck-Segal, $(k,b)=(0.65,0.423)$, the so-called ``two
single-covered figure eights''. Note that this is quite close to a
cnoidal solution ($(k,b)=(.65,.4225)$ is cnoidal, for instance).
This was solved with $N=512$ Fourier modes. The
dots [blue online] represent discretely sampled approximation to the continuous
curves of spectrum, while the stars [red online]
represent the four roots of the quartic Eq.~\eqref{eqn:quartic}, which gives the slope
of spectral curves in a neighborhood of the origin. We see very good
agreement between the numerical simulations and the asymptotic theory:
the two roots in the upper half-plane correspond to the outer figure
eight in the upper half-plane, while two roots in the lower half-plane
lie on the inner figure eight. (of course, replacing $\mu$ with $-\mu$
flips the upper and lower half-planes and completes the quartets---it
is only for the trivial phase solutions that one gets all four members
of a quartet for a single $\mu$ value.

The remaining three pictures were generated with the same parameter
values as in Figure 7(b-d) in \cite{DeconinckSegal}: they are (in order) the non-self-intersecting
butterfly $(k,b)=(0.9,0.95)$, the triple figure eight
$(k,b)=(0.89,0.84)$ and the self-intersecting
butterfly $(k,b)=(0.9,0.85)$. Each of these required only $N= 20$ Fourier
modes. In each case we find spectral pictures in good qualitative agreement with
the results of both Deconinck-Segal and our modulation calculation.

\begin{figure}
  \begin{tabular}{|c|c|} \hline
    \includegraphics[width=0.45\textwidth]{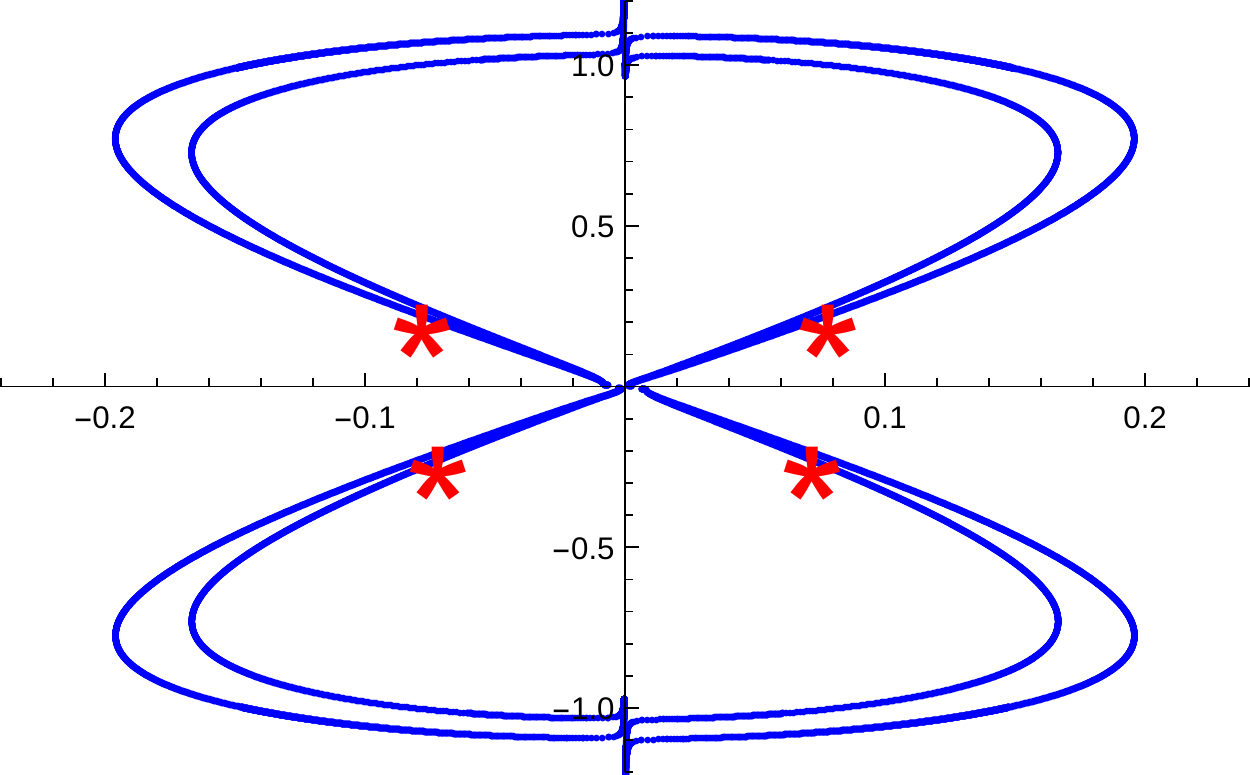} &
    \includegraphics[width=0.45\textwidth]{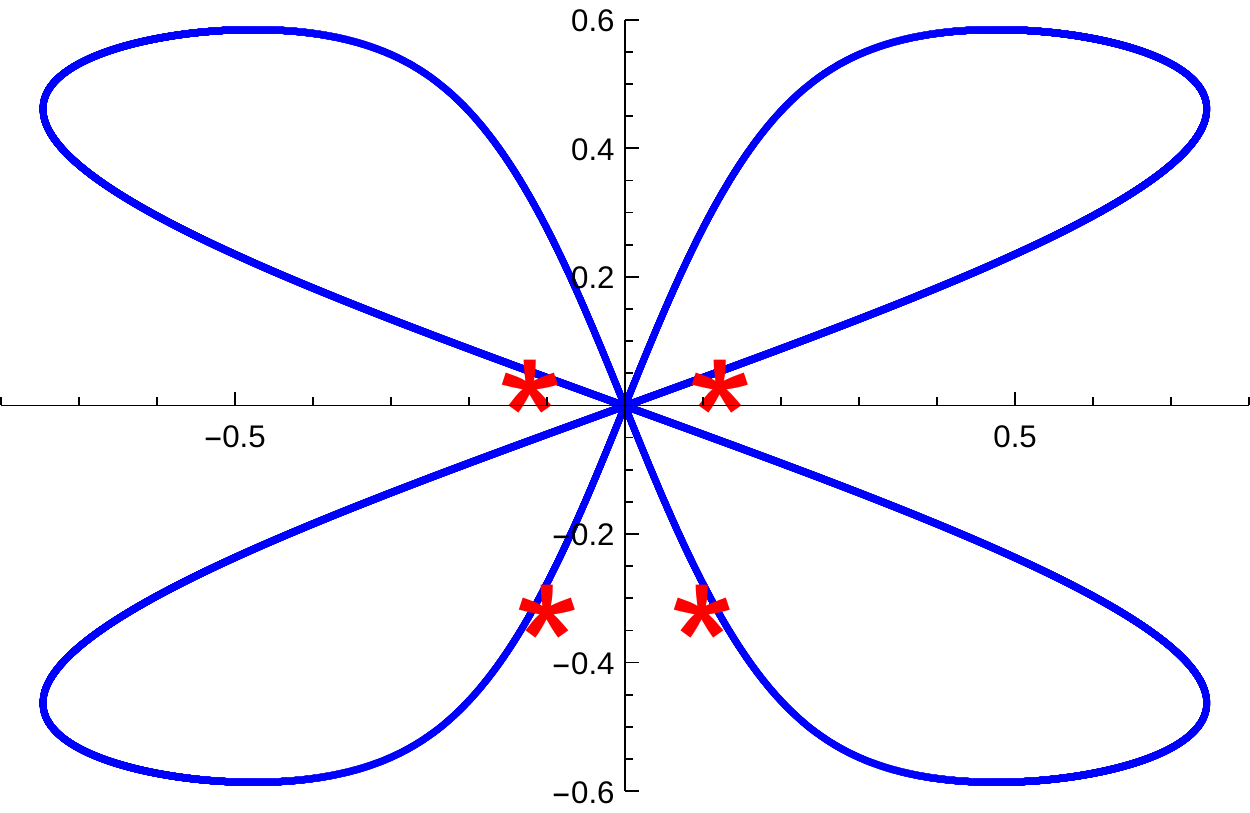}  \\ \hline
    \includegraphics[width=0.45\textwidth]{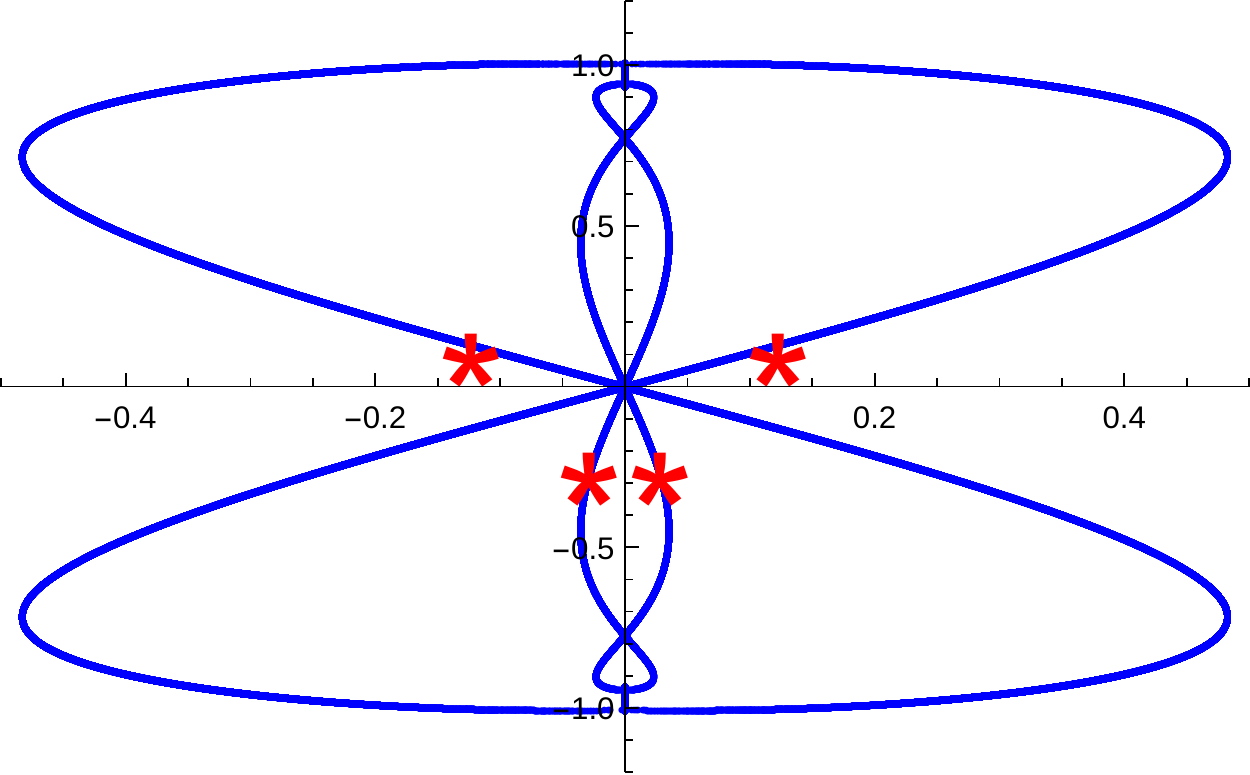} & 
    \includegraphics[width=0.45\textwidth]{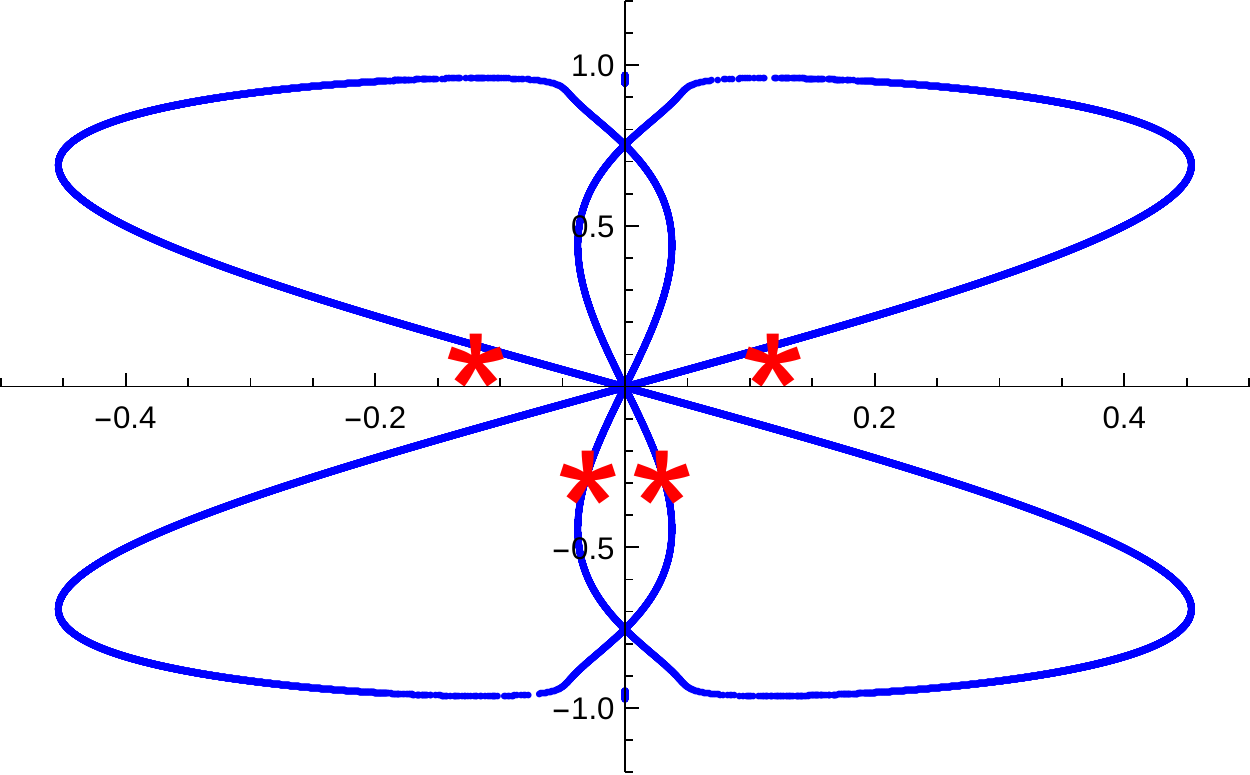}  \\ \hline
  \end{tabular}
\caption[]{The spectrum of the linearized focusing cubic NLS flow around a traveling
  wave solution for $(k,b) = (.65,.423),\ (0.9,0.95),\ (0.89,0.84)$ and
  $(0.9,0.85)$ (left to right, top to bottom) respectively. The dots [blue online] represent the discretely sampled approximation to the continuous curves of the spectrum, while the stars [red online] represent the four roots of the quartic of Eq.~\eqref{eqn:quartic}.}
\label{fig:DeconinckSegal}
\end{figure}

\subsection{Defocusing cubic case: stability to longitudinal perturbations}

The defocusing case, being unconditionally stable, is somewhat uninteresting, so we will not include any specific numerics in this case, but only discuss briefly how we will parameterize it.  
In order to similarly scale the parameter space for the defocusing cubic NLS, we first consider parameters $k$ and $b$ from Bottman, Deconinck, and Nivala \cite{BottmanDeconinckNivala}:  
\begin{align*}
  &\kappa = \sqrt{b(1+b)(b+k^2)} &\\
  &E = b(1+k^2)+\frac{3 b^2}{2} + \frac{k^2}{2} & \\
  & \omega = 1 + k^2 + 3 b & \\
  & \zeta =-2 .& 
\end{align*}
With this paramterization, Eq.~\eqref{eqn:cubicA2} has roots at $A^2=b,\ b+k^2$, and $b+1$, as shown in Figure \ref{fig:cubicpoly}(b).
Note that in the focusing case, in order to have $\kappa^2>0$, we needed $b>k^2$.  However, in the defocusing case, this restriction of the parameter space is no longer necessary.  Thus, to more effectively cover the whole parameter space, we introduce a different parameterization that is perhaps more intuitive but does not use the elliptic modulus.  Instead, we will parameterize in the following way. First observe that the largest positive real root of the cubic represents the maximum amplitude of the periodic traveling wave.  From the scaling invariance we can rescale the maximum amplitude
to be one, so we can assume that the largest positive root of Eq.~\eqref{eqn:cubicA2}  is one, and denote the other two roots by $l$ and $m$, and without loss of generality set $0<l<m<1$, as shown in Figure \ref{fig:cubicpoly}(c).   With this parameterization, the relation between $(l,m)$ and $(E,\kappa,\w)$ is
\begin{align*}
    &\kappa=\sqrt{lm}\\
    &E=\frac{1}{2}(l+m+lm)\\
    &\w=1+l+m\\
    &\zeta=-2.
\end{align*}

\subsection{Cubic NLS: visualizing the entire parameter space}

One benefit of our method is that it is numerically quite cheap---if one uses the Picard-Fuchs relations then one only needs to compute a fixed number of non-singular integrals and do a bit of linear algebra. For this reason one can efficiently explore the entire parameter space, counting the number of roots of \eqref{eqn:quartic} with nonzero real part.  Most of the remainder of this section will be dedicated to figures indicating instability or apparent lack there-of.  Figure \ref{fig:cubiclong} contains plots over parameter space in the $k$-$b$ plane for the focusing cubic NLS (left) and in the $l$-$m$ plane for the defocusing cubic NLS (right).    Note that in the focusing case, since $\kappa^2=b(1-b)(b-k^2)$ must hold, the parameter space is restricted to $k^2<b<1$, while in the defocusing case, to avoid redundancy the parameter space is restricted to $0<l<m<1$.  Each point in the domain is colored according to the number of roots of \eqref{eqn:quartic} with nonzero real part: Black indicates 4 purely imaginary roots, beige indicates two purely imaginary roots and two roots with nonzero real part, and blue indicates all 4 roots have nonzero real part.  In the focusing case we have modified this slightly by doing thresholding: any eigenvalue with real part less that $10^{-4}$ is assumed to have zero real part. We have done this to illustrate an interesting feature of the parameter space: one can see a faint beige curve across the parameter space  along which two of the eigenvalues have small real parts.  This appears to correspond to the upper boundary of Deconinck and Segal's Region V---see Figure 9 and related discussion in that paper. We observe that in the focusing case, everywhere in the parameter space produces at least two eigenvalues with non-zero real part, while in the defocusing case the eigenvalues are purely imaginary in the entire parameter space.  This is in agreement with previous results.

\begin{figure}
\includegraphics[width=0.48\textwidth]{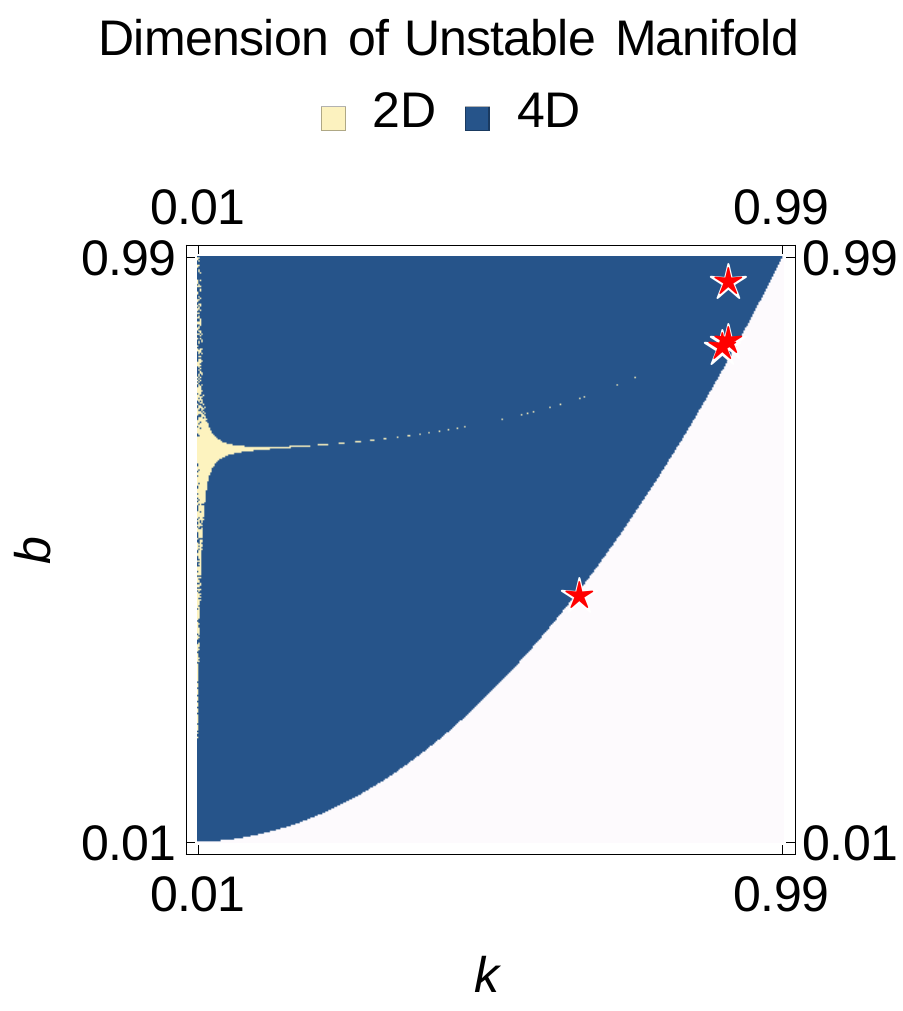}
\includegraphics[width=0.48\textwidth]{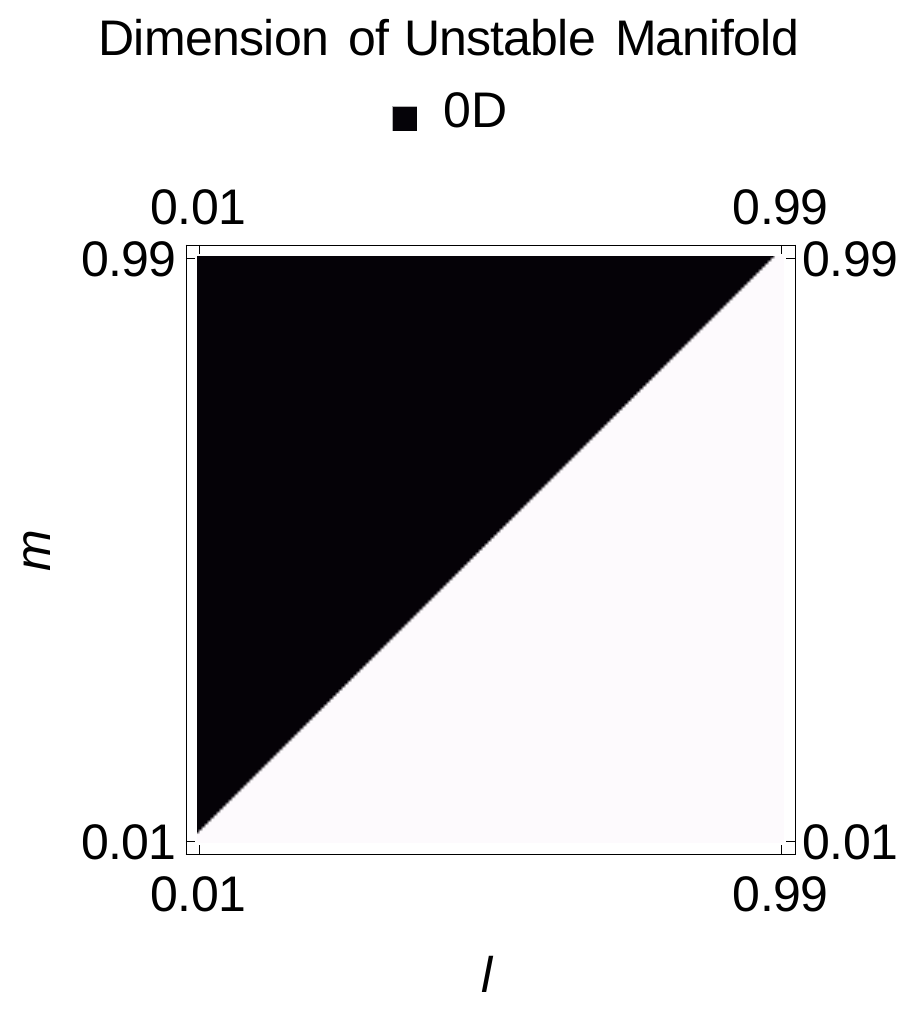}
\caption{The full parameter space colored by the possible dimensions of the unstable manifold for the cubic NLS with longitudinal perturbation.  The left is for the focusing case and the right is for the defocusing case.  The red stars in the plot on the left indicate the four specific cases shown in Figure \ref{fig:DeconinckSegal}.}\label{fig:cubiclong}
\end{figure}

\subsection{Cubic case: stability to transverse perturbations. }

One can also consider the stability of solutions to the two dimensional Schr\"odinger equation
\[   
i w_t = w_{xx} \pm w_{zz} + \zeta  f(|w|^2)w   
\] 
which are periodic in $x$ to perturbations in the transverse $z$ direction. 
  
\begin{figure}
\includegraphics[width=0.48\textwidth]{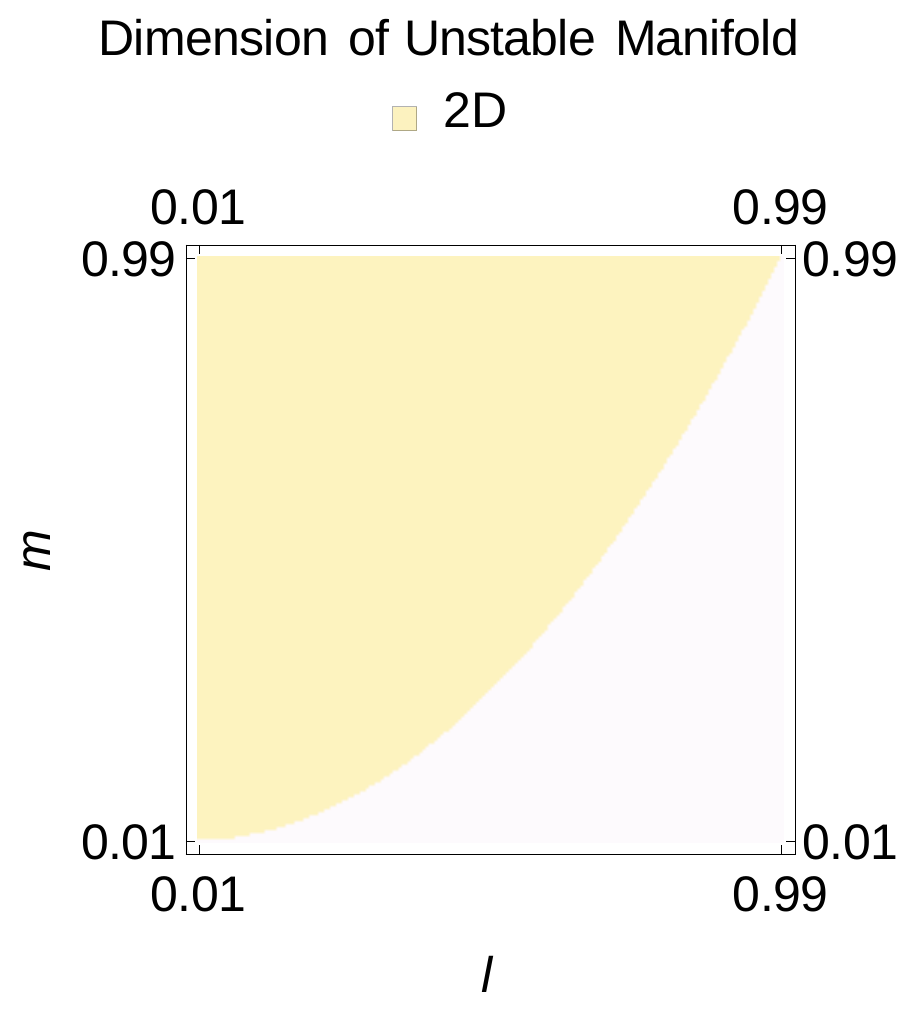}
\includegraphics[width=0.48\textwidth]{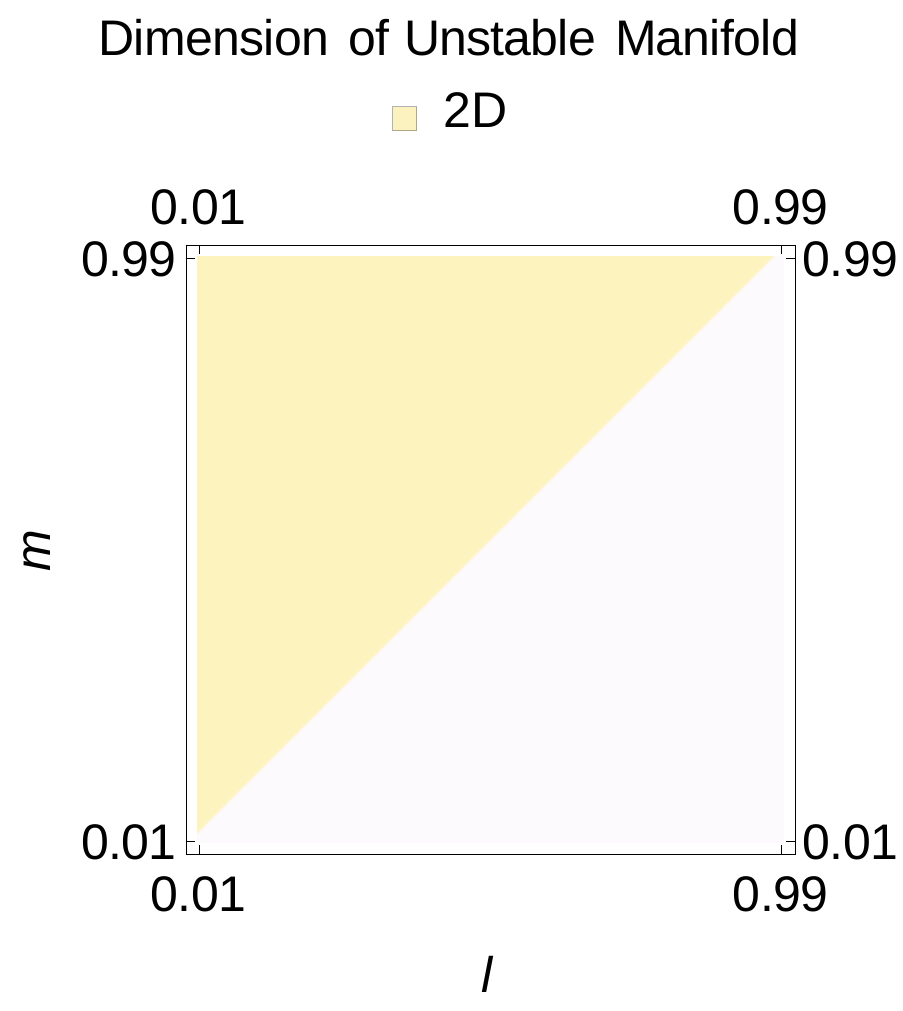}
\caption{The parameter space is colored to indicate the dimension of the unstable manifold, for transverse perturbations in the cubic NLS (left and right are focusing and defocusing, respectively).  Both are the elliptic case; the hyperbolic case looks identical since all parameters return two purely real and two purely imaginary eigenvalues.}\label{fig:cubictrans}
\end{figure}
 We will parameterize in the same way as we did in the above sections for longitudinal stability.  Figure \ref{fig:cubictrans} contains plots in the $k$-$b$ plane for the cubic focusing (left) and defocusing (right), but now coloring based on roots of Eq.~\eqref{eqn:transquartic}.

In this way we depict transverse instability for all parameter values, in both focusing and defocusing regimes, shown here in the elliptic case.  To extend to the hyperbolic case, we note that if $\lambda$ is an eigenvalue of the elliptic case, then $i\lambda$ will be an eigenvalue of the hyperbolic case.  For the cubic NLS (both focusing and defocusing), though it is not explicit in the plots, the entire parameter space produced two purely real eigenvalues and two purely imaginary eigenvalues for the elliptic case, and so the same would be true for the hyperbolic case.

\subsection{Focusing Quintic NLS: Stability to Longitudinal and Transverse Modulations}

As a second example we consider the quintic NLS equation
\[
iw_t = w_{xx} + \zeta |w|^4 w. 
\]
The quintic NLS is $L_2$
(mass) critical on ${\mathbb R}$---the natural scaling invariance
\[
  w(x,t)  \mapsto \gamma w(\gamma^2 x, \gamma^4 t)
\]
leaves the $L_2$ norm invariant. In the focusing case the quintic is the critical exponent for stability of the solitary wave solution for the power-law NLS in one dimension---the solitary wave solution to
\[
  i w_t = w_{xx} + |w|^{2a} w
\]
is stable for $a<2$ and unstable for $a>2$.  For $a=2$ the quintic NLS equation has an extra dynamic symmetry which maps a stationary solution to a collapsing one, the Talanov or lens transformation \cite{T.1970,KT.1985,Fibich} This implies that there are extra elements of the kernel of the linearized operator related to this symmetry, which become real eigenvalues as $a$ is increased. It was shown by Weinstein \cite{W.1985} that the linearized operator for the one dimensional quintic NLS around a soliton solution has two extra elements of the generalized kernel:
\begin{align*}
   & \dim(\ker_0({\mathcal L})) = 2& \\
   & \dim(\ker_1({\mathcal L})) = 2 & \\
   & \dim(\ker_2({\mathcal L})) = 1 & \\
   & \dim(\ker_3({\mathcal L})) = 1. & \\
\end{align*}
This suggests that this would be a case in which to look for failure
of the genericity condition in the periodic problem. We will see that there is in fact a failure of the genericity condition along a co-dimension one curve in parameter space. 

The mixed cubic-quintic NLS in one dimension has also been extensively studied in the optical community, where the quintic term represents the next order correction to the nonlinearity \cite{KS.1998,NMK.2007,CFT.2011}. In this paper we will mostly treat the pure cubic and quintic cases, but the mixed cubic-quintic is a fairly minor generalization of the purely quintic case treated here. 
  
Following the general construction (but with $y_0=\theta_0=c=0$) we have that the traveling waves satisfy 
\[
w = \exp(i(\omega t + S(y))) A(y)
\]
where $A(y)$ and $S(y)$ satisfy
\begin{align}
  & A_y^2 = 2 E -\w  A^2 - \frac{\kappa^2}{A^2} - \zeta A^6/3 & \label{eqn:QuintAmp}\\
  & S_y = \frac{\kappa}{A^2}. & 
\end{align}
Note that the squared amplitude $A^2(y)$ is actually an elliptic
function. To see this note that multiplying Equation
(\ref{eqn:QuintAmp}) by $A^2$ and defining $z(y) = A^2(y)$ gives
\begin{equation}
  \frac{1}{4} z_y^2 = -\kappa^2 + 2 E z -\w  z^2 - \zeta
  z^4/3=P_4(z) 
  \label{eqn:RPoly}
\end{equation}
and so $z(y)$ is expressible in terms of elliptic functions. However, given the form of the polynomial on the right---the general depressed quartic\footnote{Meaning the cubic term is zero.}---it is quite tedious to express the solution in terms of either Jacobi or Weierstrauss elliptic functions. Fortunately, given the viewpoint adopted here, it is not necessary to directly express the traveling wave in terms of the standard elliptic functions. We will, however, need to give an explicit parametrization of the traveling waves. We give a different parametrization than the one chosen by Deconinck and Bottman for the cubic case. There are actually two cases to consider, which work similarly but with minor differences.

If one considers Equation (\ref{eqn:RPoly}) for the square amplitude
in the focusing case then there are periodic
solutions if the  quartic $ P_4(z)= -\kappa^2 + 2 E z - \w  z^2 - \zeta z^4/3$ has either two or four real distinct roots.  For simplicity we will assume that the coefficient of the nonlinear term is $\zeta = 3$.

This polynomial is negative at $z=0$ and negative for $z$ large, so in
order to have a compact interval over which $z$ is positive and real we need at
least two positive real roots in $z$. Since the $z^3$ coefficient is
zero the roots must sum to zero and there can be at most two positive
roots. Thus we have two positive roots and either two negative roots
or two complex conjugate roots with negative real parts.

\subsubsection{Case 1: Four real roots} 

The largest positive real root of the
quartic represents the maximum amplitude of the periodic traveling
wave. From the scaling invariance we can rescale the maximum amplitude
to be one, so we can assume that the largest positive root of the
quartic is one. We will denote the other two roots  by $-l$ and $m$, and we
will assume $-l\leq 0\leq m\leq 1$. Since $P_4(z)$ is a depressed
quartic the roots must sum
to zero, and thus the fourth root is $l-(1+m)$, as shown in Figure \ref{fig:quinticpolys}(a).  The polynomial $P_4(z)$ has two negative roots $-l$ and $l-(1+m)$, so without loss of
generality we can require $l\leq (1+m)-l$ or $l \leq \frac{1+m}{2}$.  Equating
\[
-\kappa^2 + 2 E z -\w  z^2 - z^4 = (1-z)(z-m)(z+l)(z-l+m+1)
\]
we get
\begin{align*}
  & \kappa = \sqrt{ l m (1+m-l)} & \\
  &E = (1+m)(m-l)(l-1)/2& \\
  & \omega = -1 - m - m^2 +(1+m)l - l^2& 
\end{align*}
The parameter space is defined by the inequalities
\begin{align*}
  & 0 \leq m \leq 1 & \\
  & 0 \leq l \leq \frac{m+1}{2}. & 
\end{align*}

\begin{figure}
\includegraphics[width=0.32\textwidth]{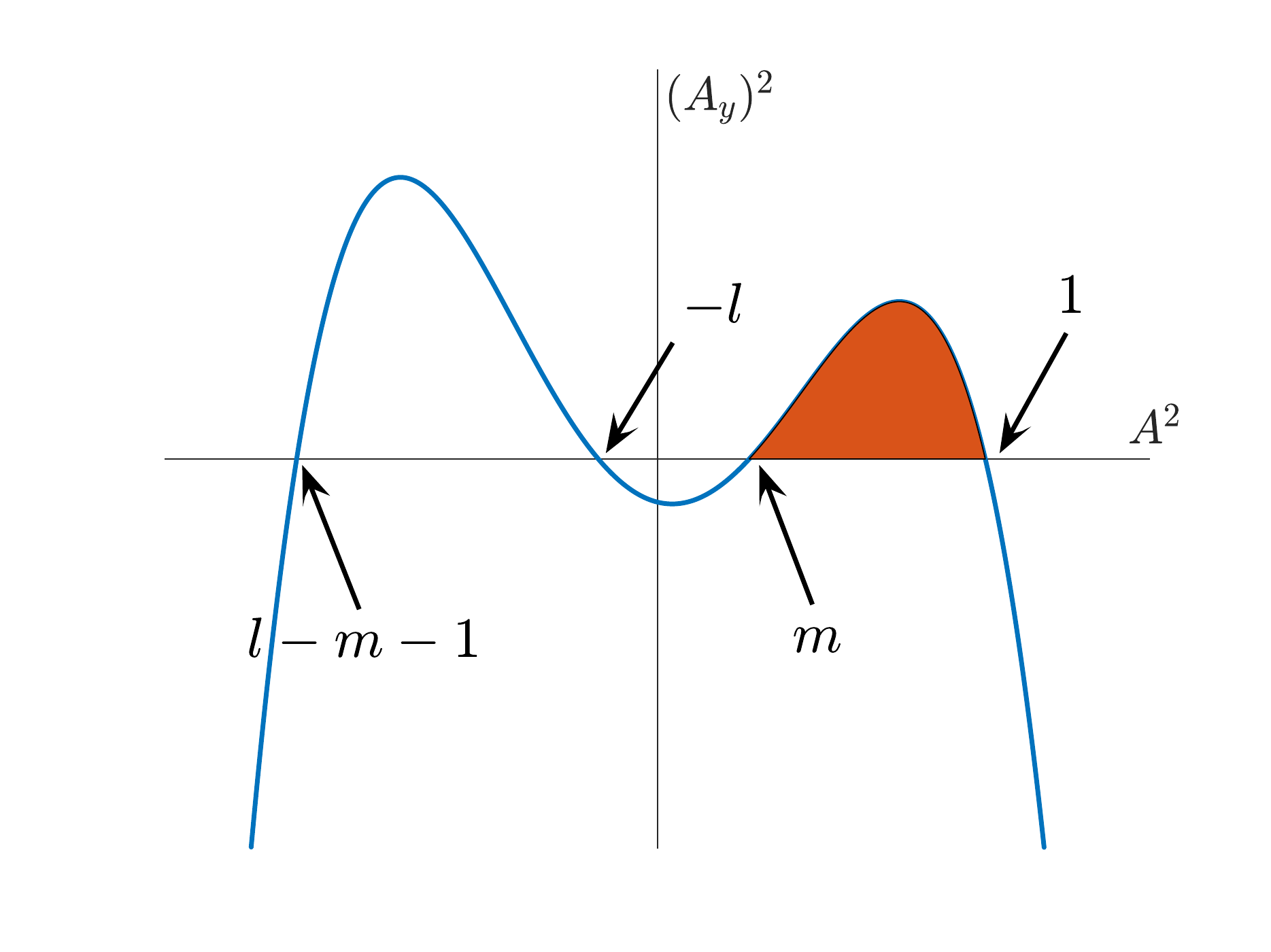}
\includegraphics[width=0.32\textwidth]{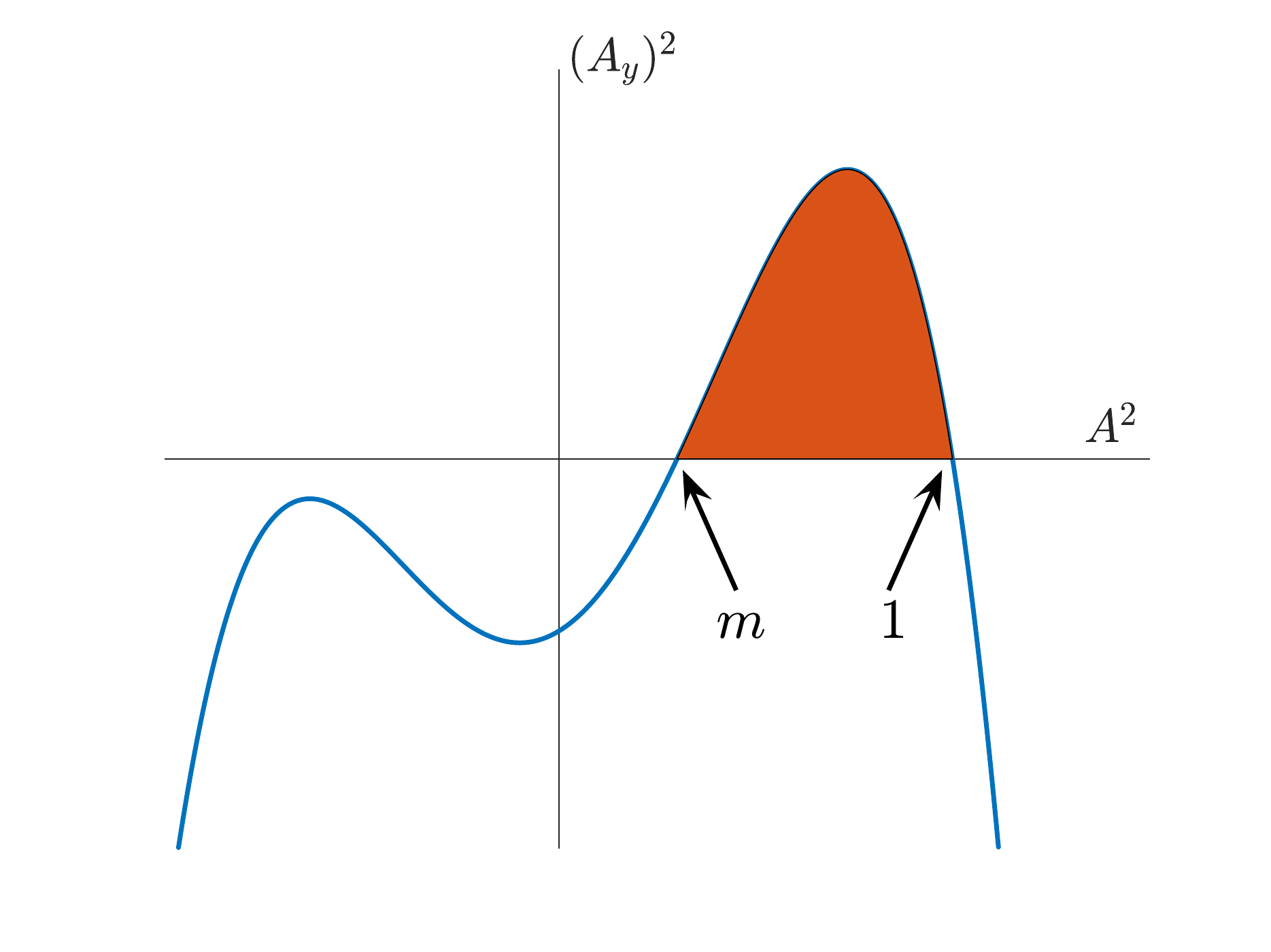}
\includegraphics[width=0.32\textwidth]{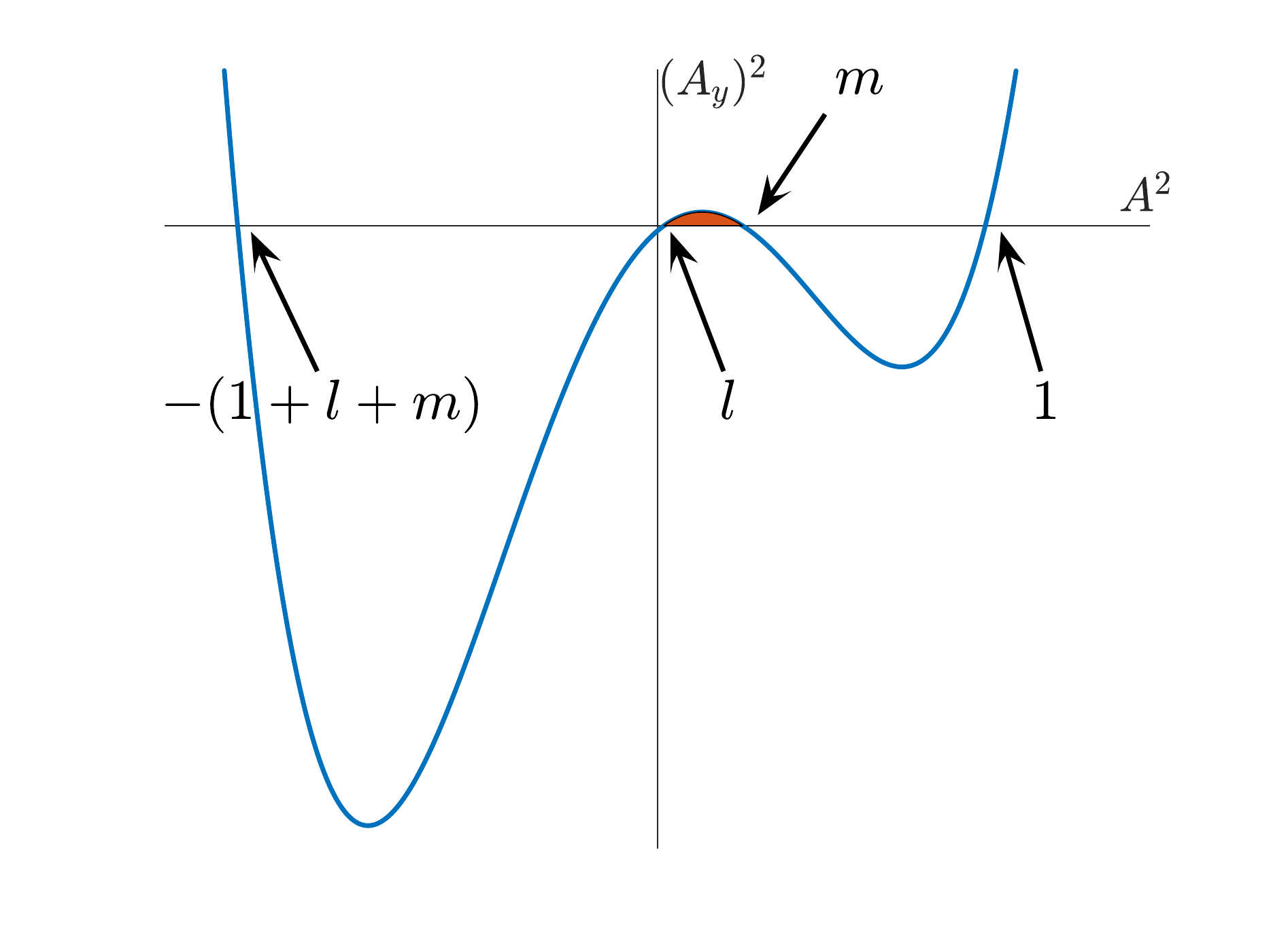}
\caption{We illustrate how we parameterize the quintic NLS by plotting the polynomial $A^2(A_y)^2=P_4(A^2)$ from Eq.~\eqref{eqn:QuintAmp}.  The left two figures are for the focusing case, on the left is a parameterization that allows for four real roots, and in the middle is one that allows for only two.  On the right is for the defocusing case.}\label{fig:quinticpolys}
\end{figure}

The boundaries of this region represent various degenerations and
special solutions, which we comment on briefly and indicate in Figure \ref{fig:quintic4}(a). First note that the
range $\frac{1+m}{2} \leq l \leq 1+m$ represents valid periodic
solutions but with the roots $-l$ and $1+m-l$ switched, which makes no
difference from the point of view of the solutions. The boundary case
$m=\frac{1+l}{2}$ is interesting, as here the quartic has a double
root. When the quartic has a root of higher multiplicity the
discriminant vanishes,
and the associated elliptic function solutions
reduce to trigonometric functions. After a tricky integration and some
slightly tedious algebra we get the following trigonometric solution in the
degenerate case $l = \frac{1+m}{2}$
\[
  z(y) = A^2(y) = \frac{1+m}{2} + \frac{1-m}{2} \frac{(1+3m) \cos^2
  \beta y  - (3+m) \sin^2  \beta y }{  (1+3m) \cos^2
  \beta y + (3+m) \sin^2  \beta y  }
\]
where
\[
\beta = \sqrt{(1+3 m)(3+m)}.
\]
This is a non-trivial phase solution, with the phase following from
$S_y = \frac{\kappa}{A^2(y)}.$ We have not been able to find this
solution in the literature, so we suspect that it is new. 
Note that this solution extends to the mixed cubic-quintic
nonlinearity in a straightforward way-- in the mixed cubic-quintic
equation the coefficient of $z^3$ is not zero but is instead
fixed. One can  parameterize the roots the same way, $1,m,-l$, with
only the fourth root differing. The elliptic function degenerates to a
trigonometric function when the third and fourth roots are
equal. (Note, however, that in moving to the mixed cubic-quintic the
scale invariance is lost.)

Along the boundary $m=1$ we have constant amplitude solutions---the
Stokes or plane waves. 
The other boundary points represent other trivial-phase special
cases: the boundaries $l=0$ and $m=0$ each represent a single
one-parameter family of trivial-phase solutions: for $l=0$ the
amplitude $A(y)$ oscillates between two positive values (similar to
the dnoidal solution in the cubic case) while for $m=0$ the amplitude
$A(y)$ oscillates between $-1$ and $1$, similar to the cnoidal
solution in the cubic case.  Finally  the
intersection $l=m=0$ represents the well-known solitary wave
solution. The first graph in Figure \ref{fig:quintic4} depicts the parameter space, with the various special solutions above marked.  
 
 \begin{figure}
\begin{tabular}{p{0.45\textwidth} p{0.45\textwidth}}
  \vspace{0pt} \includegraphics[width=0.45\textwidth]{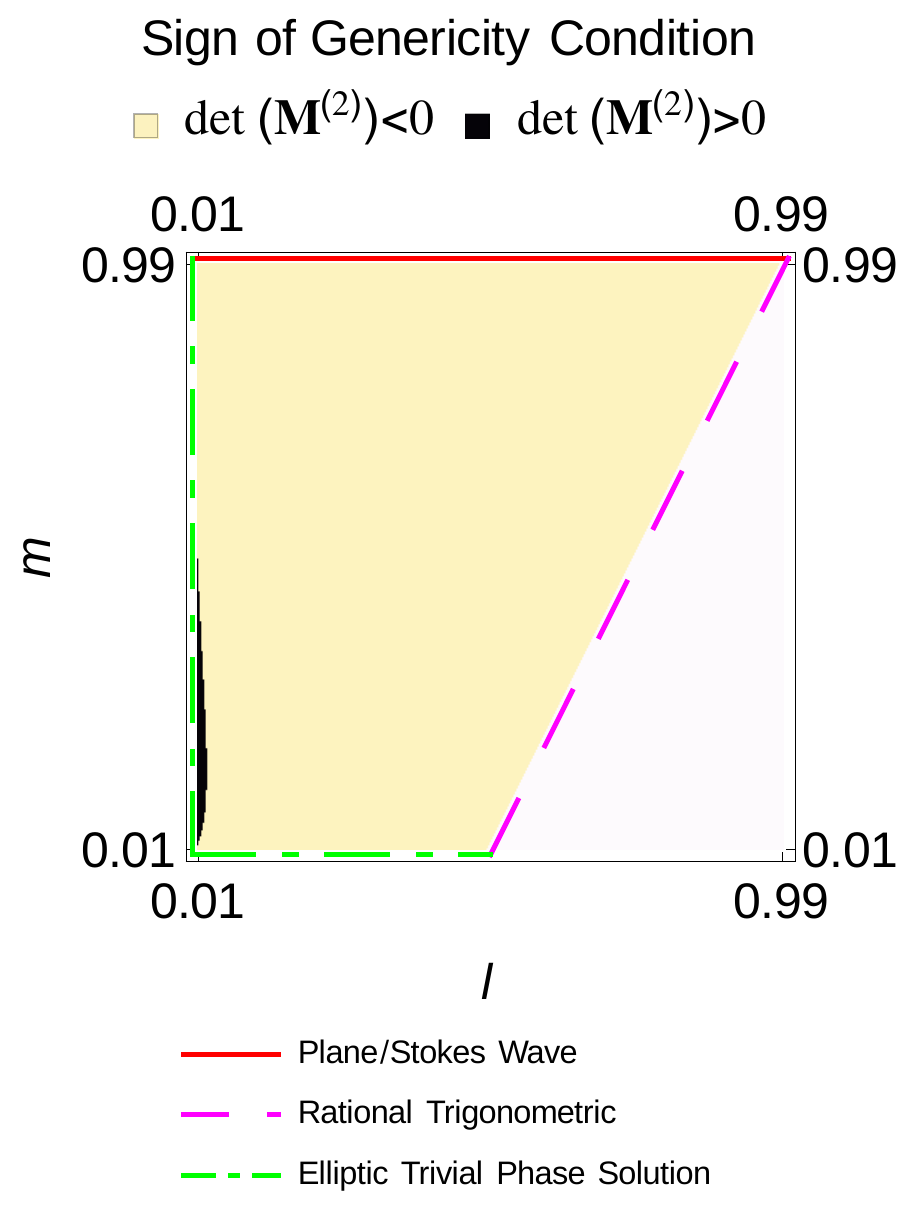} &
  \vspace{0pt} \includegraphics[width=0.45\textwidth]{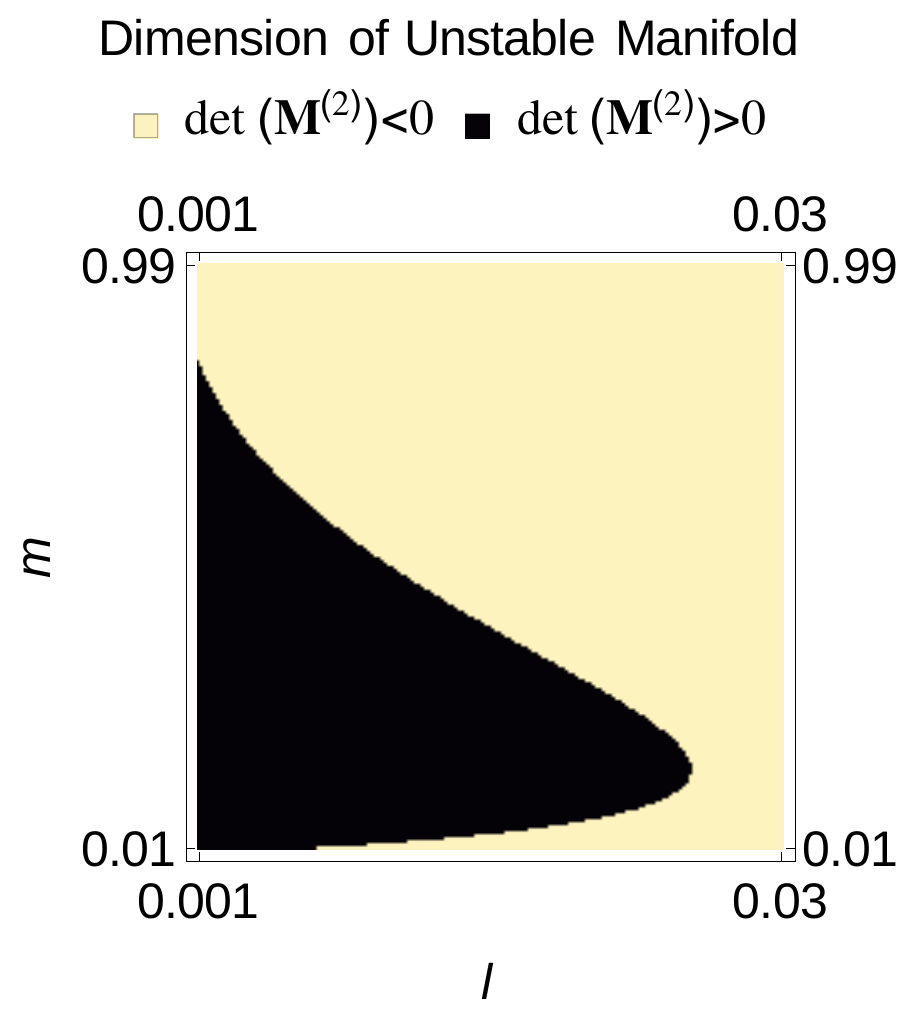}
\end{tabular}
\caption{The parameter space $m \in (0,1)$, $l \in [0,\frac{1+m}{2})$ showing the sign of $\det({\bf M}^{(2)})$. Recall that $\det({\bf M}^{(2)}) = 0$ implies failure of the second genericity condition and the existence of a higher generalized kernel. The plot on the right is rescaled to highlight the region in which  $\det({\bf M}^{(2)})>0$.  The various bordering lines in the plot on the left indicate the type of degeneration or special solution that exists on the boundaries of the parameter space.}\label{fig:quintic4}
\end{figure}
As was shown in Theorem \ref{thm:NullSpace} there are two genericity conditions. The
null-space of the linearized operator is two dimensional unless the
quantity $\sigma = T_E \eta_{\kappa} - T_{\kappa} \eta_E$ is zero, in which case
the null-space is higher dimensional. The second genericity condition
is the vanishing of the determinant
\[
\det({\bf M}^{(2)})= \left| \begin{array}{cc}  a_2& b_2 \\ b_2 & d_2 \end{array}\right|
\]
which signals the existence of a Jordan chain of length longer than
two. The fact that the solitary wave has a Jordan chain of length
four \cite{W.1985} suggests in particular that the second might
fail in the periodic case. Numerical computations show that $\sigma$
never vanishes, but Figure \ref{fig:quintic4}(a) depicts the sign of
$a_2 d_2 - b_2^2$ as a function of $(l,m)$, while  Figure \ref{fig:quintic4}(b) gives a closeup
of the region $(l,m) \in
(0,0.03)\times (0,1)$. One can see that there is a small region where $a_2 d_2 - b_2^2=0$, a curve emerging from
the solitary wave case $(l,m)=(0,0)$ along which this determinant
vanishes and the linearized operator has a Jordan chain of length at least three.

\begin{figure}
\includegraphics[width=0.31\textwidth]{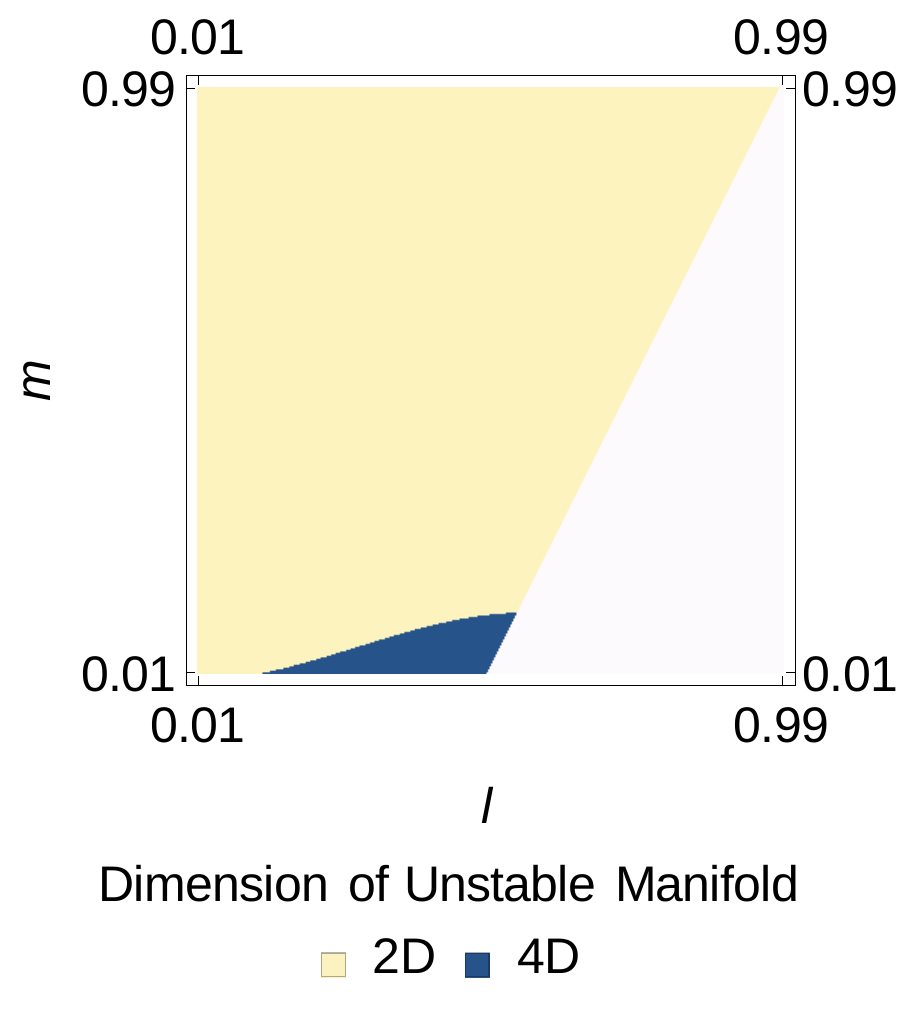}
\includegraphics[width=0.31\textwidth]{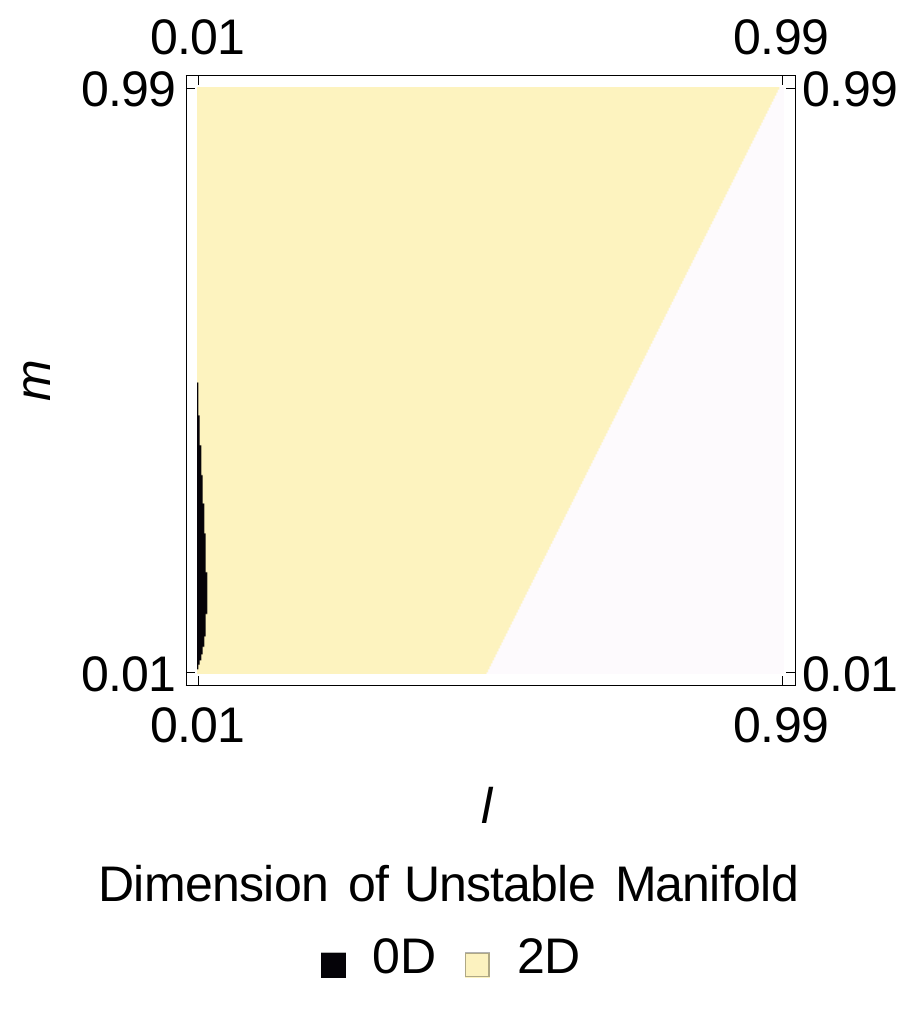}
\includegraphics[width=0.31\textwidth]{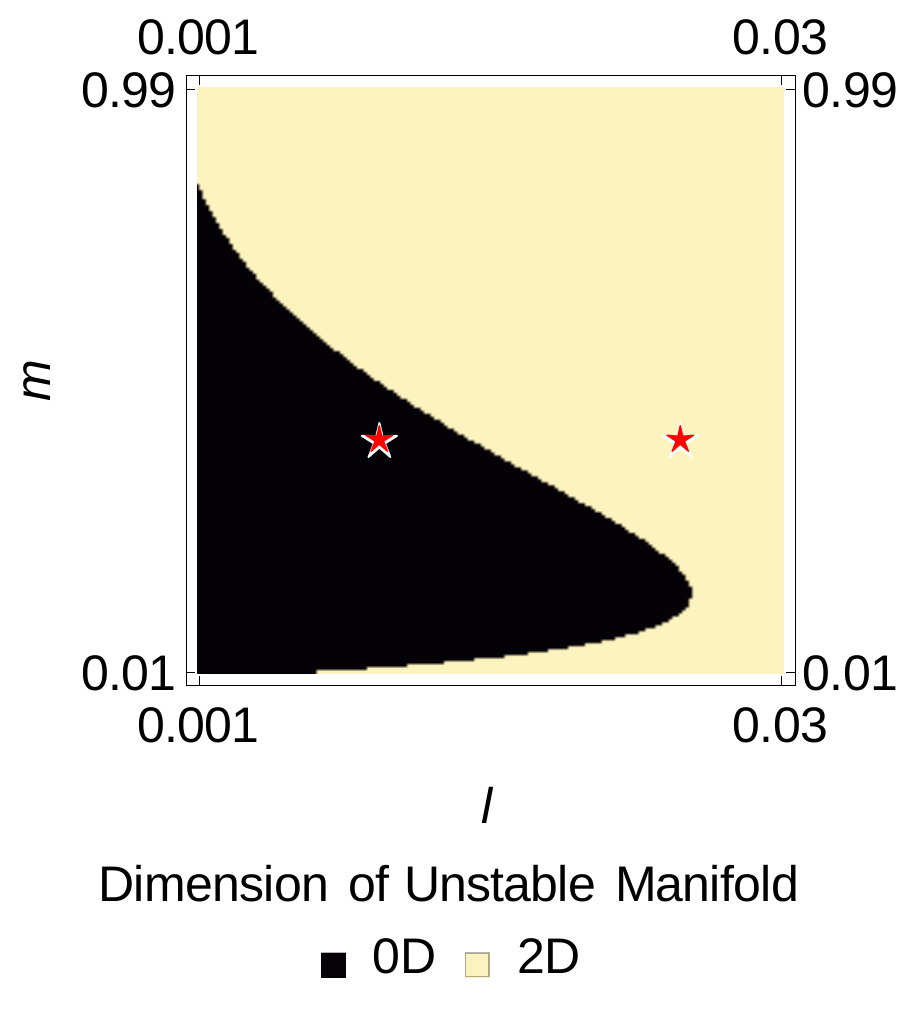}
\caption{The full parameter space colored by the possible dimensions of the unstable manifold for the focusing quintic NLS, parameterized with four real roots of $P_4(A^2)$.  The left is for longitudinal perturbations, and the middle and right are for transverse perturbations (the right is rescaled, similarly to Figure \ref{fig:quintic4}).  The red stars in the plot on the right indicate the two specific cases shown in Figure \ref{fig:quintic4spectrum}.}\label{fig:quintic4trans}
\end{figure}
\begin{figure}
\includegraphics[width=0.48\textwidth]{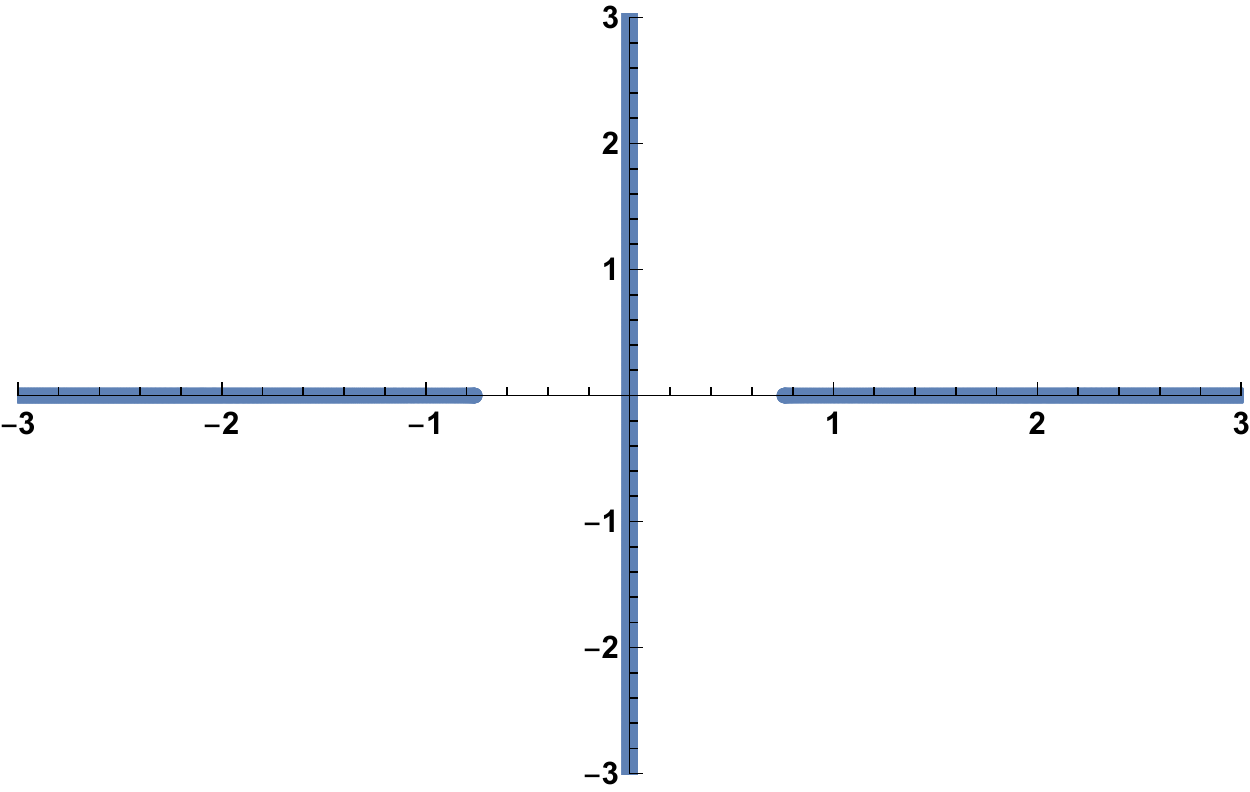}
\includegraphics[width=0.48\textwidth]{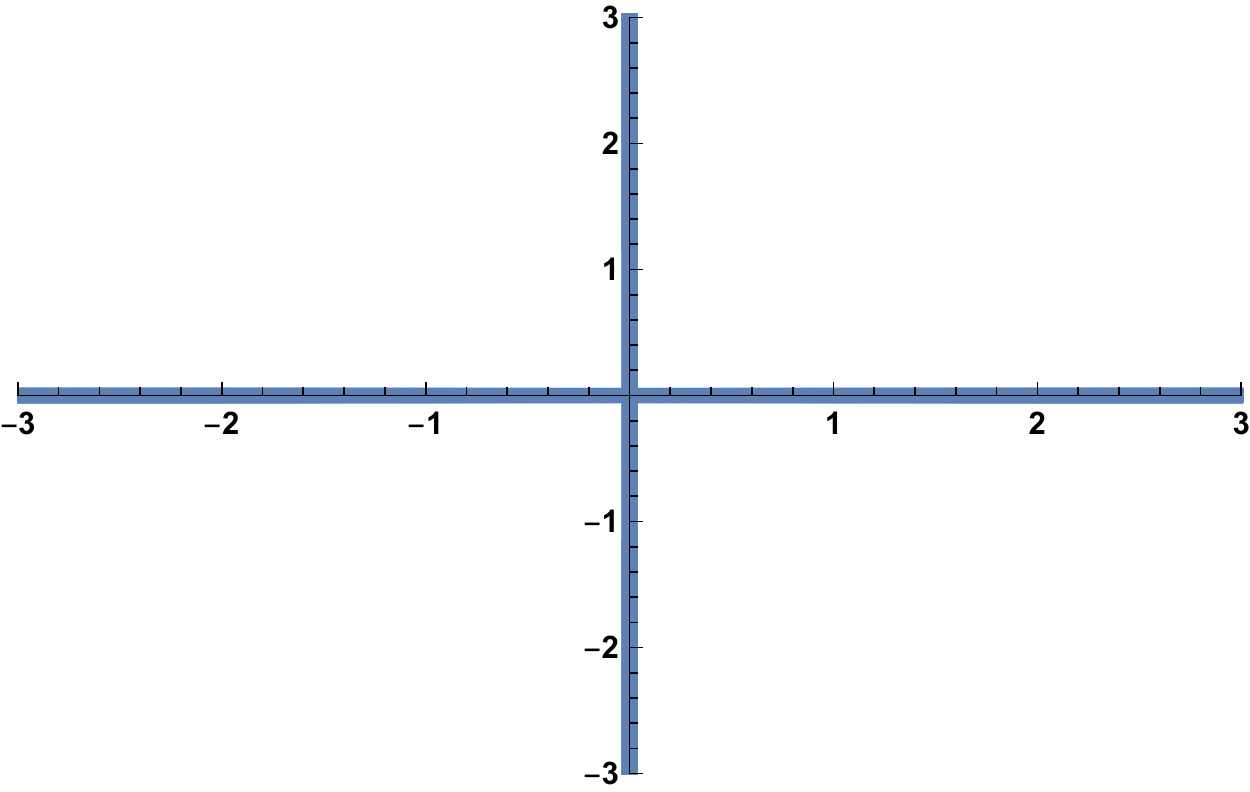}
\caption{Shown are the discretely sampled approximation to the continuous curves of the spectrum for transverse perturbations of the focusing quintic NLS, for two pairs of parameter values.  On the left, we use $(l,m) = (.01,.4)$, which falls in the region with a 0D unstable manifold, while on the right, we use $(l,m)=(.025,.4)$, which falls in the region with a 2D unstable manifold. }\label{fig:quintic4spectrum}
\end{figure}

The next figure, Figure \ref{fig:quintic4trans}, depicts the modulational stability of the periodic
traveling waves to longitudinal perturbations (in left panel) and transverse elliptic perturbations (in middle and right panel). This figure is
essentially identical to the first. In most of the parameter space, the
four dimensional null-space at $k=0$ breaks up into two real eigenvalues and two
purely imaginary eigenvalues. Inside the curve along which $\det({\bf
  M}^{(2)})=0$, however, the four dimensional null-space breaks up into
four purely imaginary eigenvalues - there is no modulational
instability. Interestingly, however, there is a finite wavelength
instability. Note that since there are no fully complex eigenvalues in this case, the transverse hyperbolic figure would look similar - except that inside the curve along which $\det({\bf M}^{(2)})=0$, the four dimensional null-space breaks up into four purely real eigenvalues, so there is modulational instability. The final figure in this series, Figure \ref{fig:quintic4spectrum}, depicts the
numerically computed spectrum for two points, one $(l,m)=(.01,.4)$ in the region where
$\det({\bf M}^{(2)})>0$ and there is no modulational instability, and one
$(l,m)=(.025,.4)$ in the region where $\det({\bf M}^{(2)})<0$ and there is a modulational
instability. We see that the numerically computed spectrum agrees with
the predictions of modulation theory. In the first case the spectrum
in the neighborhood of the origin is along the imaginary axis, with a
band of real spectrum beginning at $\lambda \approx \pm .75$. In the
second case, in contrast there are curves of continuous spectrum along
the real and imaginary axes. There is good quantitative agreement here
too between  direct numerical simulation and the modulation
theory predictions. In the first case numerical simulations and
numerical differentiation with $\mu = .01$ give
$\lambda_i(\mu) = \pm 11.13i + O(\mu^2)$ and $\lambda_i(\mu) = \pm
2.14 i + O(\mu^2)$ while the normal form calculation gives   $\lambda_i(\mu) = \pm 11.25i + O(\mu^2)$ and $\lambda_i(\mu) = \pm
2.13 i + O(\mu^2)$ In the second case direct numerical simulation
gives $\pm 1.92i + O(\mu^2)$ and $\pm 8.35 +O(\mu^2)$ and modulation
theory gives  $\pm 1.92i + O(\mu^2)$ and $\pm 8.37 +O(\mu^2)$. 

\subsubsection{Case 2: Two real roots}

In this case the depressed quartic $P_4(z)$ has two real positive
roots and two complex conjugate roots in the left half-plane. We
choose a slightly different normalization in this situation. We can
again normalize so that the largest real root is equal to 1, and denote the other real root by $m$.  Then we will take the complex conjugate pair of roots to be $-\frac{(1+m)}{2} \pm i
\tan \phi$, with the angle $\phi \in (0,\frac{\pi}{2})$, as shown in Figure \ref{fig:quinticpolys}(b). In this case the parameters can be expressed as
\begin{align*}
  & \kappa = \sqrt{m \left(\tan^2 \phi + \left(\frac{1+m}{2}\right)^2\right)}& \\
  & E = \frac{1}{2} (1+m)\left(\tan^2\phi + \left(\frac{1-m}{2}\right)^2\right) & \\
  &  \omega = -\frac{1}{4}(3 m^2 + 2 m +3 - 4 \tan^2\phi )& 
\end{align*}
The allowable parameter space is the rectangle  $(\phi,m) \in [0,\frac{\pi}{2})\times [0,1].$ When $\phi=0$ there is a real root of multiplicity two on the real axis and the
solution reduces to the trigonometric solution given in the previous
case.
 See Figure \ref{fig:quintic2}(a) for longitudinal instability: for much of the parameter domain, the unstable manifold is four dimensional, but for a sizable portion it is only two dimensional.  Again, in Figure \ref{fig:quintic2}(b), we see that for transverse instability, the unstable manifold is always two dimensional, both in the elliptic (shown) and the hyperbolic cases.
\begin{figure}
\includegraphics[width=0.48\textwidth]{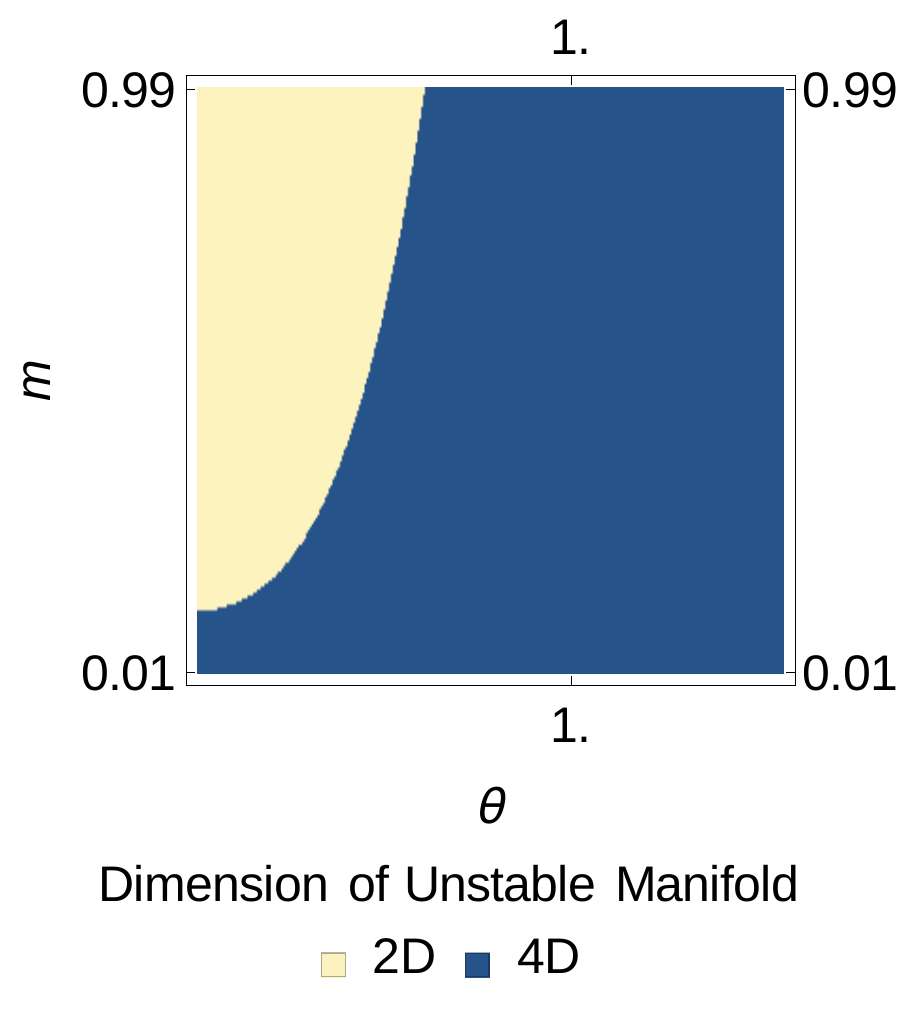}
\includegraphics[width=0.48\textwidth]{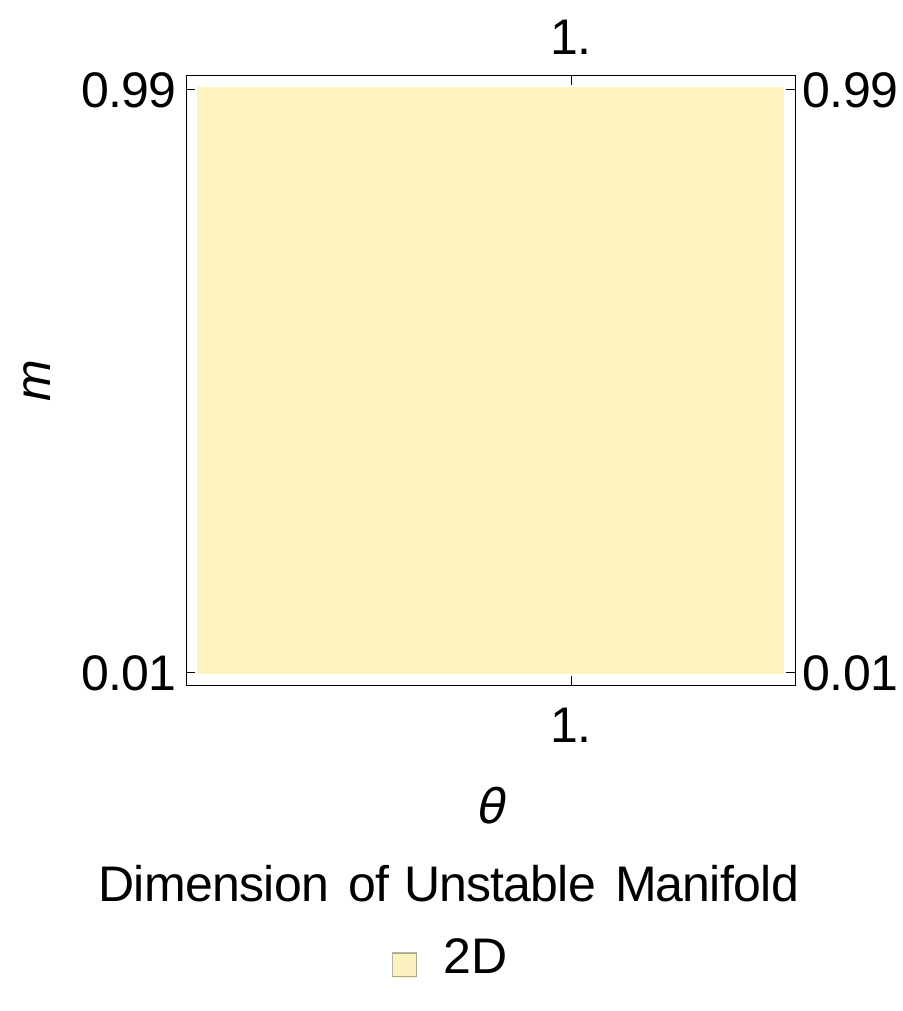}
\caption{The full parameter space colored by the possible dimensions of the unstable manifold for the focusing quintic NLS, parameterized with two real roots of $P_4(A^2)$.  The left is for longitudinal perturbations, and the right is for transverse perturbations.}\label{fig:quintic2}
\end{figure}

\subsection{Defocusing Case}

In the defocusing case we take $\zeta = -3$. In this case in order to
have a periodic solution we need $P_4(z)$ to have a compact interval
of the positive real axis on which $P_4(z)>0$. Therefore $P_4(z)$ must have
three positive real roots and (from the fact that $P_4(z)$ has no
cubic term) one negative root. We can after rescaling take the three positive roots to
be $1, l, m$ in which case the fourth root is $-(1+l+m)$, as shown in Figure \ref{fig:quinticpolys}(c). The parameter space in this case is the triangle $0\leq m \leq l \leq 1.$
In this case the parameters can be expressed as
\begin{align*}
  & \kappa = \sqrt{l m (1+l+m) }& \\
  & E = \frac{1}{2} (1+l)(1+m)(l+m) & \\
  & \omega = 1 +l^2 +(l+m)(1+m) & 
\end{align*}
Given the ordering the trivial-phase solutions occur when $l=0$. There
are two degenerations, when $m=1$ and when $m=l$, where the quartic
has roots of higher multiplicity. The case $m=l$ admits a constant
amplitude solution, which is not particularly interesting. The other
case, when $m=1$ and $1$ is a double root of the quartic, represents a
homoclinic type solution or dark solitary wave. 

The formula for $z(y)$ in the case $m=1$ is given by
\[
  z(y) = A^2(y) = \frac{l(l+3) + (l+2)(1-l) \tanh^2 (\beta y)}{(l+3)-(1-l) \tanh^2 (\beta y)} 
\]
 
with $\beta = \sqrt{(1-l)(l+3)}.$ Again we have modded out the
scaling, translation $U(1)$ and Galilean invariances. Obviously we have $z(y)
\rightarrow 1$ as $|y| \rightarrow \infty$ and $z(0)=l$. This is a
non-trivial phase solution but when $l=0$ it reduces to the front
type solution
\[
  A(y) = \frac{\sqrt{2} \tanh(\sqrt{3} y)}{\sqrt{3 - \tanh^2 (\sqrt{3} y) }}.
\]
Each of these degenerations and special solutions is indicated in Figure \ref{fig:quinticdef}(b).  Finally, we also observe in Figure \ref{fig:quinticdef} that the defocusing quintic NLS appears longitudinally stable for the entirety of the parameter domain, with four purely imaginary eigenvalues, and transversely unstable for entirety of the parameter domain, with a two dimensional unstable manifold (in both elliptic and hyperbolic cases).
\begin{figure}
\begin{tabular}{p{0.45\textwidth} p{0.45\textwidth}}
\vspace{0pt}\includegraphics[width=0.45\textwidth]{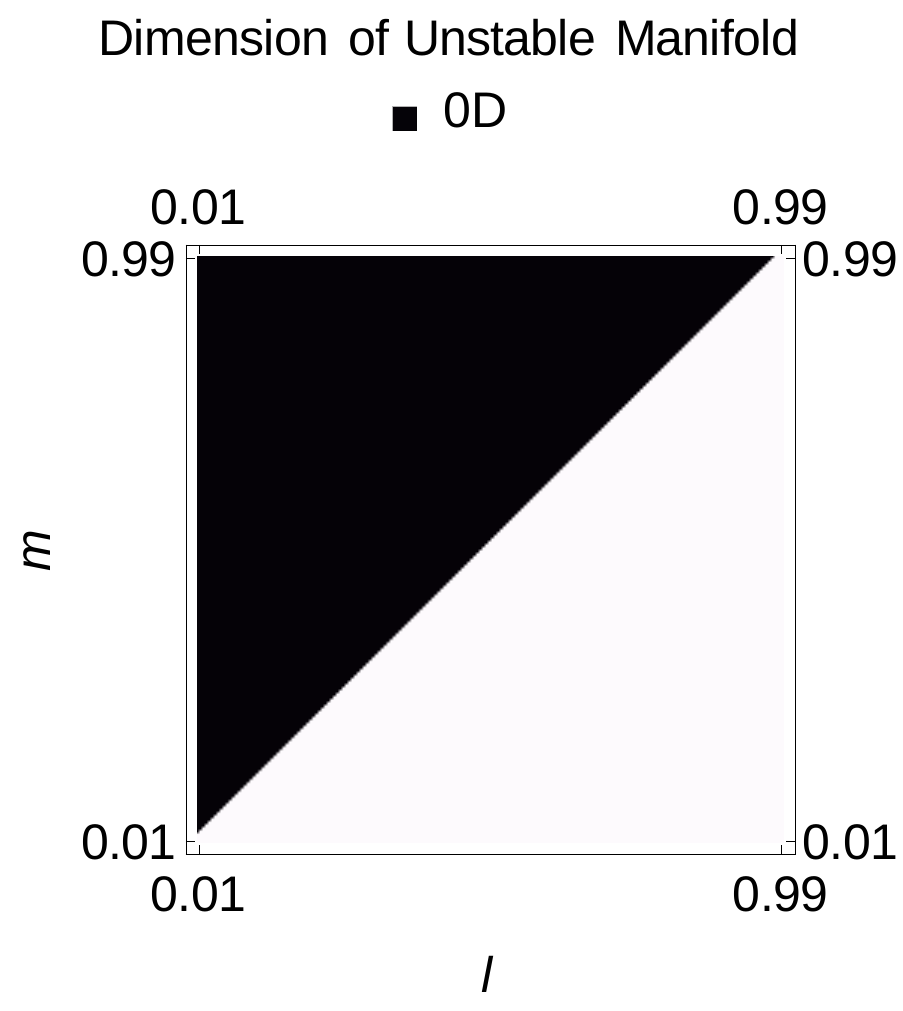} &
\vspace{0pt}\includegraphics[width=0.45\textwidth]{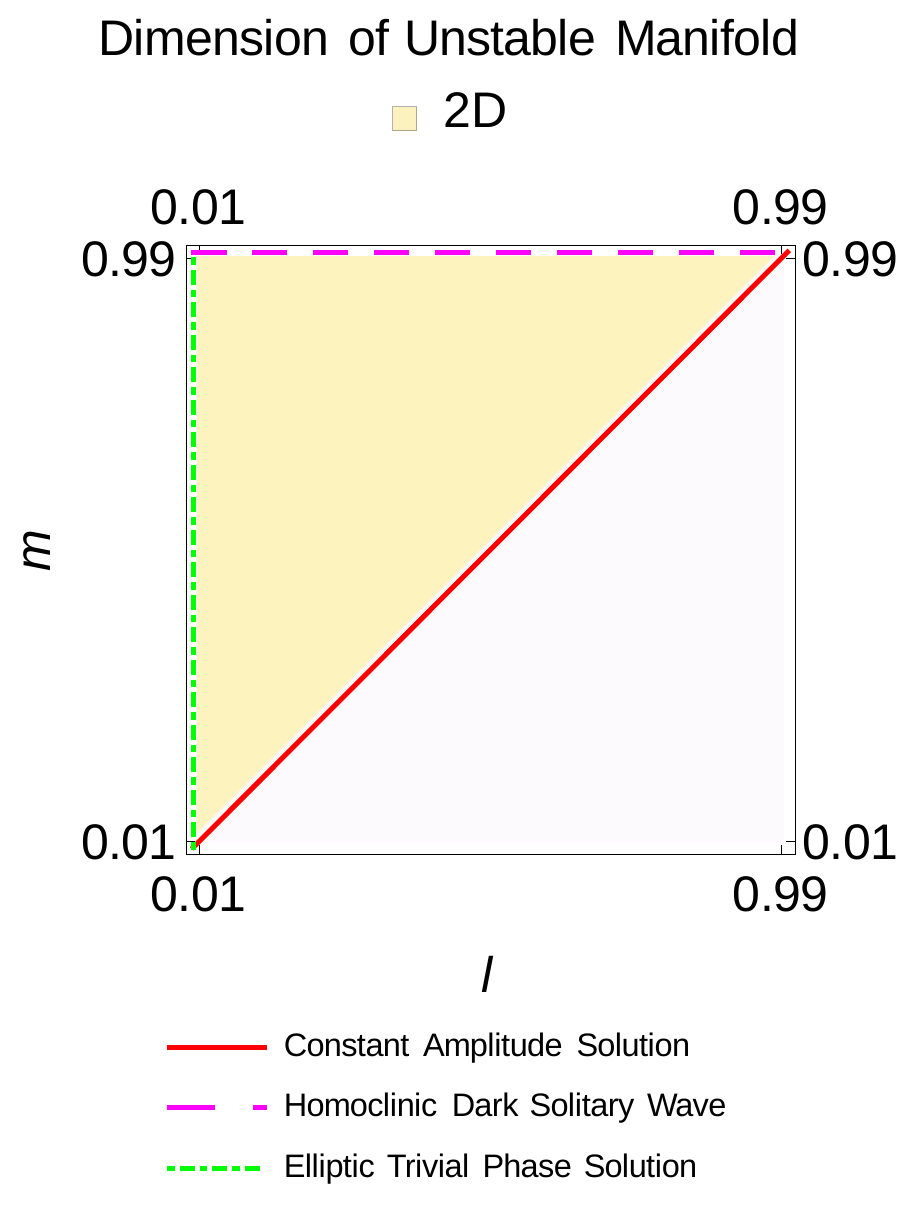}
\end{tabular}
\caption{The full parameter space colored by the possible dimensions of the unstable manifold for the defocusing quintic NLS.  The left is for longitudinal perturbations, and the right is for transverse perturbations.  The various bordering lines in the plot on the right indicate the type of degeneration or special solution that exists on the boundaries of the parameter space.}\label{fig:quinticdef}
\end{figure}

\section{Conclusion}
In this paper we have presented a rigorous derivation of the normal form for the branches of the continuous spectrum emerging from the origin in the spectral plane, where the operator in question arises from the linearization of the nonlinear Schr\"odinger equation around a traveling wave solution. In the case of the nonlinear Schr\"odinger equation the spectrum locally consists of four curves that are locally straight lines through the origin---if the lines are not purely vertical, corresponding to an imaginary eigenvalue, then the traveling wave is necessarily (spectrally) unstable.
General considerations suggest that for a Hamiltonian flow with a non-degenerate symplectic form and $k$ conserved quantities then generally there will be $2k$ curves emerging from the origin. In the cubic case these results agree (in the case of longitudinal perturbations) with those derived by Deconinck and Segal for the focusing
cubic NLS using the machinery of the inverse scattering transform. For the non-cubic case, and for the question of the (in)stability of periodic traveling waves to transverse perturbations in the cubic case, the results are new.

There are a couple of points in this calculation that would be interesting to clarify
further. One such point is the linear nature of the spectral curves in a neighborhood of the origin. The Hamiltonian structure implies that the linearized operator will generically have a non-trivial Jordan block structure. It is well-known that under a generic perturbation a Jordan block will break non-analytically---typically the eigenvalues, etc.~will be represented by a Puiseaux series in fractional powers of the perturbation parameter $\mu$. In all of the examples of which we are aware, the form of the perturbations guarantees that the bifurcation is analytic---in powers of $\mu$. It would be interesting to see this emerge from the Hamiltonian structure, rather than from a direct calculation using a basis for the null-space. It would also be interesting to understand the modulational instability for a system in which there are conserved quantities which {\em do not} all Poisson commute. One example is the NLS system of Manakov type: 
\begin{align*}
& i w^{(1)}_t = w_{xx}^{(1)} + \zeta f(|w^{(1)}|^2 + |w^{(2)}|^2) w^{(1)} & \\
& i w^{(2)}_t = w_{xx}^{(2)} + \zeta f(|w^{(2)}|^2 + |w^{(1)}|^2) w^{(2)}. &
\end{align*}

\medskip

{\bf Acknowledgements}: J.C.B. would like to acknowledge support from the National Science Foundation under grant DMS 16-15418. M.A.J. would like to acknowledge support from the National Science Foundation under grant DMS-1614785. 

\appendix

\section{Constructing the periodic eigenfunctions}\label{ap:eigenfunctions}

In this Appendix, we show how one can find the two periodic null vectors and two periodic generalized null vectors for $\mathcal L$.  Given the four solutions to $\mathcal L {\bf u}=0$ from Eq.~\eqref{eqn:ODESOLs} and the two additional quantities from Eq.~\eqref{eqn:ODESOLSgeneral}, we can write linear combinations
\begin{subequations}
\begin{align}
{\bf u}_0=&\sigma\left(\begin{array}{c}0\\ A\end{array}\right)\\
{\bf u}_1=&\gamma\left(\begin{array}{c}A_E\\AS_E\end{array}\right)+\rho\left(\begin{array}{c}A_\kappa\\AS_\kappa\end{array}\right)+\sigma\left(\begin{array}{c}A_\w\\AS_\w\end{array}\right)\\
{\bf u}_2=&\sigma\left(\begin{array}{c}A_y\\AS_y\end{array}\right)\\
{\bf u}_3=&\tau\left(\begin{array}{c}A_E\\AS_E\end{array}\right)+\nu\left(\begin{array}{c}A_\kappa\\AS_\kappa\end{array}\right)+\sigma\left(\begin{array}{c}0\\yA/2\end{array}\right).
\end{align}
\end{subequations}
Then $\mathcal L{\bf u}_{0,2}=0$ and $\mathcal L{\bf u}_{1,3}={\bf u}_{0,2}$.  The null vectors ${\bf u}_{0,2}$ are already periodic, so we need to choose $\sigma,\ \gamma,\ \rho,\ \tau,\ \nu$ to enforce the boundary conditions ${\bf u}_j(T)={\bf u}_j(0)$ on the generalized null vectors ${\bf u}_{1,3}$.  
Given that $A(0)=A(T)$ and $S(0)=S(T)-\eta$, we have
\begin{subequations}
\begin{align}
    A_E(0) &= \left[A(0)\right]_E=\left[A(T)\right]_E=A_y(T)T_E+A_E(T)\\
\nonumber    S_E(0) &= \left[S(0)\right]_E=\left[S(T)-\eta\right]_E\\
    &=S_y(T)T_E+S_E(T)-\eta_E
\end{align}

and similarly for $A_\kappa$, $A_\omega$, $S_\kappa$, and $S_\omega$. 
\end{subequations}
Using these to equate ${\bf u}_1(0)={\bf u}_1(T)$ and ${\bf u}_3(0)={\bf u}_3(T)$ results in the following four equations:
\begin{subequations}
\begin{align}
    &\gamma T_E+\rho T_\kappa+\sigma T_\w=0, \\
    &\gamma\eta_E+\rho\eta_\kappa+\sigma\eta_\w=0, \\
    &\tau T_E+\nu T_\kappa=0, \\
    &\tau\eta_E+\nu\eta_\kappa+\sigma T/2=0,
\end{align}
\end{subequations}
which we can solve for $\sigma,\ \gamma,\ \rho,\ \tau,\ \nu$, finding
\begin{subequations}
\begin{align}
        & \gamma  = \{T,\eta\}_{\kappa,\omega}, & \\
        & \rho = \{T,\eta\}_{\omega,E}, & \\
        & \tau = T T_\kappa/2, &\\
        & \nu  = -T T_E/2, & \\
        & \sigma = \{T,\eta\}_{E,\kappa}. &
\end{align}
\end{subequations}

\section{Evaluating Matrix Elements}\label{ap:matrixelements}




In this Appendix, we show some details of how to compute the matrix elements from Eq.~\eqref{eqn:quarticmatrixelements}. 

We begin by computing the elements of ${\bf M}^{(2)}$, which are also the elements of the gram matrix:
\begin{align}
   & {\bf v}_0 {\bf u}_0  = -\sigma \int_0^T A (\gamma A_E + \rho A_{\kappa} + \sigma A_{\omega}) dy & \\
   & {\bf v}_2 {\bf u}_0  = -\sigma \int_0^T A (\tau A_E + \nu A_{\kappa}) dy & \\
   &{\bf v}_2 {\bf u}_2  = \sigma \int_0^T A A_y (\tau  S_E + \nu S_{\kappa} + \sigma   \frac{y}{2})
   - A S_y  (\tau A_E + \nu A_{\kappa} ) dy. & 
\end{align}
The first two are somewhat trivial, the third is only slightly more complicated and requires some integration by parts and the fact that $(\tau \partial_E + \nu \partial_\kappa) \eta +\sigma T/2 = 0$.  
We see that
\begin{subequations}
\begin{align}
    &{\bf v}_0 {\bf u}_0 = -\frac{\sigma}{2} (\gamma M_E + \rho M_\kappa + \sigma M_\omega)  
            = - \frac{\sigma}{2} \{\eta,T,M\}_{\kappa,E,\omega}& \\
    &{\bf v}_2 {\bf u}_0 = -\frac{\sigma}{2} (\tau M_E + \nu M_\kappa)  
            = -\frac{\sigma T}{4}\{T,M\}_{\kappa,E}\\
    & {\bf v}_2 {\bf u}_2 = -\frac{\sigma}{2} (\tau \kappa T_E + \nu \kappa T_\kappa + \nu T + \sigma M/2)  
            =-\frac{\sigma^2M}{4}+\frac{\sigma T^2T_E}{4}.
\end{align}\label{eqn:viuj}
\end{subequations}

We also note that using $(\gamma \partial_E + \rho \partial_\kappa +
    \sigma \partial_\omega) \eta = 0$, we also have 
\begin{multline}
    {\bf v}_0 {\bf u}_2 = \sigma \int_0^T A A_y (\gamma  S_E + \rho S_{\kappa} + \sigma S_{\omega})
   - A S_y  (\gamma A_E + \rho A_{\kappa} + \sigma A_{\omega} ) dy\\
           = -\frac{\sigma}{2}(\gamma \kappa T_E+\rho \kappa T_\kappa +\rho T+\sigma \kappa T_\omega) = -\frac{\sigma \rho T}{2}.
\end{multline}

It remains to compute the quantities for ${\bf M}^{(1)}$ and ${\bf M}^{(0)}$.  
For ${\bf M}^{(1)}$, because of the symmetry of $\mathcal L^{(1)}$, we need only to find the four quantities ${\bf v}_{0,2}\mathcal L^{(1)}{\bf u}_{0,2}$.  One can show that $\mathcal L^{(1)}{\bf u}_0=2i{\bf u}_2$, so that ${\bf v}_0\mathcal L^{(1)}{\bf u}_0$ and ${\bf v}_2\mathcal L^{(1)}{\bf u}_0$ follow from ${\bf v}_j{\bf u}_k$. We have 
\begin{align*}
    {\bf v}_0\mathcal L^{(1)}{\bf u}_0&=-i\sigma \rho T,\\
    {\bf v}_2\mathcal L^{(1)}{\bf u}_0&= -i\sigma(T\nu+\sigma M/2).
\end{align*}

Here we will compute ${\bf v}_0\mathcal L^{(1)}{\bf u}_2$, leaving computation of ${\bf v}_2\mathcal L^{(1)}{\bf u}_2$ as an exercise for the reader.  
The integral we seek to compute is
\begin{subequations}
\begin{multline*}
{\bf v}_0\mathcal L^{(1)}{\bf u}_2=-2i\sigma\int_0^TA(\gamma S_E+\rho S_\kappa +\sigma S_\w)(-A' S'-(A S')')\\
+(\gamma A_E+\rho A_\kappa +\sigma A_\w)(-A''+A (S')^2)dy,
\end{multline*}
which can be simplified using that $2A'S'+AS''=0$ and rearranged as
\begin{multline*}
{\bf v}_0\mathcal L^{(1)}{\bf u}_2=-2i\sigma\left[\gamma\int_0^T -A''A_E+AA_E( S')^2dy\right.\\
\left.+\rho\int_0^T -A''A_\kappa +AA_\kappa ( S')^2dy
+\sigma\int_0^T -A''A_\w+AA_\w( S')^2dy\right].
\end{multline*}
Next we integrate some terms by parts, using that $AA'=(A^2)'/2$.
This leads to
\begin{multline*}
{\bf v}_0\mathcal L^{(1)}{\bf u}_2=-2i\sigma\left[\gamma\int_0^T (A')^2_E/2+AA_E( S')^2dy\right.\\
\left.+\rho\int_0^T (A')^2_\kappa/2 +AA_\kappa ( S')^2dy
+\sigma\int_0^T (A')^2_\w/2+AA_\w( S')^2dy\right].
\end{multline*}
Now we use that $\int_0^T (A')^2dy = K$ with it's relevant derivatives to simplify:
\begin{multline*}
{\bf v}_0\mathcal L^{(1)}{\bf u}_2=-2i\sigma\left[ \gamma T/2-\rho\eta/2-\sigma M/4\right.\\
\left.+\gamma\int_0^T AA_E( S')^2dy+\rho\int_0^T AA_\kappa ( S')^2dy+\sigma\int_0^T AA_\w( S')^2dy\right].
\end{multline*}
Finally, one can show that 
$AA_E( S')^2=\left(\frac{A^2( S')^2}{2}\right)_E-A^2 S' S'_E=\left(\frac{\kappa  S'}{2}\right)_E-\kappa  S'_E=-\kappa  S'_E/2$.  The same is true with $\omega$ derivatives, and similar is true with $\kappa$ derivatives.  This leads to
\begin{multline*}
{\bf v}_0\mathcal L^{(1)}{\bf u}_2=-2i\sigma\left[\frac{\gamma T}{2}-\frac{\rho\eta}{2}-\frac{\sigma M}{4}\right.\\
\left.+\gamma\int_0^T -\kappa  S'_E/2dy+\rho\int_0^T -\kappa  S'_\kappa /2+ S'/2dy+\sigma\int_0^T -\kappa  S'_\w/2dy\right].
\end{multline*}
Now, since $S(T)-S(0)=\eta$, we can complete the integration and simplify using $\gamma\eta_E+\rho\eta_\kappa +\sigma\eta_\w=0$ to obtain
\begin{equation*}
{\bf v}_0\mathcal L^{(1)}{\bf u}_2=-2i\sigma\left[\frac{\gamma T}{2}-\frac{\sigma M}{4}\right].
\end{equation*}
\end{subequations}

In a similar way one can find
\[{\bf v}_2\mathcal L^{(1)}{\bf u}_2 = -2i\sigma \left[\frac{\tau T}{2}+\frac{\sigma \kappa T}{4}\right].\]

Now, to obtain the matrix elements of ${\bf M}^{(0)}$, we will need 8 more quantities: ${\bf v}_{1,3}\mathcal L^{(2)}{\bf u}_{0,2}$ and ${\bf v}_{1,3}\mathcal L^{(1)}\mathcal L^{-1}\mathcal L^{(1)}{\bf u}_{0,2}$.  The first four are:
\begin{align*}
    {\bf v}_1\mathcal L^{(2)}{\bf u}_0 &=\sigma^2 M,\\
    {\bf v}_1\mathcal L^{(2)}{\bf u}_2={\bf v}_3\mathcal L^{(2)}{\bf u}_0&=\sigma^2\kappa T,\\
    {\bf v}_3\mathcal L^{(2)}{\bf u}_2 &= \sigma^2(2ET-\omega M-\zeta U),
\end{align*}
where $U=\int_0^TF(A^2)dy$.  Three of the latter four follow from ${\bf v}_{0,2}\mathcal L^{(1)}{\bf u}_{0,2}$ due to the fact that $\mathcal L^{(1)}{\bf u}_0=2i{\bf u}_2$ and ${\bf v}_1\mathcal L^{(1)}=2i{\bf v}_3$.  Thus 
\begin{align*}
    {\bf v}_1\mathcal L^{(1)}\mathcal L^{-1}\mathcal L^{(1)}{\bf u}_0&=2i{\bf v}_2\mathcal L^{(1)}{\bf u}_0\\
    &=2\sigma\left(\nu T+\frac{\sigma M}{2}\right)
\end{align*}
and (also using symmetry)
\begin{equation*}
    {\bf v}_1\mathcal L^{(1)}\mathcal L^{-1}\mathcal L^{(1)}{\bf u}_2={\bf v}_3\mathcal L^{(1)}\mathcal L^{-1}\mathcal L^{(1)}{\bf u}_0= 2\sigma\left(\tau T+\frac{\sigma \kappa  T}{2}\right).
\end{equation*}
Finally, it remains to compute ${\bf v}_{3}\mathcal L^{(1)}\mathcal L^{-1}\mathcal L^{(1)}{\bf u}_{2}$.  First, we note that we can write 
\begin{align*} 
\mathcal L^{(1)}{\bf u}_2 
 &= -i\sigma\left(\begin{array}{c}0\\ 2A_{yy}-2A(S_y)^2\end{array}\right)\\
 &= 2i\omega {\bf u}_0+2i\zeta {\bf u}_4,
\end{align*}
using Eqs.~\eqref{eqn:TravWave1}, \eqref{eqn:TravWave2}, \eqref{eqn:TravWave3}. 
Then we can write 
\[ {\bf v}_3\mathcal L^{(1)}\mathcal L^{-1}\mathcal L^{(1)}{\bf u}_2 = 2i\omega {\bf v}_3\mathcal L^{(1)}{\bf u}_1+2i\zeta {\bf v}_3\mathcal L^{(1)}{\bf u}_5.
\]
We already know the first two terms, and the third can be integrated in a similar way as
\[
{\bf v}_3\mathcal L^{(1)}{\bf u}_5 = -2i\sigma\left[\frac{\xi T}{2}-\frac{\sigma U}{4}\right].
\]
Finally, we can combine and simplify to obtain
\begin{equation*}
    {\bf v}_{3}\mathcal L^{(1)}\mathcal L^{-1}\mathcal L^{(1)}{\bf u}_{2}
    =2\sigma\left(\frac{2\w\gamma T}{2}-\frac{2\w\sigma M}{4}
    +\zeta\xi T-\frac{\sigma\zeta U}{2}\right).
\end{equation*}
Putting all these quantities together, we obtain expressions for all the matrix elements.

\bibliographystyle{plain}
\bibliography{Stability}

\end{document}